\documentclass[11pt]{article}
\usepackage{graphicx}
\usepackage{a4wide}
\usepackage[title,titletoc,toc]{appendix}
\usepackage[T1]{fontenc}
\usepackage[applemac]{inputenc}
\usepackage{amsmath}
\usepackage{amssymb}
\usepackage{amsfonts}
\usepackage[english]{babel}
\usepackage{enumerate}
\usepackage{bbm}

\usepackage{lipsum}

\usepackage{pdfsync}
\usepackage{graphicx,color}

\usepackage{caption}
\usepackage{epsfig}
\graphicspath{{./figures/}}
\usepackage{epstopdf}
\usepackage{algorithmic}
\ifpdf
  \DeclareGraphicsExtensions{.eps,.pdf,.png,.jpg}
\else
  \DeclareGraphicsExtensions{.eps}
\fi

\usepackage{hyperref}

\numberwithin{equation}{section}

\newtheorem{thm}{Theorem}[section]
\newtheorem{lem}[thm]{Lemma}
\newtheorem{prop}[thm]{Proposition}
\newtheorem{cor}[thm]{Corollary}
\newtheorem{remark}[thm]{Remark}

\newcommand{\MMM}{\color{black}}
\newcommand{\KKK}{\color{black}}
\newcommand{\eps}{\varepsilon}

\newcommand{\R}{{\mathbb R}}

\newcommand{\be}[1]{\begin{equation}\label{#1}}
\newcommand{\ee}{\end{equation}}

\newcommand{\bC}{{\mathbbm{1}}}

\newcommand{\prf}{\par\smallskip\noindent{\sl Proof. \/}}
\newcommand{\finprf}{\unskip\null\hfill$\;\square$\vskip 0.3cm}

\newenvironment{proof}{\prf}{\finprf}

\newtheorem{definition}{Definition}[section]

\numberwithin{equation}{section}

\newcommand{\argmin}{{\rm argmin\ }}

\newenvironment{proofad1}{\removelastskip\par\medskip
\noindent{\textit {Proof of {\rm  Lemma  \ref{Minimiz}}}.}
\rm}{\penalty-20\null\hfill$\square$\par\medbreak} 

\newenvironment{proofad2}{\removelastskip\par\medskip
\noindent{\textbf {Proof of Theorem  \ref{furtherregularity*}}.}
\rm}{\penalty-20\null\hfill$\square$\par\medbreak} 

\newenvironment{proofad3}{\removelastskip\par\medskip
\noindent{\textbf {Proof of Theorem  \ref{theorem:Gammaconvergence*}}.}
\rm}{\penalty-20\null\hfill$\square$\par\medbreak} 

\newenvironment{proofad4}{\removelastskip\par\medskip
\noindent{\textbf {Proof of Theorem  \ref{existencethm}}.}
\rm}{\penalty-20\null\hfill$\square$\par\medbreak} 

\newenvironment{proofad5}{\removelastskip\par\medskip
\noindent{\textbf {Proof of Theorem  \ref{stozero}}.}
\rm}{\penalty-20\null\hfill$\square$\par\medbreak} 

\newenvironment{proofad6}{\removelastskip\par\medskip
\noindent{\textbf {Proof of Lemma  \ref{RieszW11}}.}
\rm}{\penalty-20\null\hfill$\square$\par\medbreak} 

\newenvironment{proofad7}{\removelastskip\par\medskip
\noindent{\textbf {Proof of Lemma  \ref{lastbutone}}.}
\rm}{\penalty-20\null\hfill$\square$\par\medbreak} 

\newenvironment{proofad8}{\removelastskip\par\medskip
\noindent{\textbf {Proof of Lemma  \ref{lastlemma}}.}
\rm}{\penalty-20\null\hfill$\square$\par\medbreak} 

\def\qed{\,\unskip\kern 6pt \penalty 500
\raise -2pt\hbox{\vrule \vbox to8pt{\hrule width 6pt
\vfill\hrule}\vrule}\par}
\definecolor{darkblue}{rgb}{0.05, .05, .65}
\definecolor{darkgreen}{rgb}{0.1, .65, .1}
\definecolor{darkred}{rgb}{0.8,0,0}
\begin{document}
\title{Nonlinear aggregation-diffusion equations with Riesz potentials
}
\author{Yanghong Huang, Edoardo Mainini, Juan Luis V\'azquez, Bruno Volzone}
\date{ }
\newcommand{\Addresses}{{
  \bigskip
  \footnotesize
\noindent Yanghong Huang: Department of Mathematics, University of Manchester, Oxford Road, Manchester M13 9PL, United Kingdom.\\
E-mail: {\tt yanghong.huang@manchester.ac.uk} \\\\
\noindent Edoardo Mainini: DIME,
  Universit\`a  degli studi di Genova, Via all'Opera Pia, 15 - 16145 Genova, Italy. \\
E-mail: {\tt mainini@dime.unige.it}\\\\
\noindent Juan Luis V\'azquez: Departamento de Matem\'aticas, Universidad Aut\'onoma de Madrid. 28049  Madrid, Spain.\\
E-mail: {\tt juanluis.vazquez@uam.es}\\\\
\noindent Bruno Volzone: Dipartimento di Scienze e Tecnologie, Universit\`a degli Studi di
Napoli ``Parthenope'', 80143 Napoli, Italy. \\
E-mail: {\tt bruno.volzone@uniparthenope.it}\\\\
}}

\maketitle

\begin{abstract}
We consider an aggregation-diffusion model, where the diffusion is nonlinear of porous medium type  and the aggregation is governed by the Riesz potential of order $s$. The addition of a quadratic diffusion term produces a more precise competition with the aggregation term for small $s$, as they have the same scaling if $s=0$.    
We prove existence and uniqueness of stationary states and  we characterize their asymptotic behavior  as $s$ goes to zero. 
Moreover, we prove existence of gradient flow solutions to the evolution problem by applying  the JKO scheme.
\end{abstract}

\section{Introduction}
We consider the  Cauchy problem in the whole space $\R^{d}$, $d\geq1$, for the aggregation-diffusion equation
\begin{equation}\label{cauchy1}\left\{\begin{array}[c]{lll}
\partial_t\rho=\Delta\rho^m+\beta \Delta \rho^{2}-\chi\nabla\cdot(\rho\nabla(K_s\ast\rho)),\\
&\\
\rho(0)=\rho^0,
\end{array}\right.\end{equation}
where the initial datum $\rho^0$ is a nonnegative mass density in $L^1(\R^d)\cap L^m(\R^d)$, and the parameters satisfy $\chi>0$, $\beta\ge 0$, $m>2$.
Here,  $K_{s}$ denotes the Riesz kernel of order $s\in(0,d/2)$, namely $K_{s}(x)=c_{d,s}|x|^{2s-d}$, being $c_{d,s}$ an explicit normalization constant defined \-as
\begin{equation}\label{cds}
c_{d,s}:=\pi^{-d/2}2^{-2s}\Gamma(\tfrac d2-s)/\Gamma(s)\qquad \mbox{(as $s\downarrow 0$ there holds $c_{d,s}\sim \pi^{-d/2}\Gamma(\tfrac d2)\,s$)}.
\end{equation}
The natural free energy associated
with the nonlocal PDE  \eqref{cauchy1} is given by
\begin{align}\label{functional}
\mathcal{F}_s[\rho]&=\frac{1}{m-1}\int_{\mathbb{R}^d}\rho^m(x)\,dx+\beta\int_{\mathbb{R}^d}\rho^2(x)\,dx-\frac\chi2\int_{\mathbb{R}^d}\int_{\mathbb{R}^d}K_{s}(x-y)\rho(x)\rho(y)\,dx\,dy.
\end{align}
We notice that \eqref{cauchy1} has the structure of a continuity equation
$
  \partial_t\rho + \nabla\cdot(\rho \mathbf{u})=0,
$
where the velocity vector field is a gradient $\mathbf{u} = -\nabla\psi$ and the velocity potential
 \begin{equation}\label{u*}
\psi:=\frac{m}{m-1}\rho^{m-1}+2\beta\rho-\chi \,K_{s}\ast\rho
\end{equation}
is the functional derivative of $\mathcal{F}_s[\rho]$ with respect to $\rho$. For this reason, the evolution equation \ref{cauchy1} is formally the gradient flow of functional $\mathcal F_s$ with respect to the Wasserstein distance. \\

Our first objective is the analysis of stationary states of the dynamics, with most emphasis on their behavior as $s$ becomes small.
In fact, in our first result we show that for any given mass $M>0$, $\mathcal F_s$
has a unique minimizer over
 \begin{equation*}\label{YM}
   \mathcal{Y}_M:=\left\{\rho\in L^1_+(\mathbb{R}^d)\cap L^m(\mathbb{R}^d): \int_{\mathbb{R}^d}\rho(x)\,dx=M,\;\int_{\mathbb{R}^d}x\rho(x)\,dx=0\right\},
\end{equation*}
coinciding with the unique radial stationary state of  the dynamics with mass $M$ and center of mass at the origin.
Properties of stationary states have been thoroughly investigated for $\beta=0$ and for different ranges of  $m,s$, which are usually classified as follows:
by considering the homogeneity property of the terms of functional $\mathcal F_s$, diffusion and aggregation are in balance if $m$ is equal to the critical exponent $m_c:=2-2s/d$, which  is  the so called fair competition regime that is analyzed in \cite{CCH, CCH2}. The diffusion dominated regime $m>m_c$ was investigated in \cite{CHMV}.
 Uniqueness of stationary states have also been proved in the different regimes \cite{CCH3, CGHMV, DYY}.
Since the diffusion exponents in \eqref{cauchy1} are greater than $2$,  here we are considering a diffusion dominated model. However, the competition between the additional quadratic diffusion term and the aggregation term becomes crucial in the small $s$ regime, since the limiting critical exponent is exactly $2$.

Let us  introduce the precise notion of stationary state. We will check in Section \ref{stationarysection} that  the assumptions on $\rho$ in the next definition entail $\rho\nabla\psi\in L^1_{loc}(\R^d)$, where $\psi$ is given by  \eqref{u*}. 
 \begin{definition}\label{steadydef*}
 Let $\rho\in W^{1,1}(\R^{d})\cap L^{\infty}(\R^{d})\cap C(\R^{d})$ be a nonnegative density. Let $\Omega:=\left\{ x\in\mathbb R^d:\,\rho(x)>0\right\}$ and let
 $\psi$ be defined by \eqref{u*}. We say that $\rho$ is a stationary state for the evolution equation in \eqref{cauchy1} if
$\psi\in W^{1,\infty}(\Omega)$ and $\nabla\cdot(\rho \nabla\psi)=0$ in  $\mathcal D'(\R^d)$.
 \end{definition}

We stress that this definition differs from the one appearing in \cite{CCH} and in  later works. The definition that we propose is better suited to treat the small $s$ regime. Indeed, as we explain in Section \ref{stationarysection}, minimizers of $\mathcal F_s$ always satisfy the new definition.
We have the following
\begin{thm}\label{furtherregularity*}
Let $\beta\ge 0$. If $\beta=0$ and $s<1/2$ assume in addition that $m<\tfrac{2-2s}{1-2s}$.  For any mass $M>0$, there exists a unique stationary state of mass $M$ and center of mass $0$. Such steady state is radially decreasing, compactly supported,  H\"older on $\mathbb R^d$ and smooth inside its support. It coincides with the unique minimizer  of the energy functional $\mathcal{F}_{s}$ in the class $\mathcal{Y}_{M}$.
\end{thm}

If $\beta=0$ and $s<1/2$, without the additional restriction $m<\tfrac{2-2s}{1-2s}$ we are not able to apply the radiality result from \cite{CHVY} and \cite{CHMV}, and for this reason we do not get the same conclusion, but we shall still obtain uniqueness in the class of radial stationary states.\\

We are mostly interested in the limiting behavior of stationary states as $s\to 0$. In this perspective, since $K_s\to\delta_0$, the limit functional is formally given by
\begin{equation*}
\mathcal{F}_0[\rho]:=\frac{1}{m-1}\int_{\mathbb{R}^d}\rho^m(x)\,dx+\left(\beta-\frac\chi2\right)\int_{\mathbb{R}^d}\rho^2(x)\,dx\label{limitfunct*}.
\end{equation*}
It is clear that the minimization problem $\min_{\mathcal Y_M}\mathcal F_0$ is strongly influenced by the sign of the coefficient $\beta-\chi/2$. Indeed, it has solutions if and only if $\beta< \chi/2$, and in such case we will check that there is a unique radially decreasing minimizer, given by the characteristic function of a ball. Our second main result is the following. It will be proven in section \eqref{asymptoticsection}, where some illustration of stationary states from numerical simulations will also be provided.

\begin{thm}\label{theorem:Gammaconvergence*}
 For any $s\in (0,1/2)$, let $\rho_s\in\mathcal{Y}_M$ be the unique minimizer of $\mathcal{F}_s$ over $\mathcal Y_M$.
If $0\le\beta<\chi/2$,
 there exists $\rho\in\mathcal{Y}_M$ such that $\rho_s\to \rho$ strongly in $L^m(\mathbb{R}^d)$ as $s\downarrow 0$, and moreover
$\rho$ is the unique radially decreasing minimizer
of the functional \eqref{limitfunct} over $\mathcal{Y}_M$. Else if $\beta\ge\chi/2$, we have $\lim_{s\downarrow 0}\mathcal F_s[\rho_s]=0$ and $\rho_s\to 0$ uniformly on $\mathbb R^d$.
 \end{thm}

We next focus on the gradient flow structure of evolution problem \eqref{cauchy1}. In this case, the initial datum $\rho^0$ is supposed to belong to $\mathcal Y_{M,2}$, being $\mathcal Y_{M,2}$ the set of all densities in $\mathcal Y_{M}$ with finite second moment:
\[
\mathcal Y_{M,2}:=\left\{\rho\in L^1_+(\mathbb{R}^d)\cap L^m(\mathbb{R}^d): \int_{\mathbb{R}^d}\rho(x)\,dx=M,\;\int_{\mathbb{R}^d}x\rho(x)\,dx=0,\;
\int_{\mathbb{R}^d}|x|^{2}\rho(x)\,dx<\infty\right\}.
\]
The analysis of Keller-Segel models with Newtonian potential as Wasserstein gradient flows is found in \cite{BCC, B2, BL, B6}. More generally, there are many studies about gradient flow approach for interaction-driven evolutions. The case of the Riesz potential appears in \cite{LMS}, in the analysis of the porous medium equation with fractional pressure introduced in \cite{CV}.
 Problem \eqref{cauchy1} is formally preserving mass, positivity and center of mass, and a solution is naturally seen as a trajectory in the space \MMM $\mathcal Y_{M}$. \KKK
  A narrowly continuous curve $[0,+\infty)\ni t\mapsto \rho(t,\cdot)\in\MMM\mathcal Y_{M}\KKK$ (i.e., $t\mapsto \int_{\R^d}\varphi(x)\rho(t,x)\,dx$ is  continuous  for every continuous bounded function $\varphi$ on $\R^d$),
is a  weak solution to \eqref{cauchy1} if $\rho(0)=\rho^0$ and for every $\varphi\in C^{\infty}_c(\mathbb R^d)$ and every $\eta\in C^\infty_c((0,+\infty))$
\begin{equation}\label{veryweak}\begin{aligned}
&-\int_0^{+\infty}\int_{\mathbb R^d}\rho(t,x)\varphi(x)\eta'(t)\,dx\,dt=\int_{0}^{+\infty}\int_{\mathbb R^d}\eta(t)\Delta\varphi(x)\,\big(\rho(t,x)^m+\beta\rho(t,x)^2\big)\,dx\,dt
\\&\quad-\frac{(d-2s)\,c_{d,s}\,\chi}{2}\int_0^{+\infty}\int_{\mathbb R^d}\int_{\mathbb R^d}\eta(t)\frac{\big(\nabla\varphi(x)-\nabla\varphi(y)\big)
\cdot(x-y)}{|x-y|^{d+2-2s}}\,\rho(t,x)\rho(t,y)\,dx\,dy\,dt.
\end{aligned}
\end{equation}
 This notion of solution was introduced in \cite{SS} for the case of the interaction with the logarithmic
potential, see also \cite{BCC}.
We shall construct  weak solutions to problem \eqref{cauchy1} by applying the Jordan-Kinderlehrer-Otto \cite{JKO} scheme. Therefore, denoting by $W_2$ the Wasserstein distance of order $2$, for a discrete time step  $\tau>0$, we shall solve the recursive minimization problems
\begin{equation*}\label{MM*}
\rho_\tau^0=\rho^0,\qquad\quad
\rho_{\tau}^k\in \displaystyle\mathrm{arg}\!\!\!\min_{\rho\in \mathcal Y_M}\left(\mathcal F_s[\rho]+\frac1{2\tau}\,W_2^2(\rho,\rho_\tau^{k-1})\right),\;\;\; k\in\mathbb N,
\end{equation*}
and we shall prove that piecewise constant in time interpolations of minimizers do converge to a weak solution to \eqref{cauchy1} as $\tau\to 0$ along a suitable vanishing sequence $(\tau_n)_{n\in\mathbb N}$. A weak solution that is constructed in this way, that is, as a limit of the JKO scheme applied to $\mathcal F_s$, will be called a gradient flow solution. We have the following existence result

\begin{thm}\label{existencethm} Let $\beta\ge 0$ and $0<s<\min\{1,d/2\}$. If $\beta=0$, assume in addition that $d\ge 2$ and $1/2\le s<1$. Let $\rho^0\in\mathcal Y_{M,2}$. Then there exists a gradient flow solution to problem \eqref{cauchy1}.
\end{thm}

Solutions are global-in-time, as expected in a diffusion dominated model.
Further properties of solutions will be described in Section \ref{evolutionsection}, along with some numerical analysis of the problem in Section \ref{qualitative}. We shall also discuss an interesting feature: radial decreasing initial data do not necessarily preserve radial monotonicity during the dynamics.\\

We finally turn the attention to the behavior of solutions as $s\to 0$. The formal limiting equation reads
\begin{equation}\label{pme}
\partial_t \rho=\Delta \rho^m+\left(\beta-\chi/2\right)\,\Delta \rho^2,
\end{equation}
and its behavior is again crucially depending on the sign of the coefficient $\beta-\chi/2$. Here, we limit ourselves to treat the case $\beta\ge \chi/2$, in which \eqref{pme} is a standard degenerate parabolic equation. We have the following

\begin{thm}\label{stozero} Let $\beta\ge \chi /2$. Let $\rho^0\in\mathcal Y_{M,2}$.
 Let $(s_n)_{\{n\in\mathbb N\}}\subset (0,1/2)$ be a vanishing sequence, and for every $n\in\mathbb N$ let $\rho_n$ be
 a gradient flow solution to \eqref{cauchy1} with $s=s_n$. 
 Then the sequence $(\rho_n)_{n\in\mathbb N}$ admits strong $L^2_{loc}((0,+\infty);L^2(\mathbb R^d))$ limit points. If  $\rho$ is one of such  limit points, then
  $[0,+\infty)\ni t\mapsto \rho(t,\cdot)$ is narrowly continuous with values in $\mathcal Y_{M,2}$, $\rho(0,\cdot)=\rho^0$ and $\rho$ is a distributional solution to the nonlinear diffusion equation \eqref{pme}, i.e.,
 \[
 -\int_0^{+\infty}\int_{\mathbb R^d}\rho(t,x)\varphi(x)\eta'(t)\,dx\,dt=\int_{0}^{+\infty}\int_{\mathbb R^d}\eta(t)\Delta\varphi(x)\,\big(\rho(t,x)^m+(\beta-\chi/2)\rho(t,x)^2\big)\,dx\,dt
 \]
 for every $\varphi\in C^\infty_c(\R^d)$ and every $\eta\in C^\infty_c((0,+\infty))$.
\end{thm}

\subsection*{Plan of the paper}
In Section \ref{stationarysection} we discuss Definition \ref{steadydef*} and prove uniqueness of stationary states along with some regularity properties.
Asymptotic behavior of  stationary states  as $s$ approaches $0$ is investigated  in Section \ref{asymptoticsection}.
In Section \ref{evolutionsection} we prove the theorems about the evolution problem.  In Section \ref{qualitative}
we provide some numerical examples demonstrating other phenomena, and  in Section \ref{open} we provide a discussion on some open problems. 

 \section{Stationary states and minimizers of the free energy}\label{stationarysection}

%

We start this section by discussing the definition of stationary states for the evolution equation in \eqref{cauchy1}, see Definition \ref{steadydef*}.
 We begin  with the following lemma about Riesz potentials of bounded continuous $W^{1,1}(\mathbb R^d)$ densities $\rho$,
 showing that indeed $\rho\nabla\psi\in L^1_{loc}(\R^d)$, where $\psi$ is the velocity potential defined by \eqref{u*}.
 The lemma includes the equivalence of two formulations for the equation that  governs stationary states in Definition \ref{steadydef*}.

\begin{lem}\label{RieszW11}
Let  $\rho\in W^{1,1}(\R^{d})\cap L^{\infty}(\R^{d})\cap C(\R^{d})$ be a nonnegative function. Then
$K_{s}\ast \rho \in  W^{1,1}_{loc}(\R^{d})$ and
$
 \nabla(K_{s}\ast \rho)=K_{s}\ast \nabla\rho.
 $
 Moreover,   $\psi\in W^{1,1}_{loc}(\R^{d})$, where $\psi$ is defined by \eqref{u*}. Finally,
\begin{equation}
 \nabla\cdot (\rho\nabla \psi)=0\quad\textit{ in }\mathcal{D}^{\prime}(\R^{d})\label{continuityeq}
\end{equation}
if  and only if
\begin{equation}\label{varequsteady}\begin{aligned}
 & \int_{\mathbb R^d}\nabla(\rho^m+\beta \rho^{2})\cdot\nabla\varphi\,dx\\&\quad+
\frac{(d-2s)\,c_{d,s}\,\chi}2\int_{\mathbb R^d}\int_{\mathbb R^d}(\nabla\varphi(x)-\nabla\varphi(y))\cdot(x-y)|x-y|^{2s-d-2}\rho(x)\,\rho(y)\,dx\,dy=0
 \end{aligned}\end{equation}
for any $\varphi\in C_{c}^{\infty}(\R^{d})$.
 \end{lem}


We postpone the technical proof of the above lemma to the appendix. Here, we focus  on its consequences in relation with the definition of stationary states and on further properties that follow from Definition \ref{steadydef*}. 

\begin{remark}\rm
By virtue of Lemma \eqref{RieszW11}, we have the equivalence between the stationary version of \eqref{veryweak} and \eqref{continuityeq}. Therefore, Definition \ref{steadydef*} agrees with
the natural definition of steady solutions of the evolution equation in \eqref{cauchy1} in the formulation \eqref{veryweak},
i.e., time independent solutions.

\end{remark}
 \begin{remark}\rm
 Definition \ref{steadydef*} provides a notion of  stationary state which is weaker than the ones previously used in \cite{CCH, CCH2,CCH3, CHMV, CHVY, DYY}, see \cite[Definition 2.1]{CCH}. Indeed, the latter definition requires
$
\rho\nabla \psi=0 \quad\text{in }\mathcal{D}^{\prime}(\R^{d})
$ along with more regularity properties of $\rho$.  We point out that  the $W^{1,1}(\mathbb R^d)$ regularity of the density $\rho$ and the $W^{1,\infty}$
regularity of the velocity potential $\psi$ in the interior of the support of $\rho$ do always agree with the properties of the minimizers of the free energy
functional $\mathcal{F}_{s}$ that we shall discuss later on. On the other hand, such minimizers do not match the notion of stationary states from \cite[Definition 2.1]{CCH} if $s$ is small.
\end{remark}

Our aim is now to show that if $\rho$ is a steady state, then we can  obtain the natural zero-dissipation identity
$$\int_{\mathbb R^d}|\nabla \psi|^2\,\rho\,dx=0,$$
for $\psi$ defined by \eqref{u*}.
 \begin{prop}\label{nodissipation}
 Assume  that $\rho$ is a stationary state for the evolution equation in \eqref{cauchy1}, according to {\rm Definition \ref{steadydef*}}. Then
 $\int_{\mathbb R^d}|\nabla \psi|^2\,\rho\,dx=0$. In particular $\psi$ is constant in each connected component of $\Omega$.
\end{prop}
\begin{proof}
Since $\psi\in W^{1,1}_{loc}(\mathbb R^d)$ by Lemma \ref{RieszW11},
for any given $\eta\in C^\infty_c(\mathbb R^d)$ we have $\eta \psi \in W^{1,1}(\mathbb R^d)$, and  then there exists a sequence $(\varphi_n)_{n\in\mathbb
N}\subset C^{\infty}_c(\mathbb R^d)$ such that $\varphi_n\to \eta \psi$ in $W^{1,1}(\mathbb R^d)$ as $n\to\infty$. Since $\mathrm{div}(\rho\nabla \psi)=0$ in $\mathcal D'(\mathbb R^d)$, we deduce that
\[
\int_{\Omega}\nabla\varphi_n \cdot\nabla \psi\,\rho\,dx=0
\]
for every $n\in\mathbb N$.
Since $\rho\nabla \psi\in L^{\infty}(\Omega)$, by taking the limit as $n\to\infty$ we find
\begin{equation}\label{witheta}
\int_{\mathbb R^d}\nabla \psi\cdot\nabla(\eta \psi)\,\rho\,dx=0
\end{equation}
for every $\eta\in C^\infty_c(\mathbb R^d)$.
Let us now consider a radially decreasing function $\zeta\in C^\infty_c(\mathbb R^d)$ such that $\zeta(x)=1$ if $|x|\le 1$ and such that $0\le\zeta(x)\le 1$ for every $x\in\mathbb R^d$. For every $k\in\mathbb N$, let $\eta_k(x)=\zeta(x/k)$, so that $\zeta_k\to1$ pointwise and monotonically on $\mathbb R^d$ and $\nabla \eta_k\to 0$ in $L^\infty(\mathbb R^d)$ as $k\to\infty$.  With this choice of the test functions, from \eqref{witheta} we get
\begin{equation}\label{limitet}
0=\int_{\mathbb R^d}\psi\rho\nabla \psi\cdot\nabla\eta_k\,dx+\int_{\mathbb R^d}|\nabla \psi|^2\,\eta_k\,\rho\,dx
\end{equation}
for every $k\in\mathbb R$. The first term in the right hand side vanishes as $k\to+\infty$, since
\[
\left|\int_{\mathbb R^d}\nabla \psi\cdot\nabla\eta_k\, \psi\,\rho\,dx\right|
\le \|\psi\|_{L^\infty(\R^{d})}\|\nabla \psi\|_{L^\infty(\Omega)}\|\rho\|_{L^1(\mathbb R^d)}\|\nabla \eta_k\|_{L^\infty(\mathbb R^d)}
\]
 and since $\nabla \eta_k\to 0$ in $L^\infty(\mathbb R^d)$. Therefore from \eqref{limitet}, by applying the monotone convergence theorem to the second term in
 right hand side, we obtain $\int_{\mathbb R^d}|\nabla \psi|^2\,\rho\,dx=0$. Since $\rho>0 $ in $\Omega$, this implies $\nabla \psi=0$ a.e. in $\Omega$, thus
 $\psi$ is constant in each connected component of $\Omega$.
\end{proof}

Stationary states are closely related to minimizers of the free-energy functional $\mathcal F_s$. Indeed, by adapting the arguments of \cite{CHMV}, we shall
prove that for every $M>0$ a minimizer of functional $\mathcal F_s$ over $\mathcal Y_M$ does exist, that it is necessarily radially decreasing and it
satisfies a suitable Euler-Lagrange equation. Thanks to these properties, a minimizer  of $\mathcal F_s$ over $\mathcal Y_M$ will be immediately seen to be  a stationary state of mass $M$.
In the analysis of minimizers we shall make use of the following notation
\begin{equation}\label{hm}\begin{aligned}
\mathcal{H}_{m}[\rho]&:=\frac{1}{m-1}\int_{\mathbb{R}^d}\rho^m(x)\,dx+\beta\int_{\mathbb{R}^d}\rho^2(x)\,dx,\\
\mathcal{W}_{s}[\rho]&:=-\frac\chi2\int_{\mathbb{R}^d}\int_{\mathbb{R}^d}K_{s}(x-y)\rho(x)\rho(y)\,dx\,dy,
\end{aligned}\end{equation}
so that
$\mathcal F_s[\rho]=\mathcal{H}_{m}[\rho]+\mathcal{W}_{s}[\rho]$. We notice that since $m>2$, both $\mathcal H_m[\rho]$ and $\mathcal W_s[\rho]$ are finite if $\rho\in L^1(\mathbb R^d)\cap L^m(\mathbb R^d)$, as a consequence of the  Hardy-Littlewood-Sobolev inequality (see \cite[Theorem 4.3]{LiebLoss}), that reads (notice that due to the condition $m>2$ we always have $1<2d/(d+2s)<m$)
\begin{equation}\label{HLS}
\int_{\mathbb R^d}\int_{\mathbb R^d}|x-y|^{2s-d}\rho(x)\,\rho(y)\,dx\,dy\le H_{d,s}\|\rho\|^2_{L^\frac{2d}{d+2s}(\mathbb R^d)},
\end{equation}
where the optimal constant $H_{d,s}$ is given by
\begin{equation*}\label{hds}
 H_{d,s}=\pi^{\frac{d-2s}{2}}\,\frac{\Gamma(s)}{\Gamma(s+d/2)}\left(\frac{\Gamma(d/2)}{\Gamma(d)}\right)^{-\frac{2s}{d}}.
 \end{equation*}
As a direct consequence, we check that for any $\rho\in \mathcal{Y}_{M}$ we have, for some constant $\bar C$ only depending on $\chi,m,s,d$,
\begin{equation}\label{boundLm}
\frac{1}{2(m-1)}\|\rho\|_{L^{m}(\R^{d})}^{m}\leq \mathcal{F}_{s}[\rho]+\bar C.
\end{equation}
Indeed, By the sharp Hardy-Littlewood-Sobolev inequality \eqref{HLS}, letting
\begin{equation}\label{sds}S_{d,s}:=c_{d,s}\,H_{d,s}=(4\pi)^{-s}\frac{\Gamma(-s+d/2)}{\Gamma(s+d/2)}\left(\frac{\Gamma(d/2)}{\Gamma(d)}\right)^{-2s/d},\end{equation} there holds
\begin{equation}\label{18}
|\mathcal{W}_{s}[\rho]|=\frac\chi 2\,c_{d,s} \int_{\mathbb R^d}\int_{\mathbb R^d}|x-y|^{2s-d}\rho(x)\,\rho(y)\,dx\,dy\le\frac\chi2\,S_{d,s} \|\rho\|^2_{L^\frac{2d}{d+2s}(\mathbb R^d)}.
\end{equation}
By taking advantage of the interpolation inequality \begin{equation*}\|\rho\|^2_{L^{\frac{2d}{d+2s}}(\mathbb R^d)}\le M^{2-2\theta}\|\rho\|^{2\theta}_{L^m(\mathbb R^d)},\qquad \theta:=\frac{(d-2s)m}{2d(m-1)}\end{equation*} and of Young inequality (taking conjugate exponents $p,p'$ with $p=2\theta$), from \eqref{18} we deduce
\begin{equation}\label{distantref}\begin{aligned}
|\mathcal{W}_{s}[\rho]|=\frac\chi2 \,c_{d,s} \int_{\mathbb R^d}\int_{\mathbb R^d}|x-y|^{2s-d}\rho(x)\,\rho(y)\,dx\,dy&\le\frac\chi2\,S_{d,s} M^{2-2\theta}\|\rho\|^{2\theta}_{L^m(\mathbb R^d)}\\&\le
\bar C+\frac1{2(m-1)}\|\rho\|^m_{L^m(\mathbb R^d)},
\end{aligned}\end{equation}
where $\bar C=\bar C(\chi,m,s,d)$ is defined by
\begin{equation}\label{cbar}
\bar C=\frac{m-2\theta}{m}\left(\frac{m}{4\theta(m-1)}\right)^{-\frac{2\theta}{m-2\theta}}\, \left(\frac\chi2 \,S_{d,s}M^{2-2\theta}\right)^{\frac{m}{m-2\theta}},\qquad \theta=\frac{(d-2s)m}{2d(m-1)}.
\end{equation}
Then we have
\[
\frac{1}{m-1}\int_{\R^{d}}\rho^{m}dx\leq \mathcal{F}_{s}[\rho]+|\mathcal{W}_{s}[\rho]|\leq \mathcal{F}_{s}[\rho]+
\bar C+\frac1{2(m-1)}\|\rho\|^m_{L^m(\mathbb R^d)}
\]

%
We have the following result about existence of minimizers, which employs the classical concentration compactness theorem of Lions \cite{L}.
\begin{lem}\label{Minimiz} 
The functional $\mathcal{F}_s$ admits a minimizer over $\mathcal{Y}_M$. If $\rho_s\in\mathrm{argmin}_{\mathcal{Y}_M}\mathcal{F}_s$, then $\rho_s$ is radially decreasing, compactly supported and it satisfies
\begin{equation}\label{Euler}
\frac m{m-1}\,\rho_s^{m-1}+2\beta\rho_{s}=\left(\chi K_{s}\ast\rho_s-\mathcal{C}_s\right)_+\qquad\mbox{ in $\mathbb{R}^d$},
\end{equation}
where
\begin{equation}\label{explicitconstant}
0<\mathcal{C}_s:=-\frac2M\mathcal{F}_s[\rho_s]-\frac1M\frac{m-2}{m-1}\int_{\mathbb{R}^d}\rho_s^m(x)\,dx.
\end{equation}
In particular, we have
\begin{equation}\label{functminim}
\mathcal{F}_{s}[\rho_{s}]=-\frac{1}{d-2s}\left(\frac{dm-2d+2s}{m-1}\|\rho_s\|_{m}^{m}+2s\beta\|\rho_s\|_{2}^{2}\right)<0
\end{equation}
\begin{equation}\label{explicitconstant1}
\mathcal{C}_s=\frac{1}{M(d-2s)}\left(\frac{dm+2sm-2d}{m-1}\|\rho_s\|_{m}^{m}
+4\beta s\|\rho_{s}\|_{2}^{2}
\right)
\end{equation}
\end{lem}

For the proof of the above lemma
we follow the concentration compactness argument as applied in Appendix A.1 of \cite{KY}. Indeed, the proof is based on \cite[Theorem II.1, Corollary II.1]{L}, that are next recalled, denoting by $\mathcal{M}^p(\R^d)$ the Marcinkiewicz space or weak $L^p$ space.
\begin{thm}{{\rm \cite[Theorem II.1]{L}}}\label{Lionsthm}
 Suppose $K\in \mathcal{M}^p(\R^d)$, $1<p<\infty$, and consider the problem
 \begin{equation*}
  I_M = \inf_{\rho \in \mathcal{Y}_{q,M}} \left\{\int_{\R^d} \mathsf{j}(\rho)dx - \frac{\chi}{2}\iint_{\R^d\times \R^d} K(x-y)\rho(x)\rho(y)\,dxdy\right\}\, .
 \end{equation*}
 Here,
 \begin{equation*}
  \mathcal{Y}_{q,M}=\left\{\rho \in L^q(\R^d)\cap L^1(\R^d)\, , \,  \rho \geq 0\, \,  a.e., \int_{\R^d} \rho(x) \, dx=M\right\}\, ,
  \qquad q=\frac{p+1}{p}\,,
 \end{equation*}
 and the nonlinearity $\mathsf{j}:\R^{+}\rightarrow\R^{+}$ is a strictly convex nonnegative function such that
 \[
 \lim_{t\rightarrow 0^{+}}\frac{\mathsf{j}(t)}{t}=0,\quad \lim_{t\rightarrow +\infty}\frac{\mathsf{j}(t)}{t^{q}}=+\infty.
 \]
 Then there exists a minimizer of problem $(I_M)$ if the following holds:
\begin{equation}\label{cc1}
 I_{M_0} < I_{M} + I_{M_0-M}
 \qquad \text{for all} \, \, M \in (0,M_0)\, .
\end{equation}
\end{thm}
\begin{prop}{{\rm \cite[Corollary II.1]{L}}}\label{Lionsprop}
 Suppose there exists some $\lambda \in (0,N)$ such that
 $$
 K(tx) \geq t^{-\lambda}K(x)
 $$
 for all $t\geq 1$. Then \eqref{cc1} holds if and only if
 \begin{equation}\label{cc2}
 I_{M} < 0
 \qquad \text{for all} \, \,M>0\, .
\end{equation}
\end{prop}
\begin{proofad1}
The proof is similar to \cite[Theorem 5]{CHMV}, but we sketch the main lines of the arguments for the sake of completeness.
Let $p=\tfrac{d}{d-2s}$ and $q=\tfrac{p+1}{p}$.
We first notice that our potential $K_s(x)=c_{d,s}|x|^{2s-d}$ is in $\mathcal{M}^p(\R^d)$ and it is clear that $K_{s}$ verifies
the homogeneity assumption in Proposition \ref{Lionsprop} for $\lambda=d-2s$. Moreover the nonlinearity
\[
\mathsf{j}(t)=\frac{1}{m-1}t^{m}+\beta t^{2}
\]
verifies the properties of Theorem \ref{Lionsthm}, since $m>2$ implies $m>q$. Then we just have to show that there exists some density $\rho \in \mathcal{Y}_{q,M}$ such that
$\mathcal{F}_s[\rho]<0$. We fix $R>0$ and define
$$
\rho_*(x):=\frac{dM}{\sigma_{d} R^d}\,  \bC_{B_R}(x)\, ,
$$
where $B_R$ denotes the ball centered at zero and of radius $R>0$, and where $\sigma_d=2 \pi^{d/2}/\Gamma(d/2)$ is the surface area of the $d$-dimensional unit ball.
Then
\begin{align*}
 \mathcal{H}_{m}[\rho_*]&=\frac{1}{m-1} \int_{\R^d} \rho_*^m dx+\beta\int_{\mathbb{R}^d}\rho^2(x)\,dx
 =\frac{(dM)^m \sigma_d^{1-m}}{d(m-1)} \, R^{d(1-m)}+\beta\frac{(dM)^2 \sigma_d^{-1}}{d} \, R^{-d}\, , \\
 \mathcal{W}_s[\rho_*]
 &=-\frac{\chi}{2}\iint_{\R^d\times \R^d} K_s(x-y)\rho_*(x)\rho_*(y)\, dx dy\\
 &=  -\chi\,c_{d,s}\frac{(dM)^2}{2\sigma_d^2 R^{2d}}\iint_{\R^d\times \R^d}|x-y|^{2s-d}\bC_{B_R}(x)\bC_{B_R}(y)\, dx dy \\
 & \leq -\chi \,c_{d,s}\frac{(dM)^2}{2\sigma_d^2 R^{2d}} (2R)^{2s-d}\, \frac{\sigma_d^2}{d^2}R^{2d}
 = -\chi \,c_{d,s}\,2^{2s-d-1}M^2 R^{2s-d}<0\, .
\end{align*}
Therefore we find
$$
\mathcal{F}_s[\rho_*]\leq
\frac{M^m d^{m-1} \sigma_d^{1-m}}{(m-1)} \, R^{d(1-m)}+\beta d\,M^2 \sigma_d^{-1} \, R^{-d}-\chi \,c_{d,s}\,2^{2s-d-1}M^2 R^{2s-d}\, .
$$
Since $m>2$ we have $d(1-m)<-d<2s-d$, then we can choose $R>0$ large enough such that $\mathcal{F}_s[\rho_*]<0$, and hence condition \eqref{cc2}
is verified. Then Proposition \ref{Lionsprop} and Theorem \ref{Lionsthm} implies that there exists a minimizer $ \rho_s$ of $\mathcal{F}$ in $\mathcal{Y}_{q,M}$

Now,  the Hardy-Littlewood-Sobolev inequality \eqref{HLS}  implies that $\rho_s\in L^{m}(\R^d)$.
Moreover, the Riesz rearrangement inequality \cite[Theorem 3.7]{LiebLoss} yields that $ \rho_s$ is radially decreasing (then $K_s\ast\rho_s$ is also radially decreasing and vanishing at infinity).

The fact that $\rho_s$ satisfies \eqref{Euler} can be obtained by computing the first variation of $\mathcal{F}_s$ at $ \rho_s$, as  done in \cite[Theorem 3.1]{CCV}. Then \eqref{explicitconstant} follows by taking into account the equation \eqref{Euler} satisfied by $\rho_s$, multiplying it by $\rho_s$ and integrating over $\mathbb R^d$. Since $\mathcal C_s$ in \eqref{Euler} is positive, and since both $\rho_s$ and $K_s\ast\rho_s$ are radially decreasing and vanishing at infinity, from \eqref{Euler} we deduce that $\rho_s$ is compactly supported. In particular $x\rho_s\in L^1(\mathbb R^d)$ and $\rho_s\in \mathcal Y_M$.

  Eventually we  prove \eqref{functminim} and \eqref{explicitconstant1}. Using the mass invariant dilation
$
\rho_{s,\lambda}(x)=\lambda^{d}\rho_{s}(\lambda\,x)
$
we easily find
\[
h(\lambda):=\mathcal{F}_{s}[\rho_{s,\lambda}]=\frac1{m-1}\lambda^{d(m-1)}\|\rho_{s}\|^{m}_{m}+\beta\lambda^{d}\|\rho_{s}\|_{2}^{2}+\lambda^{d-2s}\mathcal{W}_{s}[\rho_{s}].
\]
By the minimality of $\rho_{s}$, the function $h(\lambda)$ has its unique minimizer at $\lambda=1$, which implies by differentiation
\[
\|\rho_{s}\|^{m}_{m}+\beta\|\rho_{s}\|_{2}^{2}=-\frac{d-2s}{d}\mathcal{W}_{s}[\rho_{s}].
\]
Then we have \eqref{functminim} hence substituting in the expression of $\mathcal{C}_{s}$ we find  \eqref{explicitconstant1}.
\end{proofad1}

Concerning the regularity of the minimizers, we have the following result
\begin{lem}\label{regularitymin}
Let $\beta>0$.
Any minimizer in $\mathcal{Y}_{M}$ of $\mathcal{F}_s$ is bounded and Lipschitz continuous in $\R^{d}$ and it is smooth inside its support.
\end{lem}
\begin{proof}
The boundedness of the minimizers follows by \cite[Theorem 7]{CHMV} up to minor modifications, nevertheless in the case $s\in (0,1)$ we propose here a more direct proof based on
\cite[Proposition 2.8]{CGHMV}. Indeed, let $\rho_{0}\in \mathcal{Y}_{M}$ be any minimizer of $\mathcal{F}_s$, which is radially decreasing and compactly supported by Lemma \ref{Minimiz}. Then \eqref{Euler} implies that the Riesz potential $v_{0}:=K_{s}\ast\rho_{0}$ is a radially decreasing, vanishing at infinity solution to the fractional PDE
\begin{equation}
(-\Delta)^{s}v=\mathfrak{g}(v),\label{fraPDE}
\end{equation}
where the nonlinearity $\mathfrak{g}$ is defined as $\mathfrak{g}(t)=f^{-1}((t-\mathcal{C}_{s})_{+})$, being $f^{-1}$ the inverse function of the convex nonlinearity $f(t)=\frac m{m-1}\,t^{m-1}+2\beta t$, $t\geq0$. Since $\beta>0$, $f^{-1}$ is Lipschitz continuous therefore so is $\mathfrak{g}$. Now, since $\rho_{0}\in L^{m}(\mathbb R^d)\cap L^{1}(\mathbb R^d)$ with $m>2$, we have $\rho_{0}\in L^{2d/(d+2s)}(\mathbb R^d)$  \KKK thus the Hardy-Littlewood-Sobolev inequality implies that $v_{0}\in L^{2^{\ast}_{s}}(\mathbb R^d)$, being $2^{\ast}_{s}=2d/(d-2s)$. Since $\mathfrak{g}(v_{0})\leq (2\beta)^{-1} v_{0}$,  \KKK we find in particular $\mathfrak{g}\circ v\in L^{2d/(d+2s)}(\mathbb R^d)$: indeed,
if $q=2d/(d+2s)$, by H\"{o}lder inequality and the radial monotonicity there is a constant $C\ge 0$ such that
\[
\int_{\mathbb R^d} \mathfrak{g}(v_{0})^{q}\,dx=\int_{\{v_0> \mathcal C_s\}} \mathfrak{g}(v_{0})^{q}\,dx\leq C\int_{\{v_{0}>\mathcal{C}_{s}\}}v_{0}^{q}\leq C \left(\int_{\mathbb R^d} v_{0}^{2^{\ast}_{s}}\,dx\right)^{\frac{d-2s}{d+2s}}.
\]
Then we can proceed exactly as in \cite[Proposition 2.8]{CGHMV} in order to conclude that $v_{0}\in L^{\infty}(\mathbb R^d)$, which implies in turns by \eqref{Euler} that $\rho_{0}\in L^{\infty}(\mathbb R^d)$.

Now we turn to the regularity properties of $\rho_{0}$ and we refer mainly to \cite[Theorem 8]{CHMV}. Let us first consider the easiest case $s>1/2$. We have in this case $u_{0}\in W^{1,\infty}(\R^{d})$ from \cite[Lemma 1]{CHMV}, thus by \eqref{fraPDE} we have
\begin{equation}
\rho_{0}=\mathfrak{g}(u_{0}),\label{bootstrap}
\end{equation}
therefore (since $\mathfrak{g}$ is Lipschitz) $\rho_{0}\in W^{1,\infty}(\R^{d})$. If $s\in (0,1/2)$, by the first part of the proof of \cite[Theorem 8]{CHMV} it follows that $u_{0}\in C^{0,\gamma}$ for any $\gamma<2s$, then equation \eqref{bootstrap} gives $\rho_{0}\in C^{0,\gamma}$ for any $\gamma<2s$. By the H\"{o}lder regularity of the Riesz potential (see again \cite[Eq. 3.24]{CHMV}) we have $u_{0}\in C^{0,\gamma}$ for any $\gamma<4s$. Bootstrapping, we finally find $u_{0},\,\rho_{0}\in W^{1,\infty}(\R^{d})$. The case $s=1/2$ can be treated analogously, see again the proof of \cite[Theorem 8]{CHMV}. Finally, \cite[Theorem 10]{CHMV} gives the smoothness of $\rho_{0}$ inside its support.
\end{proof}

\begin{remark}\label{rmbeta0}\rm
The case $\beta=0$ of  Lemma \ref{regularitymin} is treated in \cite[Theorem 8, Remark 2]{CHMV}. As shown therein, if $\beta=0$  any minimizer in $\mathcal{Y}_{M}$ of $\mathcal{F}_s$ is bounded, smooth inside its support, and enjoys suitable H\"older regularity on $\mathbb R^d$ depending on $m,s$.
\end{remark}
\KKK

It is now easy to check that minimizers of functional $\mathcal F_s$ over $\mathcal Y_M$, whose existence is ensured by Lemma \ref{Minimiz}, are stationary states.

\begin{prop}\label{id} Let $\beta\ge 0$. If $\rho_s$ minimizes  $\mathcal F_s$ over $\mathcal Y_M$, then $\rho_s$ is a stationary state according to {\rm Definition \ref{steadydef*}}.
\end{prop}
\begin{proof}  By Lemma \ref{Minimiz}, Lemma \ref{regularitymin} and Remark \ref{rmbeta0},   $\rho_s$ is continuous, compactly supported, radially decreasing
  and smooth on $B_R$, where $R$ is the radius of its support. Therefore its radial profile is absolutely continuous in $[0,R]$, and belongs to the weighted
  Sobolev space $W^{1,1}((0,R), r^{d-1}\,dr)$. This implies $\rho_s\in W^{1,1}(B_R)$, see \cite[Theorem 2.3]{F}. Since $\rho_s$ is also vanishing on $\mathbb
  R^d\setminus B_R$, we conclude that it belongs to $W^{1,1}(\mathbb R^d)$. By Lemma \ref{RieszW11}, we get $\psi\in W^{1,1}_{loc}(\mathbb R^d)$, where $\psi$ is defined by \eqref{u*},
 thus $\rho\nabla \psi\in L^1(\mathbb R^d)$. But the validity of \eqref{Euler} implies that $\psi$ is constant on $B_R$. In particular $\rho\nabla \psi=0$ a.e. in
 $\mathbb R^d$ and \eqref{continuityeq} follows.
\end{proof}

The complete characterization of stationary states (according to Definition \ref{steadydef*}) is finally  given by the next theorem.

\begin{thm}\label{furtherregularity}
Let $\beta>0$ (resp. $\beta=0$).  For any mass $M>0$, there exists a unique stationary state (resp. a unique radial stationary state) of mass $M$ and center of mass $0$. Such steady state is radially decreasing, compactly supported, Lipschitz on $\mathbb R^d$ (resp. H\"older on $\mathbb R^d$) and smooth inside its support. Moreover, it coincides (up to translation) with the unique minimizer  of the energy functional $\mathcal{F}_{s}$ in the class $\mathcal{Y}_{M}$.
\end{thm}
\begin{proof}
Let $\beta>0$. Let $\rho$ be a stationary state according to Definition \ref{steadydef*}. Let $v:=K_s\ast\rho$. Besides $v\in W^{1,1}_{loc}(\mathbb R^d)$, which is proven in Lemma \ref{RieszW11}, by using Sobolev embeddings it is possible to prove that exists $\gamma\in(0,1)$ such that $v\in C^{0,\gamma}(\mathbb R^d)$ (see again \cite[Theorem 8]{CHMV}).


Let as usual $\Omega=\left\{\rho>0\right\}$. Since $\rho$ is continuous, $\Omega$ is open and thus $\Omega=\cup_{n=1}^{\infty}\Theta_n$, where the $\Theta_n$'s are the countably many open connected components of $\Omega$. Let us introduce the continuous functions
\begin{equation}\label{roenne}
\rho_n(x):=\left\{\begin{array}{ll}\rho(x)\quad&\mbox{if $x\in\Theta_n$}\\0\quad&\mbox{if $x\in\mathbb R^d\setminus\Theta_n$}.\end{array}\right.
\end{equation}
From Proposition \ref{nodissipation}, for each $n\in\mathbb N$  there is a constant $Q_n\in \mathbb R^d$ such that there holds
\[
f (\rho_n(x))=f(\rho(x))=\chi\,v(x)-Q_n\quad\mbox{ for every $x\in\Theta_n$},
\]
where $f(t):=\tfrac{m}{m-1} t^{m-1}+2\beta t$, $t\ge 0$. The constant $Q_n$ is nonnegative, since $\rho$ and $v$ are nonnegative continuous and $\rho$ vanishes on $\partial\Theta_n$.
 Therefore we have
\[\rho_n(x)=\mathfrak g(\chi v(x)-Q_n)\quad\mbox{ for every $x\in\overline{\Theta_n}$},
\]
 where $\mathfrak g:[0,+\infty)\to[0,+\infty)$ is the inverse function of $f$, and we notice that $\mathfrak g$ is Lipschitz on $[0,+\infty)$ since $\beta>0$. Hence, we have $\rho_n\in C^{0,\gamma}(\overline{\Theta_n})$, and \eqref{roenne} implies $\rho_n\in C^{0,\gamma}(\mathbb R^d)$. We  stress that the H\"older constant of $\rho_n$, that we denote by $c$, is independent of $n$, since it only depends on the H\"older constant of $v$ and the Lipschitz constant of $f$. Now, for every two distinct points $x\in \mathbb R^d$ and $y\in\mathbb R^d$, there exist $n\in\mathbb N$ and $m\in\mathbb N$ such that $\rho(x)=\rho_n(x)$ and $\rho(y)=\rho_m(y)$ and
 \[
 |\rho(x)-\rho(y)|=|\rho_n(x)-\rho_m(y)|\le |\rho_n(x)-\rho_n(y)|+|\rho_m(x)-\rho_m(y)|\le 2c|x-y|^\gamma,
 \]
 so that $\rho\in C^{0,\gamma}(\mathbb R^d)$. Since $v=K_s\ast\rho$, the H\"older estimates from \cite{RS} entail $v\in C^{0,\gamma+2s-\eps}(\mathbb R^d)$ for every arbitrarily small $\eps$. Then we may bootstrap this argument and in a finite number of steps we get $v\in C^{0,1}(\mathbb R^d)$ and $\rho\in C^{0,1}(\mathbb R^d)$. Since $m>2$ and $\rho$ is bounded, we conclude that $\rho^{m-1}$ is Lipschitz on $\mathbb R^d$ as well. We recall that by Proposition \eqref{nodissipation} the the velocity potential $\psi$ defined by \eqref{u*}
is constant on each connected component of $\rho$.

Now, an easy modification of \cite[Theorem 3]{CHMV}, which crucially exploits the Lipschitz regularity of $\rho^{m-1}$, shows the radiality of the steady state $\rho$: indeed, the actual entropy
\[
\mathcal{H}_m[\rho_*]=\frac{1}{m-1} \int_{\R^d} \rho_*^m dx+\beta\int_{\mathbb{R}^d}\rho_{*}^2(x)\,dx
\]
decreases under the modified continuous Steiner symmetrization introduced in the proof therein.

All the regularity properties of the steady states can be argued by the proof of Lemma \ref{regularitymin}. The uniqueness follows by the results of \cite{DYY}, namely by the Remark under the statement of \cite[Remark 1.2]{DYY}, because the nonlinearity $\Phi:\R^{+}\rightarrow\R^{+}$, defined by $\Phi(\rho)=\frac{1}{m-1}\rho^{m}+\beta\rho^{2}$  entering in the nonlinear diffusion is a strictly increasing smooth convex function. Finally, the identification (up to translation)  with the unique minimizer of $\mathcal{F}_{s}$ over $\mathcal Y_M$ follows from Proposition \ref{id}.

Finally, if $\beta=0$, we only consider the class of radial stationary states, and the uniqueness in this class follows from the results in \cite{DYY}. Again the unique radial stationary states is the unique minimizer of $\mathcal{F}_{s}$ over $\mathcal Y_M$, see Proposition \ref{id}, and the other properties  follow from Lemma \ref{Minimiz} and \cite[Theorem 8]{CHMV}.
\end{proof}

\begin{remark}\label{remark}\rm
If $\beta=0$ the general uniqueness result of Theorem \ref{furtherregularity} (without assuming radiality) is still an open problem, mainly due to the fact that the inverse function of $f(t)=t^{m-1}$ is $C^{1/(m-1)}$ for $m>2$ but not Lipschitz, preventing to obtain that $\rho^{m-1}$ is Lipschitz for a stationary state $\rho$, which is crucial for applying the radiality result of \cite{CHVY}. However, we still have radiality of every stationary states, thus uniqueness, as long as
we assume in addition that  
 $m<m^*$, 
 because in this case $\rho^{m-1}$ can be proven to be Lipschitz, see \cite[Theorem 8]{CHMV}. Here, $m^*:=\tfrac{2-2s}{1-2s}$ if $s<1/2$ and $m^*:=+\infty$ otherwise (so that there is no restriction if $d\ge 2$ and $s\ge 1/2$).
 \KKK
\end{remark}

\begin{proofad2} The result follows from Theorem \ref{furtherregularity} and Remark \ref{remark}.
\end{proofad2}

\section{Asymptotic behavior of stationary states as $s\to 0$}\label{asymptoticsection}

We investigate the asymptotic behavior of  $\rho_s$ as $s$ approaches $0$, where $\rho_s$ is, for each small $s$, the unique stationary state of given mass $M$ and center of mass $0$ for the equation in \ref{cauchy1}, provided by Theorem \ref{furtherregularity}.
Thanks to the identifications with minimizers of the free-energy,
this will be done by showing that functionals $\mathcal{F}_{s}$  $\Gamma$-converge to the limit energy functional $\mathcal{F}_{0}$ defined by
\begin{equation}
\mathcal{F}_0[\rho]:=\frac{1}{m-1}\int_{\mathbb{R}^d}\rho^m(x)\,dx-\frac{\chi-2\beta}{2}\int_{\mathbb{R}^d}\rho^2(x)\,dx\label{limitfunct},
\end{equation}
whose minimization is governed by the following proposition.

\begin{prop}[Minimization of the limit functional $\mathcal F_0$.]\label{minF0}
Suppose that $0\le \beta<\chi/2$. Then functional $\mathcal F_0$ admits a unique radially decreasing minimizer over $\mathcal Y_M$, given by
\begin{equation}\label{rad0}
\rho_0(x):=\left(\frac{\chi-2\beta}{2}\right)^{1/(m-2)}\,\bC_{B_{R_0}}(x),\qquad\mbox{where}\;\;\; R_0=\left(\frac{d M}{\sigma_{d}}\right)^{1/d}\left(\frac{\chi-2\beta}{2}\right)^{-\frac{1}{d(m-2)}}.
\end{equation}
Else if $\beta\ge \chi/2$, functional $\mathcal F_0$ does not admit a minimizer over $\mathcal Y_M$ and $\inf_{\mathcal Y_M}\mathcal F_0=0$.
\end{prop}
\begin{proof}
Through the proof, we  let for simplicity $\gamma:=-\beta+\chi/2$. For every $\rho\in \mathcal Y_M$ and every $\lambda>0$,  letting $\rho_\lambda(x):=\lambda^d\rho(\lambda x)$, $x\in\mathbb R^d$, direct computations show that
\begin{equation}\label{scalediff}
\begin{aligned}
&\mathcal F_0[\rho_\lambda]=\frac1{m-1}\lambda^{d(m-1)}\int_{\mathbb R^d}\rho^m(x)\,dx-\gamma\lambda^d\int_{\mathbb R^d}\rho^2(x)\,dx,
\\
&
\frac{d}{d\lambda} \mathcal F_0[\rho_\lambda]=d\lambda^{d-1}\left(\lambda^{d(m-2)}\int_{\mathbb R^d}\rho^m(x)\,dx-\gamma\int_{\mathbb R^d}\rho^2(x)\,dx\right).\end{aligned}
\end{equation}

Assume that $0\le \beta<\chi/2$. In this case, by \eqref{scalediff}
 the map $(0,+\infty)\ni\lambda\mapsto \mathcal F_0(\rho_\lambda)$ is uniquely minimized at $$\lambda=\lambda_*:= \left({\gamma\int_{\mathbb R^d}\rho^2(x)\,dx}\right)^{\frac1{d(m-2)}}\left(\int_{\mathbb R^d}\rho^m(x)\,dx\right)^{-\frac1{d(m-2)}}$$
with value
\[
\mathcal F_0[\rho_{\lambda_*}]=\frac{2-m}{m-1}\,\gamma^{\frac{m-1}{m-2}}{\left(\int_{\mathbb R^d}\rho^2(x)\,dx\right)^{\frac{m-1}{m-2}}}{\left(\int_{\mathbb R^d}\rho^m(x)\,dx\right)^{-\frac{1}{m-2}}}.
\]
But writing
\[
2=\frac{m-2}{m-1}+\frac{m}{m-1},
\]
H\"older inequality with exponents $p=(m-1)/(m-2)$, $p^{\prime}=m-1$ yields  $$\left(\int_{\mathbb R^d}\rho^2(x)\,dx\right)^{{m-1}}\le M^{m-2}\int_{\mathbb R^d}\rho^m(x)\,dx,$$ thus for every $\rho\in\mathcal Y_M$ there holds
\begin{equation}\label{nextmin}
\mathcal F_0[\rho]\ge \frac{2-m}{m-1}\,\gamma^{\frac{m-1}{m-2}}{\left(\int_{\mathbb R^d}\rho^2(x)\,dx\right)^{\frac{m-1}{m-2}}}{\left(\int_{\mathbb R^d}\rho^m(x)\,dx\right)^{-\frac{1}{m-2}}}\ge M\,\frac{2-m}{m-1}\,\gamma^{\frac{m-1}{m-2}}.
\end{equation}
Therefore, if there is a density $\rho$ achieving the constant at the right-hand side of \eqref{nextmin}, then $\rho$ is a minimizer of $\mathcal{F}_{0}$.
However, the above H\"older inequality is an equality if and only $\rho$ is a multiple of a characteristic function, i.e., $\rho(x)=t\bC_\Omega(x)$ for some measurable subset $\Omega$ of $\mathbb R^d$, and the condition $\int_{\mathbb R^d}\rho(x)\,dx=M$ implies $|\Omega|>0$ and $t=M|\Omega|^{-1}$. In particular, the second inequality in \eqref{nextmin} is an equality if and only if $\rho(x)=M|\Omega|^{-1}\bC_\Omega(x)$. On the other hand, the first inequality in \eqref{nextmin} is an equality if and only if $\rho=\rho_{\lambda_*}$, i.e., $\lambda_*=1$, which means $\gamma\int_{\mathbb R^d}\rho^2(x)\,dx=\int_{\mathbb R^d}\rho^m(x)\,dx$, and this condition, in case $\rho(x)=M|\Omega|^{-1}\bC_\Omega(x)$, entails $|\Omega|=M\gamma^{-\frac{1}{m-2}}$.
 We conclude that both inequalities in \eqref{nextmin} are equalities if and only if  $\rho(x)=M|\Omega|^{-1}\bC_\Omega(x)$ for some measurable set $\Omega\subset\mathbb R^d$ such that  $|\Omega|=M\gamma^{-\frac{1}{m-2}}$, implying that this family of functions coincides with $\argmin_{\mathcal Y_M}\mathcal F_0$ up to translations. In this family there is a unique radially decreasing function $\rho_0$, obtained by letting $\Omega=B_R$ and by choosing $R$ such that $|\Omega|=M\gamma^{-\frac{1}{m-2}}$. Hence, $\rho_0$  is given by \eqref{rad0}.

The statement concerning  the case $\beta\ge\chi/2$ (i.e., $\gamma\le 0$) follows by letting $\lambda\to 0$ in \eqref{scalediff}.
\end{proof}


\begin{figure}[htp]
\begin{center}
 \includegraphics[totalheight=0.25\textheight]{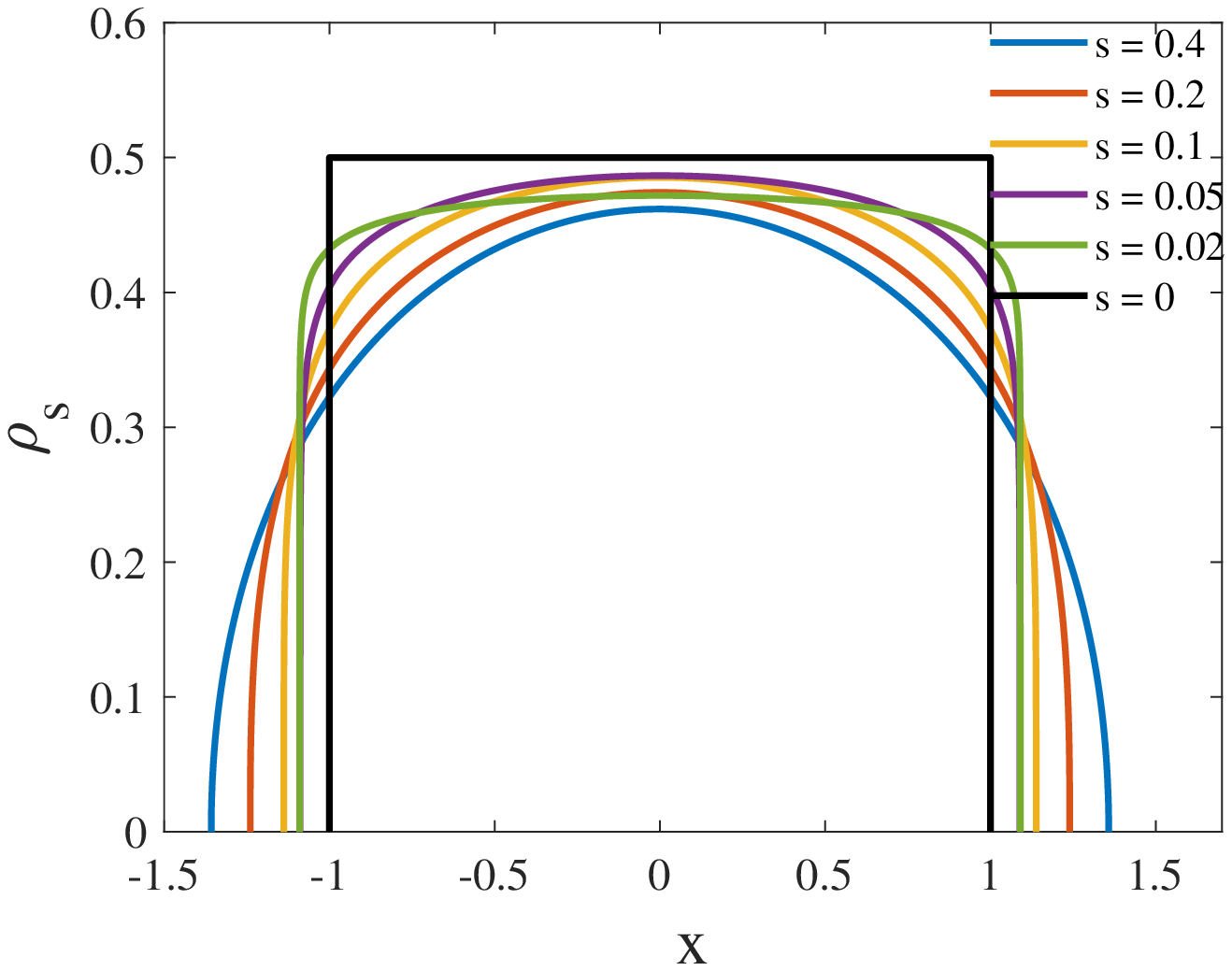}$~$
 \includegraphics[totalheight=0.26\textheight]{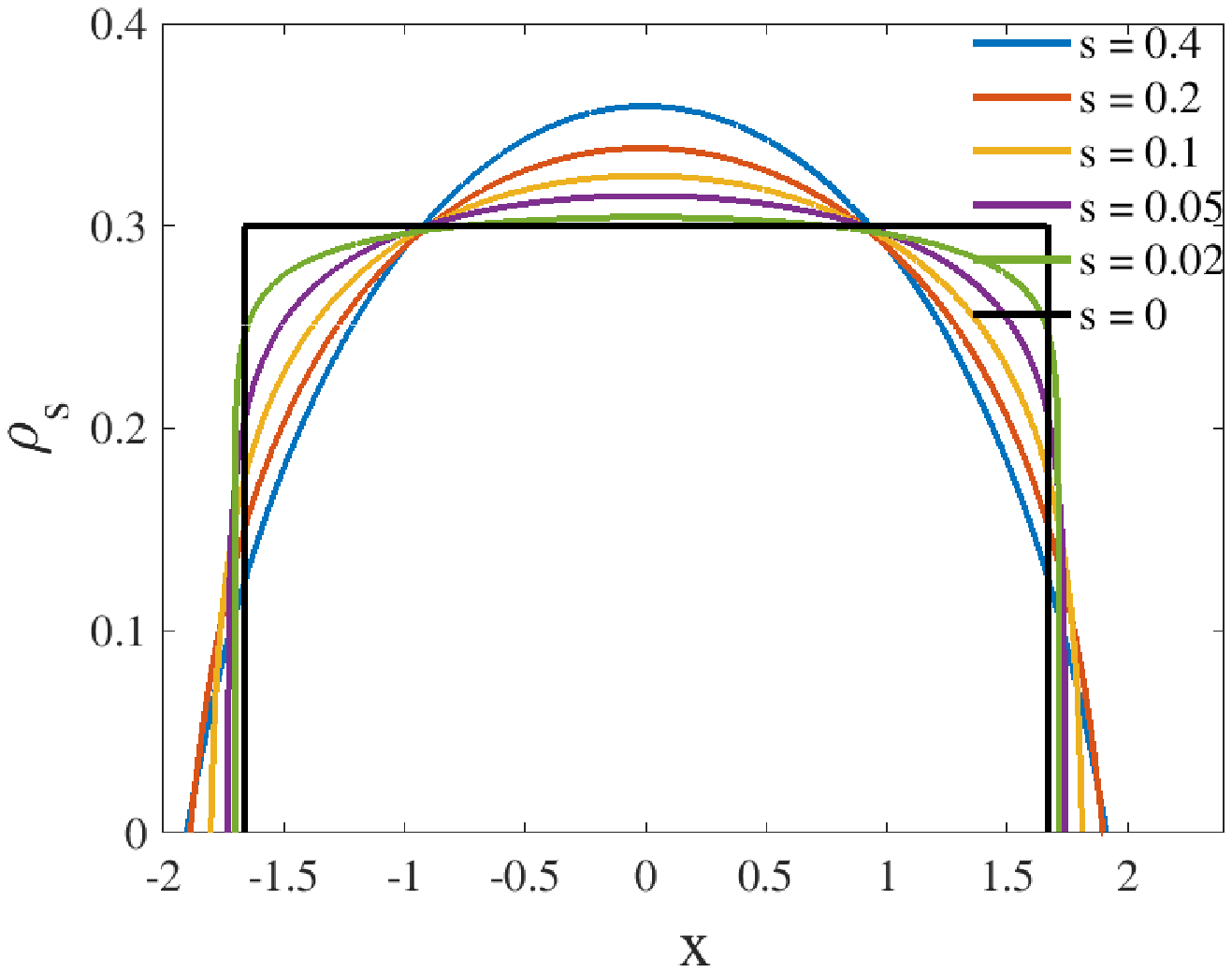}
\end{center}
\caption{The steady states for different $s>0$ with $m=3$ and $\chi = 1$ (Left figure: $\beta=0$ and Right figure: $\beta=0.2$). The expected limiting steady
state with $s=0$, which is a characteristic function with height $\left(\frac{\chi-2\beta}{2}\right)^{1/(m-2)}$ is also plotted for reference. }
\label{fig:steady_stateb01}
\end{figure}

The steady states $\rho_s$ in one dimension for different values of $s$ are shown in Figure~\ref{fig:steady_stateb01}, with $m=3$, $\chi=1$, obtained
by iterating the governing equation~\eqref{Euler}, i.e.,
\[
  \frac{m}{m-1}\big(\tilde{\rho}_s^{(new)}\big)^{m-1} + 2\beta \tilde{\rho}_s^{(new)}
  =\big(\chi K_x\ast \rho_s^{(old)}-C_s^{(old)})_+,
\]
with $C_s^{(old)}$ given by Eq.~\eqref{explicitconstant}, followed by a spatial scaling $\rho_s^{(new)}(x) = \tilde{\rho}_s^{(new)}\big(\lambda^{(new)}x\big)$ such that
the total mass is exactly $M$.
The convergence towards the limit $\rho_0$ in Proposition~\ref{minF0}
as $s$ goes to zero is
illustrated, for both $\beta=0$ (left figure) and $\beta>0$ (right figure).
The regularizing effect with $\beta>0$ is obvious, especially near the boundary of the support.\\

The next step is the investigation of  the behavior of $\mathcal{F}_s$ on characteristic functions.
\begin{lem}\label{lemma:ball}
Let $s\in (0,1)$ ($s<1/2$ if $d=1$). 
For any $R>0$, let $\mathcal{Y}_M\ni\rho_R:=\frac{dM}{\sigma_d R^d}\,\bC_{B_R}$, where $B_R$ is the ball of radius $R$, centered at the origin. Assume that $0\leq\beta<\chi/2$. Then there exists a unique positive number $R_s$ such that
$$
\mathcal{F}_s[\rho_{R_s}]=\min_{R>0} \mathcal{F}_s[\rho_R].
$$
In particular, for $\beta=0$ its value is
\begin{equation*}\label{erreesse}
R_s=\left(\frac{2\sqrt{\pi}M^{m-2}\,\Gamma(s+1)\,\Gamma(\tfrac d2+s+1)}{\chi \omega_d^{m-2}\,\Gamma(s+\tfrac12)\,\Gamma(\tfrac d2-s+1)}\right)^{\frac{1}{2s+d(m-2)}}.
\end{equation*}
Moreover, the map $(0,1/2)\ni s\mapsto \mathcal{F}_s[\rho_{R_s}]$ is continuous, it has negative value for any $s\in(0,1/2)$ and there holds
\begin{equation}
\lim_{s\downarrow 0} \mathcal{F}_s[\rho_{R_s}]= -M\frac{m-2}{m-1}\left(\frac{\chi-2\beta}{2}\right)^{\frac{m-1}{m-2}}.\label{limitchar}
\end{equation}
In particular, the value \eqref{limitchar} is exactly $\mathcal{F}_{0}(\rho_{R_0})$, where $R_{0}$  is given by \eqref{rad0}.
\end{lem}
\begin{proof}
By the proof of Lemma \ref{Minimiz} we argue that for large $R$ we have $\mathcal{F}_s[\rho_R]<0$. Moreover, by the proof of \cite[Lemma 6.1]{CGHMV} one has
\[
\mathcal{F}_{s}[\rho_{R}]= C_{1} R^{-\alpha}+C_{2} R^{-\gamma}-C_{3}R^{-\delta}
\]
where
\[
\alpha=d(m-1),\,\gamma=d,\,\delta=d-2s,\,\alpha>\gamma>\delta
\]
and
\[C_{1}=\frac{(dM)^m \sigma_d^{1-m}}{d(m-1)},\quad C_{2}=\beta\frac{(dM)^2 \sigma_d^{-1}}{d}, \,
C_{3}=\frac{\chi\,d^{2}\,M^2 \Gamma(s+\frac{1}{2})\Gamma(\frac{d}{2}-s)}
  {4\sqrt{\pi}\sigma_d \Gamma(s+1)\Gamma(\frac{d}{2}+s+1)}.
\]
Since $m>2$ and $\beta<\chi/2$, it is then readily seen that for all $s\in [0,1/2)$, the map $(0,+\infty)\ni R\mapsto\mathcal{F}_s[\rho_R]$ admits a unique minimizer $R_s$, which is the unique solution to the algebraic equation
\begin{equation}
\mathfrak{h}(s,R):=-\alpha C_{1}-\gamma C_{2}R^{\alpha-\gamma}+\delta C_{3}R^{\alpha-\delta}=0.\label{algebraic}
\end{equation}
Due to the structure of the coefficients, observe that the map $\mathfrak{h}(\cdot,R)$ is continuous for $s\in [0,1/2)$. We claim that the map $s\mapsto R_{s}$ is continuous up to $s=0$ in the interval $(0,1/2)$. To prove the claim, for any $s_{0}\in [0,1/2)$, let
$(s_0,R_{s_0})$ be the unique solution to \eqref{algebraic}. Then we have
\[
\frac{\partial\mathfrak{h}}{\partial R}(s_{0},R_{0})=\frac{1}{R_{0}}\left((\alpha-\delta)\delta C_{3}R_{0}^{\alpha-\delta}-\gamma C_{2} (\alpha-\gamma)R_{0}^{\alpha-\gamma}\right).\]
Thus, using that $\mathfrak{h}(s_{0},R_{0})=0$,
\[
\frac{\partial\mathfrak{h}}{\partial R}(s_{0},R_{0})=\frac{1}{R_{0}}\left((\gamma-\delta)\gamma C_{2}R_{0}^{\alpha-\gamma}+\alpha(\alpha-\delta)C_{1}\right)>0.
\]
Then the Implicit Function Theorem assures that the map $s\rightarrow R_{s}$ is continuous in a neighborhood of any point $s_{0}$ and the claim follows. This implies in particular that the map
$
s\rightarrow \mathcal{F}_{s}(\rho_{R_{s}})
$
is continuous at $s=0$ and we have
\[
\lim_{s\downarrow 0} \mathcal{F}_s[\rho_{R_s}]=\mathcal{F}_{0}(\rho_{R_0}),
\]
then the only issue is to compute the value $R_0$, which is obtained by equation \eqref{algebraic} letting $s\rightarrow0$. Such computation shows that the value $R_0$ is the one in \eqref{rad0}.
Then, inserting $\rho_{R_{0}}$ in the expression \eqref{limitfunct} of $\mathcal{F}_{0}$ we have
\[
\mathcal{F}_{0}[\rho_{R_0}]=-M\frac{m-2}{m-1}\left(\frac{\chi-2\beta}{2}\right)^{\frac{m-1}{m-2}},
\]
as desired.
\end{proof}
\begin{remark}\rm
Since the limit value as $s\rightarrow0$ of $\mathcal{F}_{s}$ in $\rho_{R}$ is given by
\[
\mathcal{F}_{0}[\rho_{R}]=\frac{(dM)^{m}}{d(m-1)}\sigma_{d}^{1-m}R^{-d(m-1)}+\frac{dM^{2}}{\sigma_{d}}\left(\beta-\frac{\chi}{2}\right)R^{-d}
\]
and $m>2$, we observe that $R\mapsto\mathcal{F}_{0}[\rho_{R}]$ does not admit a minimum for $\beta\geq \chi/2$.
\end{remark}

\noindent Next we investigate some asymptotic properties of minimizers as $s\downarrow 0$.

\begin{lem}\label{lemma:equibounded}  Fix any $s_{0}\in (0,1/2)$.
For any $s\in (0,s_0)$, let  $\rho_s\in\mathcal{Y}_M$ be the unique minimizer of $\mathcal{F}_s$ over $\mathcal Y_M$. Then $\sup_{s\in(0,s_0)}\|\rho_s\|_{L^\infty(\mathbb{R}^d)}<+\infty$.
\end{lem}
\begin{proof}
Since $\rho_s$ is continuous and radially decreasing by Lemma \ref{Minimiz} and Lemma \ref{regularitymin} , we have $\|\rho_s\|_{L^\infty(\mathbb{R}^d)}=\rho_s(0)$. By \eqref{Euler} and \eqref{explicitconstant} we have, letting $B_1$ denote the unit ball centered at the origin,
\[\begin{aligned}
\frac{m}{m-1}\rho_s(0)^{m-1}+2\beta \rho_{s}(0)&=\chi c_{d,s}(|\cdot|^{2s-d}\ast\rho_s)(0)-\mathcal{C}_s\le \chi c_{d,s}\int_{\mathbb{R}^d}|y|^{2s-d}\rho_s(y)\,dy\\
&\le
\chi c_{d,s}\rho_s(0)\int_{B_1}|y|^{2s-d}\,dy+\chi c_{d,s}\int_{\mathbb{R}^d\setminus B_1}\rho_s(y)\,dy\\
&\le \chi c_{d,s}\rho_s(0)\sigma_{d}\int_0^{1}r^{2s-1}\,dr+\chi c_{d,s} M\le \frac{\chi c_{d,s}\sigma_{d}}{2s}\,\rho_s(0)+\chi c_{d,s}M.
\end{aligned}
\]
Therefore $\rho_s(0)^{m-1}\le a\rho_s(0)+b$ for any $s\in(0,s_0)$, where $$a:=\frac{m-1}{m}\sup_{ s\in(0,s_0)} \,\frac{\chi c_{d,s}\sigma_{d}}{2s}<+\infty\quad \mbox{ and}\quad b:=\frac{m-1}{m}\sup_{ s\in(0,s_0)} \chi c_{d,s}M<+\infty.$$ Notice that $a<+\infty$ follows from $\lim_{s\downarrow 0}{c_{d,s}}/{s}=\pi^{-d/2}\Gamma(d/2)$, see \eqref{cds}. Since $m>2$, we conclude that $\rho_s(0)\le\bar x$ for any $s\in(0,s_0)$, where $\bar x$ is the unique positive number such that $\bar x^{m-1}=a\bar x+b$.
\end{proof}

\begin{lem}\label{lemma:equiboundedconstant} Let $0\le \beta<\chi/2$. 
For any $s\in (0,1/2)$, 
let $\rho_s\in\mathcal{Y}_M$ be the unique minimizer of $\mathcal{F}_s$ over $\mathcal Y_M$. Then
\begin{equation*}
\liminf_{s\downarrow 0}\mathcal{C}_s\geq \frac{m-2}{m-1}\left(\frac{\chi-2\beta}{2}\right)^{\frac{m-1}{m-2}}\label{lowboundC}
\end{equation*}
where $\mathcal{C}_s$ is defined in \eqref{explicitconstant1}.
\end{lem}
\begin{proof}
By \eqref{explicitconstant1} and \eqref{functminim} we obtain
\begin{align*}
\mathcal{C}_s&=\frac{1}{M(d-2s)}\left(\frac{dm-2d+2s}{m-1}\|\rho_{s}\|_{m}^{m}+4\beta s\|\rho_{s}\|_{2}^{2}+2s\|\rho_{s}\|_{m}^{m}\right)\\
&>\frac{1}{M(d-2s)}\left(\frac{dm-2d+2s}{m-1}\|\rho_{s}\|_{m}^{m}+2\beta s\|\rho_{s}\|_{2}^{2}\right)=-\frac{1}{M}\mathcal{F}_s(\rho_s)\
\end{align*}
and the result follows by the minimality of $\rho_{s}$ and Lemma \eqref{lemma:ball}.
\end{proof}

\begin{lem}\label{lemma:tightness} Let $0\le \beta<\chi/2$.
For any $s\in (0,1/2)$, let $\rho_s\in\mathcal{Y}_M$ be the unique minimizer of $\mathcal{F}_s$ over $\mathcal Y_M$. Then there exists $R\in(0,+\infty)$ and $s_{0}\in(0,1/2)$ such that $\mathrm{supp}(\rho_s)\subset B_R$ for any $s\in(0,s_0)$.
\end{lem}
\begin{proof}
We a slight abuse of notation we still denote by $\rho_s$ the radial profile of $\rho_s$ and we notice that since $\rho_s$ is radially non-increasing there holds for any $R>0$
\[
M\ge\int_{B_R}\rho_s(x)\,dx\ge\int_{B_R}\rho_s(R)\,dx=\frac{1}{d}\rho_s(R)\sigma_d R^d,
\]
which entails
\begin{equation}\label{xd}
\rho_s(x)\le\frac{dM}{\sigma_d |x|^d}\quad\mbox{for any $x\in\mathbb{R}^d\setminus\{0\}$.}
\end{equation}

Let $\mathrm{supp}(\rho_s)=:B_{R_s}$.
Form \eqref{Euler} we deduce
\begin{equation}\label{edge}
\chi c_{d,s}(|\cdot|^{2s-d}\ast\rho_s)(R_s)=\mathcal{C}_s\quad\mbox { for any  $s\in(0,s_0)$}.
\end{equation}
Notice that by \eqref{xd} we have if $R_s>1$
\[\begin{aligned}
&\chi c_{d,s} (|\cdot|^{2s-d}\ast\rho_s)(R_s)\le \chi c_{d,s}\int_{B_1} |y|^{2s-d}\rho_s(R_s-y)\,dy+\chi c_{d,s}\int_{B_1^C}|y|^{2s-d}\rho_s(R_s-y)\,dy
\\
&\qquad\le \chi c_{d,s}\frac{dM}{\sigma_{d}}\int_{B_1}\frac{|y|^{2s-d}}{|R_s-y|^d}dy+M\chi c_{d,s}\le \chi c_{d,s}\frac{dM}{2s|R_s-1|^d}+M\chi c_{d,s}.
\end{aligned}\]
Assuming by contradiction that $$\limsup_{s\downarrow 0}R_s=+\infty,$$ the above computation shows, recalling from \eqref{cds} that $c_{d,s}/s$ is bounded on $(0,s_0)$, that
\[
\liminf_{s\downarrow 0} \chi c_{d,s}(|\cdot|^{2s-d}\ast\rho_s)(R_s)=0.
\]
This contradicts \eqref{edge}, since $\liminf_{s\downarrow 0}\mathcal{C}_s>0$ by Lemma \ref{lemma:equiboundedconstant}.
\end{proof}

\begin{lem}\label{lemma:strongcompactness} Let $0\le \beta<\chi/2$.
For any $s\in (0,1/2)$, let $\rho_s\in\mathcal{Y}_M$ be the unique minimizer of $\mathcal{F}_s$ over $\mathcal Y_M$. For any vanishing sequence $(s_n)\subset(0,1/2)$, the sequence $(\rho_{s_n})$ admits limit points in the strong $L^p(\mathbb{R}^d)$ topology as $n\to+\infty$ for any $p\in[1,+\infty)$.
\end{lem}
\begin{proof} We have $\rho_s\in W^{1,1}(\mathbb R^d)$, by reasoning as done in the proof of Proposition \ref{id}.
We still denote by $\rho_s$ the radial profile of $\rho_s$ and we notice that $\nabla \rho_s(0)=0$ and $\nabla\rho_s(x)=\rho'_s(|x|)\,\tfrac{x}{|x|}$ for $x\neq0$.
Let $R>1$ and $s_{0}\in (0,1/2)$ such that $\mathrm{supp}(\rho_s)\subset B_R$ for any $s\in(0,s_0)$. The existence of such $R$, $s_{0}$ is due to Lemma \ref{lemma:tightness}. We have
\[
\int_{\mathbb{R}^d}|\nabla\rho_s(x)|\,dx=\int_{\mathbb{R}^d}|\rho_s'(|x|)|\,dx=-\sigma_{d}\int_0^R\rho_s'(r)r^{d-1}\,dr
\]
If $d=1$ we have $\sigma_{d}=2$ and $\int_0^R\rho'(r)\,dr=-\rho_s(0)$. If $d\ge 2 $ we have $$\rho_s'(r)r^{d-1}=(\rho_s(r)r^{d-1})'-(d-1)\rho_s(r)r^{d-2},$$ thus
\[\begin{aligned}
-\int_0^R\rho_s'(r) r^{d-1}\,dr&=-\int_0^R (\rho_s(r)r^{d-1})'\,dr+(d-1)\int_0^R \rho_s(r)r^{d-2}\,dr= (d-1)\int_0^R \rho_s(r)r^{d-2}\,dr\\
&\le (d-1)\left(\int_0^1\rho_s(r)\,dr+\int_1^R\rho_s(r)\,r^{d-1}\,dr\right)\le (d-1)(\rho_s(0)+M).
\end{aligned}\]
Therefore we always have
\[
\sup_{s\in(0,s_0)}\int_{\mathbb{R}^d}|\nabla\rho_s|\le d\sigma_{d}\sup_{s\in(0,s_0)}(\rho_s(0)+M)<+\infty,
\]
where the finiteness is due to Lemma \ref{lemma:equibounded}. The above uniform $BV(\mathbb{R}^d)$ estimate and the usual compact embedding $BV(\R^{d})$ in $L^{1}_{loc}(\R^{d})$, entails the strong sequential $L^{1}_{loc}(\mathbb{R}^d)$ compactness of the family $(\rho_s)$, which can be extended to the whole $L^{1}(\R^{d})$ by the tightness  due to Lemma \ref{lemma:tightness}. If $\rho$ is a limit point along a vanishing sequence $s_n$, we also have $\rho\in L^\infty(\mathbb{R}^d)$ and $\rho_{s_n}\to\rho$ strongly in $L^p(\mathbb{R}^d)$ for any $p\in[1,+\infty)$, since the sequence $(\rho_{s_n})$ is also equi-bounded by Lemma \ref{lemma:equibounded}.
\end{proof}

\begin{lem}\label{lemma:limitfunctional}  Suppose that $\rho_s\in\mathcal{Y}_M$ for any $s>0$ and that $\rho\in\mathcal{Y}_M$. If $\rho_s\to\rho$ strongly in $L^2(\mathbb{R}^d)$ as $s\downarrow 0$, then
\[
\lim_{s\downarrow 0}\int_{\mathbb{R}^{2d}}c_{d,s}|x-y|^{2s-d}\rho_s(x)\rho_s(y)\,dx\,dy=\int_{\mathbb{R}^d}\rho^2(x)\,dx.
\]
\end{lem}
\begin{proof}
By Plancherel theorem we have
\begin{equation}\label{plancherel}\begin{aligned}
&\left|\int_{\mathbb{R}^{2d}}c_{d,s}|x-y|^{2s-d}\rho_s(x)\rho_s(y)\,dx\,dy-\int_{\mathbb{R}^d}\rho^2(x)\,dx\right|\\&\qquad=\frac{1}{(2\pi)^d}\left|\int_{\mathbb{R}^d}|\xi|^{-2s}|\hat\rho_s(\xi)|^2\,d\xi-\int_{\mathbb{R}^d}|\hat\rho(\xi)|^2\,d\xi\right|\\
&\qquad \le\frac{1}{(2\pi)^d}\int_{\mathbb{R}^d}||\xi|^{-2s}-1|\,|\hat\rho_s(\xi)|^2\,d\xi+\frac{1}{(2\pi)^d}\left|\int_{\mathbb{R}^d}|\hat\rho_s(\xi)|^2\,d\xi-\int_{\mathbb{R}^d}|\hat \rho(\xi)|^2\,d\xi\right|.\end{aligned}
\end{equation}
About the first term in the right hand side, for any $R>1$ we have, since $|\hat \rho_s(\xi)|\le M$ for any $\xi\in\mathbb{R}^d$,
\begin{equation*}\begin{aligned}
\int_{\mathbb{R}}||\xi|^{-2s}-1|\,|\hat\rho_s(\xi)|^2\,d\xi&= \int_{B_R}||\xi|^{-2s}-1||\hat\rho_s(\xi)|^2\,d\xi+\int_{B_R^C}||\xi|^{-2s}-1||\hat\rho_s(\xi)|^2\,d\xi\\
&\le M^2\int_{B_R}||\xi|^{-2s}-1|\,d\xi+2\int_{B_R^C}|\hat\rho_s(\xi)|^2\,d\xi\\\
&\le M^2\int_{B_R}||\xi|^{-2s}-1|\,d\xi+4\int_{\mathbb{R}^d}|\hat\rho_s(\xi)-\hat\rho(\xi)|^2\,d\xi+4\int_{B_R^C}|\hat\rho(\xi)|^2\,d\xi.
\end{aligned}\end{equation*}
As $s\downarrow 0$, the first term in the right hand side goes to zero by dominated convergence (as dominating function we take $|\xi|^{-2s_0}+1$ for $|\xi|\le 1$ and $2$ for $1<|\xi|\le R$). Therefore, by the strong $L^2(\mathbb{R}^d)$ convergence of $\rho_s$ to $\rho$, from \eqref{plancherel} we get
\[
\limsup_{s\downarrow 0} \left|\int_{\mathbb{R}^d}c_{d,s}|x-y|^{2s-d}\rho_s(x)\rho_s(y)\,dx\,dy-\int_{\mathbb{R}^d}\rho^2(x)\,dx\right|\le
4\int_{B_R^C}|\hat\rho(\xi)|^2\,d\xi.
\]
The result follows, since $\rho\in L^2(\mathbb{R}^d)$ and $R$ is arbitrary.
\end{proof}

\noindent Now we are in the position to state the main result about convergence of minimizers $\rho_s$ towards $\rho_0$, where $\rho_0$ is defined  in \eqref{rad0}.

\begin{thm}\label{theorem:Gammaconvergence} Assume $0\le\beta<\chi/2$.
 For any $s\in (0,1/2)$,  let $\rho_s\in\mathcal{Y}_M$ be the unique minimizer of $\mathcal{F}_s$ over $\mathcal Y_M$.
Then, there exists $\rho\in\mathcal{Y}_M$ such that $\rho_s\to \rho$ strongly in $L^m(\mathbb{R}^d)$ as $s\downarrow 0$. Moreover,
$\rho$ is the unique radially decreasing minimizer
of the functional \eqref{limitfunct} over $\mathcal{Y}_M$, given by \eqref{rad0}.
 \end{thm}
\begin{proof}
Let $(s_n)\subset(0,1/2)$ be a vanishing sequence. Let $\rho\in\mathcal{Y}_M$ be such that $\rho_{s_n}\to\rho$ strongly in $L^m(\mathbb{R}^d)$ as $n\to+\infty$. The existence of such a limit point follows from Lemma \ref{lemma:strongcompactness}, and the convergence holds in $L^p(\mathbb{R}^d)$ for any $p\in[1,+\infty)$. Given $\tilde\rho\in\mathcal{Y}_M$, by the strong $L^m(\mathbb{R}^d)$ convergence and by Lemma \ref{lemma:limitfunctional} we have
\[
\mathcal{F}_0[\rho]=\lim_{n\to+\infty}\mathcal{F}_{s_n}[\rho_{s_n}]\le \lim_{n\to+\infty} \mathcal{F}_{s_n}[\tilde\rho]=\mathcal{F}_0[\tilde\rho].
\]
By the arbitrariness of $\tilde\rho$, we conclude that $\rho$ is a minimizer of $\mathcal{F}_0$ over $\mathcal{Y}_M$.  Finally, the whole family $(\rho_s)$ converges to $\rho$ in $L^m(\mathbb{R}^d)$ as $s\downarrow 0$.
 \end{proof}

\begin{remark}\rm
For $0\le \beta<\chi/2$, and given $M>0$,
we have in fact the $\Gamma$-convergence of functionals $\mathcal F_s:\mathcal Y_M\to\mathbb R$ to functional $\mathcal F_0:\mathcal Y_M\to\mathbb R^d$ as $s\to 0$, with respect to the strong $L^2(\mathbb R^d)$ topology. Indeed, if $(\zeta_s)_{s\in(0,1/2)}\subset\mathcal Y_M$, $\zeta\in \mathcal Y_M$ and $\zeta_s\to\zeta$  in $L^2(\mathbb R^d)$ as $s\to 0$, by Fatou's lemma and Lemma \ref{lemma:limitfunctional} we get
$\mathcal F_0[\zeta]\le \liminf_{s\to 0} \mathcal F_s[\zeta_s].$
On the other hand, still by Lemma \ref{lemma:limitfunctional}, for every $\zeta\in \mathcal Y_M$ we have $\mathcal F_s[\zeta]\to\mathcal F_0[\zeta]$ as $s\to0$.
\end{remark}

%

\begin{proofad3}
The case $0\le \beta<\chi/2$ has been treated in Theorem \ref{theorem:Gammaconvergence}. Therefore, in order to conclude we consider the case $\beta\ge\chi/2$. 

For every $\rho\in\mathcal Y_M$ and every $\lambda>0$, let $\rho_\lambda(x)=\lambda^d\rho(\lambda x)$, $x\in\mathbb R^d$. We have
\begin{equation}\label{fs1}
 \mathcal F_s[\rho_\lambda]\ge \beta \lambda^d\int_{\mathbb R^d}\rho^2(x)\,dx+\lambda^{d-2s}\mathcal W_s[\rho].
\end{equation}
Similarly to the proof of Proposition \ref{minF0}, we minimize the right hand side with respect to $\lambda$ and find a unique optimal value $\lambda_*$ given by
\[
\lambda_* =\left({(2s-d) \mathcal W_s [\rho]}\right)^{\frac 1{2s}}\left(d\beta\int_{\mathbb R^d}\rho^2(x)\,dx\right)^{-\frac{1}{2s}},
\]
and by inserting such a value in \eqref{fs1} we get
\[
\mathcal F_s[\rho]\ge -\frac{2s\beta}{d-2s}\left(\frac{\chi(d-2s)}{2d\beta}\right)^{\frac d{2s}}\,\left(\int_{\mathbb R^d}\rho
^2(x)\,dx\right)^{1-\frac{d}{2s}}\left(\int_{\mathbb R^d}\int_{\mathbb R^d}K_s(x-y)\rho(x)\rho(y)\,dx\,dy\right)^{\frac d{2s}}.
\]
 But  the Hardy-Littlewood-Sobolev inequality  \eqref{HLS} along with interpolation of $L^p$ norms, entails
 \[
 \int_{\mathbb R^d}\int_{\mathbb R^d}K_s(x-y)\rho(x)\rho(y)\,dx\,dy\le c_{d,s}H_{d,s}\|\rho\|_{\frac{2d}{d+2s}}^2\le  c_{d,s}\,H_{d,s}\,M^{\frac{4s}{d}}\left(\int_{\mathbb R^d}\rho^2(x)\,dx\right)^{1-\frac{2s}{d}},
 \]
 therefore we get the estimate
 \[
 \mathcal F_s[\rho]\ge-\frac{2s\beta}{d-2s}\left(\frac{\chi(d-2s)}{2d\beta}\right)^{\frac d{2s}}\, M^2\,\left(c_{d,s}\,H_{d,s}\right)^{\frac{d}{2s}}
 \]
 for every $\rho\in\mathcal Y_M$.
 It is not difficult to check that $\left(c_{d,s}{H}_{d,s}\right)^{\frac{d}{2s}}$ converges to a finite limit as $s\downarrow 0$, therefore the above right hand side is negative and converges to $0$ as $s\downarrow 0$ due to the condition $\beta\ge \chi/2$. If $\rho_s$ denotes the unique minimizer of $\mathcal F_s$ over $\mathcal Y_M$, we deduce from \eqref{functminim} that $\lim_{s\downarrow 0}\mathcal F_s[\rho_s]=0$. This implies from  \eqref{explicitconstant} and \eqref{functminim} that $\rho_s\to 0$ in $L
 ^m(\mathbb R^d)$ and $\mathcal C_s\to 0$ as $s\downarrow 0$.

 Eventually, we prove that $\|\rho_s\|_{\infty}\to 0$ as $s\downarrow 0$. Since $\rho_s$ is continuous and radially decreasing we have $\|\rho_s\|_\infty=\rho_s(0)$, and as seen in the proof of Lemma \ref{lemma:equibounded} we may take advantage of \eqref{Euler} and get
 \begin{equation}\label{nocs}\begin{aligned}
\frac{m}{m-1}\rho_s(0)^{m-1}+2\beta \rho_{s}(0)&=\chi c_{d,s}(|\cdot|^{2s-d}\ast\rho_s)(0)-\mathcal{C}_s\le \frac{\chi c_{d,s}\sigma_{d}}{2s}\rho_s(0)+\chi c_{d,s}M.
\end{aligned}
\end{equation}
Let $\bar\rho:=\limsup_{s\downarrow 0}\rho_s(0)$. By taking \eqref{cds} and $\sigma_d=2 \pi^{d/2}/\Gamma(d/2)$ into account, we have $\lim_{s\downarrow 0} (2s)^{-1}c_{d,s}\sigma_d=1$, thus from \eqref{nocs} we deduce
\[
\frac{m}{m-1}\bar\rho^{m-1}+2\beta \bar\rho\le \chi\bar \rho.
\]
Since $m>2$ and $\beta\ge\chi/2$, this forces $\bar\rho=0$ which is the desired result.
\end{proofad3}


\section{Weak solutions for the aggregation-diffusion problem}
\label{evolutionsection}
 The first objective of this section is the proof of the main existence result stated in Theorem \eqref{existencethm}. 
We mention that
an alternative existence proof for $s>1/2$ and $\beta=0$ is found in \cite{Z}.

We fix $M>0$. For $\rho^{0}\in\mathcal Y_{M,2}$, and we  consider the  Cauchy problem  \eqref{cauchy1}
We shall construct a  weak solution to problem \eqref{cauchy1} by an application of the JKO scheme to the functional $\mathcal F_s$ defined by \eqref{functional}. Therefore, for a discrete time step  $\tau>0$, we consider the minimization problem
\begin{equation}\label{MM}
\min_{\rho\in \mathcal Y_{M,2}}\left(\mathcal F_s[\rho]+\frac1{2\tau}\,W_2^2(\rho,\rho^0)\right).
\end{equation}

\begin{prop}[\bf Existence of discrete minimizers]\label{existencediscrete} Let $\tau>0$ and $\rho^0\in\mathcal Y_{M,2}$.
The minimization problem \eqref{MM} admits solutions.
\end{prop}
\begin{proof}
It is clear from Lemma \ref{Minimiz} that the functional to be minimized over $\mathcal Y_{M,2}$ is bounded from below. Let $(\rho_n)\subset\mathcal Y_{M,2}$ be one of its minimizing sequences.
The sequence $(\rho_{n})$ has uniformly bounded $L^m(\mathbb R^d)$ norm by inequality \eqref{boundLm}.
 It also has uniformly bounded second moment, thanks to the uniform bound for $W_2(\rho_n,\rho^0)$,
 which follows from the fact that $\rho_n$ is a minimizing sequence and again from \eqref{boundLm} and \eqref{distantref}.
 Hence, up to subsequences, it converges to $\rho_*$ weakly in $L^p(\mathbb R^d)$ for every $p\in[1,m]$ and narrowly by Prokhorov's theorem (see \emph{e.g.} \cite[Theorem 5.1.3]{AGS}). This implies that $\rho_*$ has mass $M$ and by \cite[Lemma 5.17]{AGS} that $\rho_*$ has 0 center of mass, therefore $\rho_*\in \mathcal{Y}_{M}$. Furthermore, by the lower semicontinuity of the second moments with respect to the narrow topology we have  $\rho_*\in \mathcal{Y}_{M_{2}}$. We notice that also the sequence of product measures $\rho_n(x)\,dx\,\rho_n(y)\,dy$ is narrowly converging to $\rho_*(x)\,dx\,\rho_*(y)\,dy$,
see for instance \cite[Theorem 2.8]{B}.

Let $\eps>0$ and let $\eta_\eps:\mathbb R^d\to[0,1]$ be a smooth cutoff function such that $\eta_\eps(x)=1$ if $|x|<\eps$ and $\eta_\eps(x)=0$ if $|x|>2\eps$. As $K_s\,(1-\eta_\eps)$ is a bounded continuous function over $\mathbb R^d$ we obtain
\[\begin{aligned}
&\lim_{n\to+\infty}\int_{\mathbb R^d}\int_{\mathbb R^d}K_s(x-y)\,(1-\eta_\eps(x-y))\,\rho_n(x)\,\rho_n(y)\,dx\,dy\\&\qquad=\int_{\mathbb R^d}\int_{\mathbb R^d}K_s(x-y)\,(1-\eta_\eps(x-y))\,\rho_{\ast}(x)\,\rho_{\ast}(y)\,dx\,dy
\end{aligned}\]
for every $\eps>0$.
On the other hand, by Cauchy-Schwarz and Young inequality we have
\begin{equation*}\begin{aligned}
\int_{\mathbb R^d}\int_{\mathbb R^d}K_s(x-y)\,\eta_\eps(x-y)\,\rho_n(x)\,\rho_n(y)\,dx\,dy
&\le \|\rho_n\|_{L^2(\mathbb R^d)}\|(\eta_\eps\,K_s)\ast\rho_n\|_{L^2(\mathbb R^d)}\\&\le \|\rho_n\|_{L^2(\mathbb R^d)}^2\|\eta_\eps\,K_s\|_{ L^1(\mathbb R^d)}\le C \|\eta_\eps\,K_s\|_{ L^1(\mathbb R^d)}
\end{aligned}
\end{equation*}
for every $n\in\mathbb N$ and every $\eps>0$, where $C$ is a constant that does not depend on $n$, since $(\rho_n)$ is a bounded sequence in $L^2(\mathbb R^d)$.
A combination of the two above relations yields
\[\begin{aligned}
&\limsup_{n\to+\infty} \left|\int_{\mathbb R^d}\int_{\mathbb R^d}K_s(x-y)\,\rho_n(x)\,\rho_n(y)\,dx\,dy-\int_{\mathbb R^d}\int_{\mathbb R^d}K_s(x-y)\,\rho_{\ast}(x)\,\rho_{\ast}(y)\,dx\,dy\right|\\&\qquad\qquad\qquad\le
\int_{\mathbb R^d}\int_{\mathbb R^d}K_s(x-y)\eta_\eps(x-y)\,\rho_{\ast}(x)\,\rho_{\ast}(y)\,dx\,dy + C \|\eta_\eps\,K_s\|_{ L^1(\mathbb R^d)}.
\end{aligned}\]
for every $\eps>0$. By dominated convergence, the two terms in the right hand side vanish as $\eps\to 0$, so that the arbitrariness of $\eps$ entails
\begin{equation}\label{limith-s}\begin{aligned}
\lim_{n\to+\infty} \int_{\mathbb R^d}\int_{\mathbb R^d}K_s(x-y)\,\rho_n(x)\,\rho_n(y)\,dx\,dy&=
\int_{\mathbb R^d}\int_{\mathbb R^d}K_s(x-y)\,\rho_{\ast}(x)\,\rho_{\ast}(y)\,dx\,dy. \end{aligned}\end{equation}
By the weak  lower semicontinuity of the $L^m(\mathbb R^d)$ and of the $L^2(\mathbb R^d)$ norms, by the narrow lower semicontinuity of $W_2(\cdot,\rho^0)$ (see \cite[Proposition 7.1.3]{AGS}) and thanks to \eqref{limith-s} we conclude that
\[
\mathcal F_s[\rho_*]+\frac1{2\tau}W_2^2(\rho_*,\rho
^0)\le\liminf_{n\to+\infty}\left(\mathcal F_s[\rho_n]+\frac1{2\tau} W_2^2(\rho_n,\rho^0)\right).
\]
Since $(\rho_n)$ is a minimizing sequence, we conclude that $\rho_*$ is a solution to problem \eqref{MM}.
\end{proof}

Once existence of a discrete solution is established, we perform a recursive minimization and apply  standard arguments from the theory of minimizing movements to obtain convergence of the scheme and existence of a limit curve, as summarized in the next two statements.

\begin{prop}[\bf Basic estimate of minimizing movements]\label{base}
Let $\tau>0$ and $\rho^0\in\mathcal Y_{M,2}$.  We let $\rho_\tau^0:=\rho^0$ and for every $k\in\mathbb N$, we take recursively
\begin{equation}\label{recursive}\rho_\tau^k\in\mathrm{argmin}_{\mathcal Y_{M,2}}\left(\mathcal F_s[\cdot]+\frac1{2\tau}W_2^2(\cdot,\rho_\tau^{k-1})\right),\end{equation}
thus defining a sequence $(\rho_\tau^k)_{k\ge 1}$ of discrete minimizers, whose existence is ensured by {\rm Proposition \ref{existencediscrete}}.
For every $k\in\mathbb N$ there hold
\begin{equation}\label{basic1}
\mathcal F_s[\rho_\tau^k]+\frac1{2\tau}\sum_{h=1}^k W_2^2(\rho_\tau^h,\rho_\tau^{h-1})\le \mathcal F_s[\rho^0]\le \frac1{m-1}\int_{\mathbb R^d}(\rho^0(x))^m\,dx+\beta\int_{\mathbb R^d}(\rho^0(x))^2\,dx,
\end{equation}
\begin{equation}\label{basic2}
\frac{1}{2(m-1)}\int_{\mathbb R^d}(\rho_\tau^k(x))^m\,dx\le \frac1{m-1}\int_{\mathbb R^d}(\rho^0(x))^m\,dx+\beta\int_{\mathbb R^d}(\rho^0(x))^2\,dx+\bar C,
\end{equation}
and
\begin{equation}\label{basic3}
\int_{\mathbb R^d}|x|^2(\rho_\tau^k(x))\,dx\le 4k\tau \mathcal F_s[\rho^0]+2\int_{\mathbb R^d}|x|^2\rho^0(x)\,dx,
\end{equation}
where $\bar C=\bar C(\chi, m,s,d,M)$ is the constant given by \eqref{cbar}, only depending on $\chi,m,s,d,M$.
\end{prop}
\begin{proof}
Estimate \eqref{basic1} directly follows from the minimality property of $\rho_\tau^k$, which is defined by \eqref{recursive}. Estimate \eqref{basic2} follows from \eqref{boundLm} and \eqref{basic1}. Moreover, we have by triangle inequality and
  by Cauchy Schwarz inequality
\[\begin{aligned}
\int_{\mathbb R^d}|x|^2\rho_\tau^k(x)\,dx&=W_2^2(\rho_\tau^k,M\delta_0)\le\left(\sum_{h=1}^k W_2(\rho_\tau^h,\rho_\tau^{h-1})+W_2(\rho^0,M\delta_0))\right)^2\\&\le 2k\tau\sum_{h=1}^k\frac1\tau W_2^2(\rho_\tau^h,\rho_\tau^{h-1})+2W_2(\rho^0,M\delta_0),
\end{aligned}\]
which entails \eqref{basic3} by means of  \eqref{basic1}.
\end{proof}

\begin{prop}[\bf Convergence of the scheme]\label{convMM}
Let $\rho^0\in\mathcal Y_{M,2}$. For every $\tau>0$,
let us consider a sequence $(\rho_\tau^k)_{k\ge 0}$ of discrete minimizers defined by \eqref{recursive} and define the piecewise constant interpolation \begin{equation}\label{morceaux}\rho_\tau(t,\cdot):=\rho_\tau^{\lceil t/\tau\rceil}(\cdot),\qquad t\ge 0,\end{equation}
where $\lceil x\rceil:=\min\{h\in\mathbb N:h\ge x\}$. Then there exist a vanishing sequence $(\tau_n)_{n\in\mathbb N}\subset(0,1)$ and a limit function $\rho\in L^\infty((0,+\infty);L^m(\R^d))$ such that $[0,+\infty)\ni t\mapsto \rho(t,\cdot)\in \mathcal Y_{M,2}$ is narrowly continuous and such that $\rho_{\tau_n}(t,\cdot)$ narrowly converge to $\rho(t,\cdot)$ for every $t\ge 0$. Furthermore, $t\mapsto \rho(t,\cdot)$ is a $AC^2([0,+\infty))$ curve with respect to the Wasserstein distance, i.e.,  there exists $g\in L^2(0,+\infty)$ such that $W_2(\rho(t_1,\cdot),\rho(t_2,\cdot))\le \int_{t_1}^{t_2}g(r)\,dr$ for every $0\le t_1<t_2<+\infty$.
\end{prop}
\begin{proof} The convergence to a narrowly continuous curve along a vanishing sequence $(\tau_n)$ follows
 from the standard convergence arguments for minimizing movements from \cite{AGS}. In order to obtain it, we let  $g_\tau:[0,+\infty)\to[0,+\infty)$ be defined as $g_\tau(t)=\tau^{-1}\,W_2(\rho_\tau(t,\cdot),\rho_\tau(t-\tau,\cdot))$, with the convention $\rho_\tau(r,\cdot)=\rho^0(\cdot)$ if $r<0$, and we have from \eqref{morceaux}
 \[\frac12\int_0^{+\infty} g^2_\tau(t)\,dt=\frac1{2\tau^2}\sum_{h=1}^{+\infty}\int_{(h-1)\tau}^{h\tau}W_2^2(\rho_\tau(t,\cdot),\rho_\tau(t-\tau,\cdot))\,dt
 =\frac1{2\tau}\sum_{h=1}^{+\infty}W_2^2(\rho_\tau^h,\rho_\tau^{h-1}).
 \]
 On the other hand, \eqref{basic1}, \eqref{basic2} and \eqref{distantref} entail for every integer $k\ge 1$
 \[
\frac1{2\tau}\sum_{h=1}^k W_2^2(\rho_\tau^h,\rho_\tau^{h-1})\le \mathcal F_s[\rho^0]-\mathcal F_s[\rho_\tau^k]\le  2\left(\bar C +\frac1{m-1}\int_{\R^d}(\rho^0)^m+\beta\int_{\R^d}(\rho^0)^2\right),
 \]
 where $\bar C$ is defined by \eqref{cbar} and \eqref{sds}. We deduce that
  \[\frac12\int_0^{+\infty} g^2_\tau(t)\,dt\le 2\left(\bar C+ \frac1{m-1}\int_{\R^d}(\rho^0)^m+\beta\int_{\R^d}(\rho^0)^2\right)
 \]
 so that $g_\tau\in L^2(0,+\infty)$ there exists a vanishing sequence $(\tau_n)_{n\in\mathbb N}\subset(0,1)$ such that $g_{\tau_n}\rightharpoonup g$ weakly in $L^2((0,+\infty))$ for some $g\in L^2((0,+\infty))$.
 For arbitrary $T>0$, the family of functions $\{\rho_{\tau_n}(t,\cdot): n\in\mathbb N,\,t\in[0,T]\}\subset \mathcal Y_{M,2}$ has uniformly bounded second moments thanks to \eqref{basic3}, hence it is narrowly relatively compact,
  and moreover  if $0\le t_1<t_2\le T$ we may  apply the triangle inequality  to find the Wasserstein equi-continuity estimate
 \[\begin{aligned}
 \limsup_{n\to+\infty} W_2(\rho_{\tau_n}(t_2,\cdot),\rho_{\tau_n}(t_1,\cdot))&\le\limsup_{n\to+\infty}\sum_{h=\lceil t_1/\tau_n\rceil+1}^{\lceil t_2/\tau_n\rceil} W_2(\rho_{\tau_n}^h,\rho_{\tau_n}^{h-1})\\
 &= \limsup_{n\to+\infty}\int_{(\lceil t_1/\tau_n\rceil+1)\tau_n}^{\lceil t_2/\tau_n\rceil\tau_n} g_{\tau_n}(t)\,dt=\int_{t_1}^{t_2}g(t)\,dt,
 \end{aligned}\]
 so that we can apply the abstract Ascoli-Arzel\`a theorem from \cite[Proposition 3.3.1]{AGS} and deduce that there exists a narrowly continuous curve $[0,T]\ni t\mapsto\rho(t,\cdot)$ such that, up to extraction of a not relabeled subsequence, $\rho_{\tau_n}(t,\cdot)\to\rho(t,\cdot)$ narrowly as $n\to+\infty$ for every $t\in[0,T]$.
 Eventually, by the above estimate and the narrow lower semicontinuity of the Wasserstein distance (see \cite[Proposition 7.1.3]{AGS}), we deduce $W_2(\rho(t_2,\cdot),\rho(t_1,\cdot))\le \int_{t_1}^{t_2}g(t)\,dt$,
 which is the desired $AC^2$ property. By the uniform estimate \eqref{basic2} and the narrow lower semicontinuity of the $L^m$ norm (see for instance \cite[Proposition 7.7]{Santambrogio} we also deduce that $\|\rho(t,\cdot)\|^m_{L^m(\R^d)}\le 2C(m-1)$ for every $t\ge 0$, where $C$ is the right hand side of \eqref{basic2}.
\end{proof}

The curve $t\mapsto \rho(t,\cdot)$ that was obtained in Proposition \ref{convMM} will be shown to be a weak solution to \eqref{cauchy1}. The first step towards this goal is to obtain a first order optimality condition for discrete minimizers of the JKO scheme.

\begin{prop}[\bf Euler-Lagrange equation for discrete minimizers] Let $\tau>0$ and $\rho^0\in\mathcal Y_{M,2}$. If $\rho_*$ is a solution to problem \eqref{MM}, then there holds
\begin{equation}\label{EL}\begin{aligned}&\frac1\tau\int_{\mathbb R^d}(T_*(x)-x)\cdot\nabla\zeta(x)\,\rho_*(x)\,dx=
-\int_{\mathbb R^d}\Delta\zeta(x)\,(\rho_*(x)^m+\beta\rho_*(x)^2)\,dx\\&\qquad+\frac{(d-2s)\chi\,c_{d,s}}2\int_{\mathbb R^d}\int_{\mathbb R^d}(\nabla\zeta(x)-\nabla\zeta(y))\cdot(x-y)|x-y|^{2s-d-2}\,\rho_*(x)\,\rho_*(y)\,dx\,dy\end{aligned}\end{equation}
for every $\zeta\in C^\infty(\mathbb R^d)$ such that $\nabla\zeta\in W^{1,\infty}(\R^{d};\R^d)$. Here, $T_{*}$ is the unique optimal transport map (for the quadratic cost) from $\rho_*(x)\,dx$ to $\rho^0(x)\,dx$. \end{prop}
\begin{proof}
Let  $\zeta\in C^\infty(\mathbb R^d)$ be such that $\nabla\zeta\in W^{1,\infty}(\R^{d};\R^d)$ and such that
\begin{equation}
\int_{\R^{d}}\rho_{\ast}\nabla\zeta\,dx=0.\label{compatibility}
\end{equation}
Let $\rho_\eps(x)\,dx$ be the push-forward measure of $\rho_*(x)\,dx$ through the map $x\mapsto x+\eps\nabla\zeta$, defined for any $\eps\in\mathbb R$
such that $|\eps|<1/L_\zeta$, where $L_\zeta$ the Lipschitz constant of $\nabla\zeta$. 
It is clear that $\rho_\eps\in \mathcal Y_M$.
By the change of variables formula and by the Taylor expansion $$\det(I+\eps\nabla^2\zeta)=1+\eps\Delta \zeta+o(\eps)$$ we get
\[\begin{aligned}
\int_{\mathbb R^d}\rho_\eps(x)^
m\,dx&=\int_{\mathbb R^d}\frac{\rho_*(x)^m}{(\det(I+\eps\nabla^2\zeta(x)))^{m-1}}\,dx\\&=\int_{\mathbb R^d}\rho_*(x)^m\,dx
-(m-1)\eps\int_{\mathbb R^d}\Delta\zeta(x)\,\rho_*(x)^m\,dx +o(\eps)
\end{aligned}\]
where $I$ is the identity matrix and $\nabla^2$ is the Hessian operator. Therefore we have
\[
\frac{d}{d\eps}\left(\frac{1}{m-1}\int_{\mathbb R^d}\rho_\eps(x)^m\,dx
\right)\Bigg{|}_{\eps=0}=
-\int_{\mathbb R^d}\Delta\zeta(x)\,\rho_*(x)^m\,dx,
\]
which is of course still true if $m$ is replaced by $2$.
On the other hand, the definition of push-forward entails
\[
\int_{\mathbb R^d}\int_{\mathbb R^d}|x-y|^{2s-d}\,\rho_\eps(x)\,\rho_\eps(y)\,dx\,dy=\int_{\mathbb R^d}\int_{\mathbb R^d}|x-y+\eps(\nabla\zeta(x)-\nabla \zeta(y))|^{2s-d}\rho_*(x)\,\rho_*(y)\,dx\,dy
\]
so that
\begin{equation*}\label{diffh-s}\begin{aligned}
&\frac1\eps\int_{\mathbb R^d}\int_{\mathbb R^d}|x-y|^{2s-d}\,\rho_\eps(x)\,\rho_\eps(y)\,dx\,dy-\frac1\eps\int_{\mathbb R^d}\int_{\mathbb R^d}|x-y|^{2s-d}\rho_*(x)\,\rho_*(y)\,dx\,dy\\
&\qquad=\frac1\eps
\int_{\mathbb R^d}\int_{\mathbb R^d}\left(|x-y+\eps(\nabla\zeta(x)-\nabla\zeta(y))|^{2s-d}-|x-y|^{2s-d}\right)\,\rho_*(x)\,\rho_*(y)\,dx\,dy.
\end{aligned}\end{equation*}
It is clear that
\[
\begin{aligned}
&\lim_{\eps\to 0}\frac1\eps\left(|x-y+\eps(\nabla\zeta(x)-\nabla \zeta(y))|^{2s-d}-|x-y|^{2s-d}\right)\\&\qquad\qquad=(2s-d)|x-y|^{2s-d-2}(x-y)\cdot(\nabla\zeta(x)-\nabla\zeta(y))\end{aligned}\]
for every $x,y\in\mathbb R^d$, $x\neq y$.
 On the other hand, 
 since $|\eps|<1/L_\zeta$ we have  $|\eps||\nabla\zeta(x)-\nabla\zeta(y)|<|x-y|$ and then we can obtain the estimate
 \[
 \left||x-y+\eps(\nabla\zeta(x)-\nabla\zeta(y))|^{2s-d}-|x-y|^{2s-d}\right|\le \left((1-|\eps|L_{\zeta})^{2s-d}-1\right)|x-y|^{2s-d}
 \]
  for every $x,y\in\mathbb R^d$, $x\neq y$. Thus, Bernoulli inequality entails, for every $x,y\in\mathbb R^d$, $x\neq y$ and every $\eps\in\mathbb R$ such that $2|\eps|<1/L_\zeta$,
 \[ 
  \left||x-y+\eps(\nabla\zeta(x)-\nabla\zeta(y))|^{2s-d}-|x-y|^{2s-d}\right|<|\eps|\,d\,L_\zeta\, 2^{d-2s}\,|x-y|^{2s-d},
 \] 
 therefore by dominated convergence we get
\[\begin{aligned}
&\frac{d}{d\eps}\left(\int_{\mathbb R^d}\int_{\mathbb R^d}|x-y|^{2s-d}\,\rho_\eps(x)\,\rho_\eps(y)\,dx\,dy\right)\Bigg{|}_{\eps=0}\\&\qquad\qquad=
(2s-d)\int_{\mathbb R^d}\int_{\mathbb R^d}|x-y|^{2s-d-2}(x-y)\cdot(\nabla\zeta(x)-\nabla\zeta(y))\,\rho_*(x)\,\rho_*(y)\,dx\,dy.
\end{aligned}\]

For the derivative of the Wasserstein distance, by a standard result (see \cite[Theorem 8.13]{Villani03}) we get
\[\frac{d}{d\eps}\left(\frac1{2\tau}W_2^2(\rho_\eps,\rho^0)\right)\Bigg{|}_{\eps=0}=\frac1\tau\int_{\mathbb R^d}(x-T_*(x))\cdot\nabla\zeta(x)\,\rho_*(x)\,dx
\]

Since $\rho_*$ is a minimizer, the derivative with respect to $\eps$ of $\mathcal F_s[\rho_\eps]+\tfrac1{2\tau}W_2^2(\rho_\eps,\rho^0)$ needs to vanish at $\eps=0$. We obtain the result.

 If we wish to remove the compatibility condition \eqref{compatibility}, we just replace $\zeta$ with 
\[
\tilde\zeta(x)=\zeta(x)-\frac{1}{M}\left(\int_{\R^{d}}\rho_{\ast}\nabla\zeta\,dx\right) \cdot x, 
\]
in order to have
\[
\nabla\tilde\zeta(x)=\nabla\zeta(x)-\frac{1}{M}\int_{\R^{d}}\rho_{\ast}\nabla\zeta\,dx,
\]
hence $\tilde \zeta$ satisfies \eqref{compatibility}. Inserting $\tilde\zeta$ in \eqref{EL} and taking into account that
\[
\int_{\R^{d}}T_{\ast}(x)\cdot \left(\int_{\R^{d}}\rho_{\ast}\nabla\zeta\,dx\right)\rho_{\ast}(x)dx=\left(\int_{\R^{d}}\rho_{\ast}\nabla\zeta\,dx\right)\cdot\,
\int_{\R^{d}}x\rho_{0}(x)dx=0,
\]
we have that \eqref{EL} holds for any test function $\zeta\in C^\infty(\mathbb R^d)$ such that $\nabla\zeta\in W^{1,\infty}(\R^{d};\R^d)$.
\end{proof}

The next result is based on a different perturbation of $\rho_*$, which gets perturbed along the solution of the heat equation originating from it.  For nonnegative $L^1(\R^d)$ functions $u$ with finite second moment on $\mathbb R^d$ we introduce the entropy functional 
\[
  \mathcal G[u
  ]:=\int_{\mathbb R^d}u(x)\log u(x)\,dx,
\]
which is a displacement convex functional in the sense of McCann \cite{Mc}.
We recall that the solution $u$ of the heat equation $\partial_t u=\Delta u$ with initial datum $\rho^0\in\mathcal Y_{M,2}$ is the Wasserstein gradient flow of $\mathcal G$ and it  satisfies the  evolution variational inequalities
\[
\frac12 W_2^2(u(t,\cdot), w)-\frac12 W_2^2(\rho^0,w)\le t\left(\mathcal G(w)-\mathcal G(u(t,\cdot))\right) \quad\mbox{ for every $w\in\mathcal Y_{M,2}$ and every $t>0$,}
\]
see \cite[Chapter 11]{AGS}.
This allows to take advantage of the flow interchange lemma introduced in \cite{MMS}, as we do in the next proof.

\begin{prop}[\bf improved regularity of discrete minimizers]\label{entropy}
Let $\tau>0$ and $\rho^0\in\mathcal Y_{M,2}$. Suppose that $\beta>0$. If $\rho_*$ is a solution to problem \eqref{MM}, then $\rho_*^{m/2}\in H^1(\mathbb R^d)$ and
\[\begin{aligned}
\frac{4}{m}\int_{\mathbb R^d}|\nabla\rho_*^{m/2}|^2\,dx&\le \chi s\,\left(\frac{\chi(1-s)}{2\beta}\right)^{\frac{1-s}{s}} \|\rho_*\|^2_{L^2(\mathbb R^d)}
+\frac{\mathcal G[\rho^0]-\mathcal G[\rho_*]}{\tau}.
\end{aligned}\]
\end{prop}
\begin{proof}
Let us introduce  the Cauchy problem
\begin{equation}\label{heat}
\left\{\begin{array}{ll}&\partial_tu=\Delta u\\ &u(0)=\rho_*.\end{array}\right.
\end{equation}
The unique solution to the heat equation with initial datum in $\rho_*\in\mathcal Y_{M,2}$ is given by $u(t,\cdot)=\Gamma_t\ast\rho_\ast$, where $\Gamma_t(x):=(4\pi t)^{-d/2}\exp\{-|x|^2/(4t)\}$ is the Gaussian kernel.  $u(t,\cdot)$ is smooth, positive  for every $t>0$, and moreover a direct computation by means of integration by parts  shows that for any $t>0$ there holds
\begin{equation}\label{derivative1bis}
\frac1{m-1}\frac{d}{dt}\int_{\mathbb R^d}u(t,x)^{m}\,dx=-\frac{4}{m}\int_{\mathbb R^d}\left|\nabla u(t,x)^{\tfrac{m}{2}}\right|^2\,dx
\end{equation}
and in particular
\begin{equation}\label{derivativebeta}\beta\,\frac{d}{dt}\int_{\mathbb R^d}u(t,x)^{2}\,dx=-{2\beta}\int_{\mathbb R^d}\left|\nabla u(t,x)\right|^2\,dx.\end{equation}
Similarly, thanks to \cite[Lemma 4.5]{LMS} we have for any $t>0$
\begin{equation}\label{derivative2bis}
\frac\chi2\frac{d}{dt}\int_{\mathbb R^d}\int_{\mathbb R^d}K_s(x-y)\,u(t,x)\,u(t,y)\,dx\,dy=-\chi\,\|u(t,\cdot)\|_{\dot H^{1-s}(\mathbb R^d)}.
\end{equation}
Here, $\dot H^r(\R^d)$ denotes the homogeneous Sobolev space of order $r\in\R$, i.e., the completion of $C^{\infty}_c(\R^d)$ with respect to the norm $\|w\|_{\dot H^r(\R^d)}^2:=(2\pi)^{-d}\int_{\R^d}|\xi|^{2r}|\hat w(\xi)|^2\,d\xi$.
 By the interpolation inequality $\|w\|_{\dot H^{1-s}(\mathbb R^d)}\le\|w\|^{s}_{L^2(\mathbb R^d)}\|w\|^{1-s}_{\dot H^1(\mathbb R^d)}$ and by Young inequality we deduce, for given $\alpha>0$,
 \begin{equation}\label{larges2}\begin{aligned}
 \chi\, \|u(t,\cdot)\|_{\dot H^{1-s}(\mathbb R^d)}^2
&\le \chi \, \|u(t,\cdot)\|_{L^2(\mathbb R^d)}^{2s}\|u(t,\cdot)\|_{\dot H^1(\mathbb R^d)}^{2(1-s)}\\&\le
\chi\, s\,\alpha^{-1/s}\|u(t,\cdot)\|_{L^2(\mathbb R^d)}^2+\chi(1-s)\,\alpha^{1/(1-s)}\|u(t,\cdot)\|^2_{\dot H^1(\mathbb R^d)}.
\end{aligned} \end{equation}
The choice $\alpha=(\chi(1-s)/(2\beta))^{s-1}$ in \eqref{larges2} entails, together with \eqref{derivative1bis} and \eqref{derivativebeta},
\begin{equation}\label{chi2beta}\begin{aligned}
\frac{d}{dt}\mathcal F_s(u(t,\cdot))&\le -\frac{4}{m}\int_{\mathbb R^d}\left|\nabla u(t,x)^{\tfrac{m}{2}}\right|^2\,dx
+\chi s\,\left(\frac{\chi(1-s)}{2\beta}\right)^{\frac{1-s}{s}}\|u(t,\cdot)\|^2_{L^2(\mathbb R^d)}
\end{aligned}\end{equation}
for every $t>0$.
Moreover,   the maps $$[0,+\infty)\ni t\mapsto \frac{1}{m-1}\int_{\mathbb R^d}u(t,x)^{m}\,dx, \qquad [0,+\infty)\ni t\mapsto \int_{\mathbb R^d}u(t,x)^{2}\,dx,$$$$
[0,+\infty)\ni t\mapsto \int_{\mathbb R^d}\int_{\mathbb R^d}K_s(x-y)\,u(t,x)\,u(t,y)\,dx\,dy$$
are continuous up to $t=0$. They are also differentiable at any $t>0$ with derivatives given by \eqref{derivative1bis}, \eqref{derivativebeta} and \eqref{derivative2bis}, therefore by Lagrange mean value theorem, for every $t>0$  there exists $\theta(t)\in(0,t)$ such that
\[\begin{aligned}
\frac{\mathcal F_s[\rho_*]-\mathcal F_s[u(t,\cdot)]}{t}&=-\frac{d}{dt}\mathcal F_s[u(t,\cdot)]\Bigg{|}_{t=\theta(t)},
\end{aligned}\]
so that by applying \eqref{chi2beta} we obtain
\[
\frac{\mathcal F_s[\rho_*]-\mathcal F_s[u(t,\cdot)]}{t}
\ge \frac{4}{m}\int_{\mathbb R^d}\left|\nabla u(\theta(t),x)^{\tfrac{m}{2}}\right|^2\,dx
-\chi s\,\left(\frac{\chi(1-s)}{2\beta}\right)^{\frac{1-s}{s}}\|u(\theta(t),\cdot)\|^2_{L^2(\mathbb R^d)}
\]
Since the $L^{2}(\mathbb R^d)$ norm decreases along the solution to the heat equation \eqref{heat}, we deduce
\[
\limsup_{t\to 0} \frac{\mathcal F_s[\rho_*]-\mathcal F_s[u(t,\cdot)]}{t}\ge \limsup_{t\downarrow0} \frac{4}{m} \int_{\mathbb R^d}|\nabla u(\theta(t),x)^{\tfrac m2}|^2\,dx-
\chi s\,\left(\frac{\chi(1-s)}{2\beta}\right)^{\frac{1-s}{s}} \|\rho_*\|^2_{L^2(\mathbb R^d)},
\]
showing that
\[\limsup_{t\to 0} \frac{\mathcal F_s[\rho_*]-\mathcal F_s[u(t,\cdot)]}{t}>-\infty.\]
Therefore, we can apply the flow interchange lemma from \cite{MMS}, in its version from \cite[Proposition 4.3]{LMS} and deduce
\[\begin{aligned}
\frac{\mathcal G[\rho^0]-\mathcal G[\rho_*]}{\tau}&\ge\limsup_{t\to 0} \frac{\mathcal F_s[\rho_*]-\mathcal F_s[u(t,\cdot)]}{t}\\&\ge \limsup_{t\downarrow0} \frac{4}{m} \int_{\mathbb R^d}|\nabla u(\theta(t),x)^{\tfrac m2}|^2\,dx-
\chi s\,\left(\frac{\chi(1-s)}{2\beta}\right)^{\frac{1-s}{s}} \|\rho_*\|^2_{L^2(\mathbb R^d)},
\end{aligned}\]
thus showing that the spatial gradient of $u(\theta(t),\cdot)^{m/2}$ stays bounded in $L^2(\mathbb R^d)$ as $t\downarrow 0$. But $u(\theta(t),\cdot)^{m/2}$  is  also bounded in $L^1(\mathbb R^d)$ as $t\downarrow 0$, since $\|u(\theta(t),\cdot)\|_{L^{m/2}(\mathbb R^d)}\le \|\rho_*\|_{L^{m/2}(\mathbb R^d)}$. Thus Sobolev embedding shows that $u(\theta(t),\cdot)^{m/2}$  is  in fact bounded in $H^1(\mathbb R^d)$ as $t\downarrow 0$. Since $\theta(t)\to 0$ as $t\downarrow 0$, and since  $u(\theta(t),\cdot)\to \rho_*$    pointwise a.e. as $t\downarrow 0$,  by the weak lower semicontinuity of the $H^1(\mathbb R^d)$ norm we finally deduce
\[
\chi s\,\left(\frac{\chi(1-s)}{2\beta}\right)^{\frac{1-s}{s}} \|\rho_*\|^2_{L^2(\mathbb R^d)}+\frac{\mathcal G[\rho^0]-\mathcal G[\rho_*]}{\tau}\ge  \frac{4}{m} \int_{\mathbb R^d}|\nabla \rho_*^{\tfrac m2}|^2\,dx,
\]
which is the desired result.
\end{proof}
\begin{remark}\label{redremark}The constant $\chi s\,\left(\frac{\chi(1-s)}{2\beta}\right)^{\frac{1-s}{s}}$ appearing in the above result is bounded as $s\to 0$ if and only if $\beta\ge \chi/2$.\end{remark}

\begin{cor}\label{gradestinterp} Let $\rho^0\in \mathcal Y_{M,2}$ and $\tau>0$. Let $\beta>0$.
Let us consider the sequence $(\rho_\tau^k)_{k\ge 0}$ of discrete minimizers defined by \eqref{recursive} and  the piecewise constant interpolation $\rho_\tau$ defined by \eqref{morceaux}. There holds for any $k\in\mathbb N$
\[
\frac{4}{m}\int_{\mathbb R^d}|\nabla(\rho_\tau^k(x)^{m/2})|^2\,dx\le \frac{\mathcal G[\rho_\tau^{k-1}]-\mathcal G[\rho_\tau^k]}{\tau}+\chi s\,\left(\frac{\chi(1-s)}{2\beta}\right)^{\frac{1-s}{s}}\,\int_{\R^d}(\rho^k_\tau(x))^2\,dx.
\]
Moreover, for every  $T>0$ 
 there holds the time integrated estimate
\begin{equation}\label{spacetime}
\frac{4}{m}\int_{0}^T\int_{\mathbb R^d}|\nabla(\rho_\tau(t,x))^{m/2}|^2\,dx\,dt\le C^*_1+C^*_2(T+\tau)+C^*_3(T+\tau)\,\chi s\,\left(\frac{\chi(1-s)}{2\beta}\right)^{\frac{1-s}{s}},
\end{equation}
where $C^*_i$, $i=1,2,3$, are  a suitable explicit constants, only depending on $\chi, M,m,s,d,\beta$, and on  $\rho^0$.
\end{cor}
\begin{proof} The first estimate in the statement is a direct consequence of Proposition \ref{entropy}, and it implies that for every $T>0$ we have
\begin{equation}\label{expl}
\begin{aligned}
&\frac4m\int_0^T\int_{\R^d}|\nabla (\rho_\tau(t,x)^{m/2})|^2\,dx\,dt\le \frac4m\int_0^{\lceil T/\tau\rceil\tau}
\int_{\R^d}|\nabla (\rho_\tau(t,x)^{m/2})|^2\,dx\,dt\\&\qquad=\frac {4\tau}m\sum_{k=1}^{\lceil T/\tau\rceil}\int_{\R^d}|\nabla(\rho_\tau^k(x)^{m/2})|^2\,dx
\\&\qquad\le \mathcal G[\rho^0]-\mathcal G[\rho_\tau^{\lceil T/\tau\rceil}]+\tau \chi s\,\left(\frac{\chi(1-s)}{2\beta}\right)^{\frac{1-s}{s}}\,\sum_{k=1}^{\lceil T/\tau\rceil}\int_{\R^d}(\rho^k_\tau(x))^2\,dx.
\end{aligned}
\end{equation}
By \eqref{basic3} and by  the standard estimate from \cite[Lemma 2.2]{BCC}, and since $\mathcal G[\rho_\tau^k]\le \int_{\R^d}(\rho_\tau^k)^m\,dx$ (recall that $m>2$), we have for every $k\in\mathbb N$
\[\begin{aligned}
|\mathcal G[\rho_\tau^k]|&\le \mathcal G[\rho_\tau^k]+\frac{1}{2}\int_{\R^d}|x|^2\rho_\tau^k(x)\,dx+d\log(4\pi)+2e^{-1}\\&\le \int_{\R^d}\rho_\tau^k(x)^m\,dx
+2\tau k\,\mathcal F_s[\rho^0]+\int_{\R^d}|x|^2\,\rho^0(x)\,dx+d\log(4\pi)+2e^{-1}.
\end{aligned}\]
From \eqref{basic2} we also have for every $k\in\mathbb N$
\begin{equation}\label{simple}\begin{aligned}
\int_{\R^d}(\rho^k_\tau(x))^2\,dx&\le M+\int_{\R^d}(\rho^k_\tau(x))^m\,dx\\&\le M+2\int_{\R^d}(\rho^0)^m\,dx+2(m-1)\left(\bar C+\beta\int_{\mathbb R^d}(\rho^0)^2\,dx\right)=:C_3^*,
\end{aligned}\end{equation}
We insert the latter two estimates, combined with  \eqref{basic1}-\eqref{basic2}, into \eqref{expl}: since $\lceil T/\tau\rceil\tau\le T+\tau$ we deduce
\[
\frac4m\int_0^T\int_{\R^d}|\nabla (\rho_\tau(t,x)^{m/2})|^2\,dx\,dt\le C^*_{1}+ C^*_2\,(T+\tau)+ C^*_3\,(T+\tau)\,\chi s\,\left(\frac{\chi(1-s)}{2\beta}\right)^{\frac{1-s}{s}},
\]
where
\begin{equation*}\label{cstar1}
C^*_1:=3\int_{\R^d}(\rho^0)^m\,dx+2(m-1)\left(\bar C+\beta\int_{\R^d}(\rho^0)^2\,dx\right)+d\log(4\pi)+2e^{-1}+\int_{\R^d}|x|^2\rho^0\,dx,
\end{equation*}
where $C_2^*$ is twice the right hand side of \eqref{basic1} and where $C_3^*$ is defined in \eqref{simple}.
The proof is concluded.
\end{proof}

In the case $\beta=0$ we provide an alternative estimate under the restriction $s\in[1/2,1)$.
\begin{prop}\label{propbis}
Let $d\ge 2,$ $s\in[1/2,1)$ and $\beta=0$.
Let $\tau>0$ and $\rho^0\in\mathcal Y_{M,2}$. If $\rho_*$ is a solution to problem \eqref{MM}, then $\rho_*^{m-1}\in H^1(\mathbb R^d)$ and
\[\begin{aligned}
\frac{m}{2(m-1)}\int_{\mathbb R^d}|\nabla\rho_*^{m-1}|^2\,dx&\le \frac{\chi^2\,S_{d,2s-1}^2(m-1)}{2m}\, \|\rho_*\|^2_{L^{\frac{2d}{d+4s-2}}(\mathbb R^d)}
\\&\qquad+\frac1{\tau(m-2)}\left(\int_{\mathbb R^d}(\rho^0(x))^{m-1}\,dx-\int_{\mathbb R^d}(\rho_*(x))^{m-1}\,dx\right),
\end{aligned}\]
where $S_{d,2s-1}$ is defined by \eqref{sds} if $1/2<s<1$ and $S_{d,0}:=1$.
\end{prop}
\begin{proof}
Let us introduce the auxiliary functional
\[
\mathcal G_m[u]:=\frac{1}{m-2}\int_{\mathbb R^d}u(x)^{m-1}\,dx
\]
and  the Cauchy problem
\begin{equation}\label{pm}
\left\{\begin{array}{ll}&\partial_tu=\Delta u^{m-1}\\ &u(0)=\rho_*.\end{array}\right.
\end{equation}
This is a standard porous media equation with initial datum in $\mathcal Y_{M,2}$ and it enjoys the following properties, for which we refer to \cite[Theorem 9.12, Proposition 9.13]{V}: there exists a unique strong solution $u$ (meaning that the equation $\partial_tu=\Delta u^{m-1}$ is satisfied pointwise a.e. in space-time) such that $u\in C^0([0,+\infty);L^1(\R^d))$ and $\nabla u^{m-1}\in L^2((0,T)\times(\R^d))$ for every $T>0$. Moreover, the map
$t\mapsto \int_{\R^d}u(t,x)^m\,dx$ is a  nonincreasing absolutely continuous map on $[0,T]$ and
\begin{equation}\label{derivative1}
\frac1{m-1}\frac{d}{dt}\int_{\mathbb R^d}u(t,x)^{m}\,dx=-\frac m{m-1}\int_{\mathbb R^d}\left|\nabla u(t,x)^{{m-1}}\right|^2\,dx\qquad
\mbox{for a.e.  $t>0$.}
\end{equation}
Moreover, the solution is the Wasserstein gradient flow of the displacement convex functional $\mathcal G_m$, see \cite[Theorem 11.2.5]{AGS}.

If we multiply \eqref{pm} by $K_s\ast u$, an integration by parts argument  shows that for every $\eta\in C^\infty((0,+\infty))$
\begin{equation}\label{parti}\begin{aligned}
&\frac12\int_0^{+\infty}\eta'(t)\int_{\R^d}(K_s\ast u)(t,x)\,u(t,x)\,dx\,dt=-\int_0^{+\infty}\int_{\R^d}\eta(t) (K_s\ast u)(t,x) \,\partial_t u(t,x)\,dx\,dt
\\&\qquad=-\int_0^{+\infty}\int_{\R^d}\eta(t)\Delta u(t,x)^{m-1}(K_s\ast u)(t,x)\,dx\,dt\\&\qquad=\int_0^{+\infty}\eta'(t)\int_{\R^d}\nabla u(t,x)^{m-1}\cdot\nabla(K_s\ast u)(t,x)\,dx\,dt
\end{aligned}\end{equation}
We notice that for any $s\in(1/2,1)$ and any $v\in L^{\frac{2d}{d+4s-2}}_+(\mathbb R^d)$, by Plancherel theorem and the Hardy-Littlewood-Sobolev inequality \eqref{HLS}, with the notation \eqref{sds}, there holds
\[
\int_{\R^d}|\nabla K_s\ast v|^2=\frac1{(2\pi)^d}\int_{\R^d}||\xi|^{1-2s}\hat v(\xi)|^2\,d\xi=\int_{\R^d}v\, K_{2s-1}\ast v\le S_{d,2s-1}\,\|v\|^2_{L
^\frac{2d}{d+4s-2}(\mathbb R^d)},
\]
thus $\nabla K_s\ast v\in (L^2(\R^d))^{d}$. If $s=1/2$, the above formula directly shows that $\|K_s\ast v\|_{L^2(\R^d)}=\|v\|_{L^2(\R^d)}$. Since $1<\tfrac{2d}{d+4s-2}\le 2<m$ and $u(t,\cdot)\in L^1\cap L^m(\R^d)$ for every $t\ge 0$ with $\|u(t,\cdot)\|_{L^m(\R^d)}\le \|\rho_*\|_{L^m(\R^d)}$, we deduce that
$\nabla K_s \ast u\in (L^2((0,T)\times\R^d))^{d}$ and then we get $\nabla u^{m-1}\cdot\nabla K_s\ast u\in L^1((0,T)\times\R^d)$. Therefore from \eqref{parti} we see that the map $$t\mapsto \int_{\R^d}\int_{\R^d}K_s(x-y)u(t,x)u(t,y)\,dx\,dy$$
is in $AC([0,T])$ with a.e. derivative given by
\begin{equation}\label{ksder}\begin{aligned}
\frac12\frac d{dt}\int_{\R^d}\int_{\R^d}K_s(x-y) u(t,x)u(t,y)\,dx\,dy&=-\int_{\R^d}\nabla u(t,x)^{m-1}\cdot\nabla(K_s\ast u)(t,x)\,dx\\& =-\left\langle u(t,\cdot),u(t,\cdot)^{m-1}\right\rangle_{\dot H^{1-s}(\mathbb R^d)},\end{aligned}
\end{equation}
where the last equality is due to Plancherel theorem, having introduced the following scalar product
 on $\dot H^{r}(\mathbb R^d)$, $r\in(0,1)$,
\[
\langle v,w\rangle_{\dot H^{r}}:=\frac{1}{(2\pi)^d}\int_{\mathbb R^d}|\xi|^{2r}\hat v(\xi)\overline{\hat w(\xi)}\,d\xi,
\]
so that by Cauchy-Schwarz inequality there holds
\[
\langle v,w\rangle_{\dot H^{r}}\le \| v\|_{\dot H^{2r-1}(\mathbb R^d)}\| w\|_{\dot H^1(\mathbb R^d)}.
\]
By taking advantage of the latter inequality, we deduce the following estimate for the scalar product $\left\langle u(t,\cdot),u(t,\cdot)^{m-1}\right\rangle_{\dot H^{1-s}(\mathbb R^d)}$. Indeed, if $1/2<s<1$, by \eqref{HLS}, by \eqref{sds} and by Young inequality we have
\begin{equation}\label{larges}\begin{aligned}
&\chi\,\left\langle u(t,\cdot),u(t,\cdot)^{m-1}\right\rangle_{\dot H^{1-s}(\mathbb R^d)}\le
 \chi\,\| u(t,\cdot)\|_{\dot H^{1-2s}(\mathbb R^d)}\| u(t,\cdot)^{m-1}\|_{\dot H^1(\mathbb R^d)}\\&\qquad\le\chi\,
  S_{d,2s-1}\|u(t,\cdot)\|_{L^\frac{2d}{d+4s-2}(\mathbb R^d)}\| u(t,\cdot)^{m-1}\|_{\dot H^1(\mathbb R^d)}\\&\qquad
 \le Q_{\chi,m,s,d}\|u(t,\cdot)\|^2_{L^{\frac{2d}{d+4s-2}}(\mathbb R^d)}+\frac{m}{2(m-1)}\int_{\mathbb R^d}|\nabla u(t,\cdot)^{m-1}|^2\,dx,
\end{aligned}\end{equation}
where $Q_{\chi,m,s,d}:={\chi^2\,S_{d,2s-1}^2(m-1)}/({2m})$,
which is readily seen to hold also for $s=1/2$, with the convention $S_{d,0}:=1$.


We have shown the absolute continuity  of the map $t\mapsto \mathcal F_s[u(t,\cdot)]$, which together with \eqref{derivative1} and \eqref{ksder} entails for every $t>0$
\[\begin{aligned}
&\frac{\mathcal F_s[\rho_*]-\mathcal F_s[u(t,\cdot)]}{t}=-\frac{1}{t}\int_0^t\mathcal F_s[u(r,\cdot)]\,dr\\&\qquad= \frac{m}{m-1}\,\frac1t\int_0^t \int_{\mathbb R^d}|\nabla u(r,x)^{m-1}|^2\,dx\,dr-\frac\chi t\int_0^t\,\left\langle u(r,\cdot),u(r,\cdot)^{m-1}\right\rangle_{\dot H^{1-s}(\mathbb R^d)}\,dr.
\end{aligned}\]
By applying  \eqref{larges}, since the $L^{\frac{2d}{d+4s-2}}(\mathbb R^d)$ norm decreases along the solution to the porous media equation \eqref{pm}, we deduce
\[\begin{aligned}
&\limsup_{t\to 0} \frac{\mathcal F_s[\rho_*]-\mathcal F_s[u(t,\cdot)]}{t}\\&\qquad\ge \limsup_{t\to 0} \frac{m}{2(m-1)} \frac1t\int_0^t\int_{\mathbb R^d}|\nabla u(r,x)^{m-1}|^2\,dx\,dr-
Q_{\chi,m,s,d} \|\rho_*\|^2_{L^{\frac{2d}{d+4s-2}}(\mathbb R^d)},
\end{aligned}\]
showing that
\[\limsup_{t\to 0} \frac{\mathcal F_s[\rho_*]-\mathcal F_s[u(t,\cdot)]}{t}>-\infty.\]
Therefore, we can apply the flow interchange lemma, in its version from \cite[Proposition 4.3]{LMS} and deduce
\[\begin{aligned}
\frac{\mathcal G_{m}[\rho^0]-\mathcal G_{m}[\rho_*]}{\tau}&\ge\limsup_{t\to 0} \frac{\mathcal F_s[\rho_*]-\mathcal F_s[u(t,\cdot)]}{t}\\&\ge \limsup_{t\to0} \frac{m}{2(m-1)} \frac1t\int_0^t\int_{\mathbb R^d}|\nabla u(r,x)^{m-1}|^2\,dx\,dr-
Q_{\chi,m,s,d} \|\rho_*\|^2_{L^{\frac{2d}{d+4s-2}}(\mathbb R^d)}.
\end{aligned}\]
By the absolute continuity of the map $t\mapsto \int_0^t\int_{\mathbb R^d}|\nabla u(r,x)^{m-1}|^2\,dx\,dr$, we may apply l'Hospital rule and get
\[\begin{aligned}
Q_{\chi,m,s,d} \|\rho_*\|^2_{L^{\frac{2d}{d+4s-2}}(\mathbb R^d)}+\frac{\mathcal G_{m}[\rho^0]-\mathcal G_{m}[\rho_*]}{\tau}&\ge \liminf_{t\to0} \frac{m}{2(m-1)} \int_{\mathbb R^d}|\nabla u(t,x)^{m-1}|^2\,dx.
\end{aligned}\]
By taking a suitable vanishing sequence $(t_n)_{n\in\mathbb N}$ of positive numbers,
the above bound shows that $\nabla u(t_n,\cdot)^{m-1}\to u_*$ weakly in $L^2(\R^d)$ as $n\to+\infty$.
But $u(t_n,\cdot)$ strongly converge to $\rho_*$ as $n\to+\infty$, hence up to subsequences we also have that
 $u(t_n,\cdot)^{m-1}\to \rho_*^{m-1}$ pointwise a.e. and weakly in $L^{\frac{m}{m-1}}(\R^d)$ (since $u(t_n,\cdot)$ is bounded in $L^{m}(\R^{d})$. This allows to conclude that $u_*=\nabla\rho_*^{m-1}$, and the weak lower semicontinuity of the $L^2(\R^d)$ norm yields the desired estimate. Since $\nabla \rho_*^{m-1}\in L^2(\R^d)$ and since $\rho^{m-1}\in L^{\frac{m}{m-1}}(\R^d)$, $m>2$, by the Gagliardo-Niremberg inequality we have $\rho^{m-1}\in L^2(\R^d)$ with the inequality
 \[
 \|\rho^{m-1}\|_{L^{2}(\R^{d})}\leq C\||\nabla \rho^{m-1}|\|^{\alpha}_{L^{2}(\R^{d})}\|\rho^{m-1}\|^{1-\alpha}_{L^{\frac{m}{m-1}}(\R^{d})},
 \]
 where
 $
 \alpha=\dfrac{N(m-2)}{2m+N(m-2)}\in (0,1).
 $
\end{proof}

By arguing as done for proving Corollary \eqref{gradestinterp}, from \eqref{propbis} we immediately deduce the following
\begin{cor}\label{secondcoro}  Let $d\ge2$, $s\in[1/2,1)$ and $\beta=0$. \KKK Let $T>0$.
Let $\rho^0\in \mathcal Y_{M,2}$ and $\tau>0$.
Let us consider the sequence $(\rho_\tau^k)_{k\ge 0}$ of discrete minimizers defined by \eqref{recursive} and  the piecewise constant interpolation $\rho_\tau$ defined by \eqref{morceaux}. Then
\begin{equation}\label{spacetimebis}
\int_{0}^T\int_{\mathbb R^d}|\nabla(\rho_\tau(t,x))^{m-1}|^2\,dx\,dt\le C_1^{**}+(T+\tau)C_2^{**}
\end{equation}
where $C_1^{**}$, $C_2^{**}$ are a suitable explicit constants, only depending on $\chi, M,m,s,d$ and the initial datum $\rho^0$.
\end{cor}

\begin{prop}[\bf Improved space-time compactness] \label{Simon}  Let $\beta\ge0$ and $\rho^0\in \mathcal Y_{M,2}$. If $\beta=0$, assume in addition that $d\ge 2$ and $1/2\le s<1$. \KKK
Let us consider a sequence $(\rho_\tau^k)_{k\ge 0}$ of discrete minimizers defined by \eqref{recursive},  the piecewise constant interpolation $\rho_\tau$ defined by \eqref{morceaux}, and let $\rho$ be a limit function obtained from {\rm Proposition \ref{convMM}} along a vanishing sequence   $(\tau_n)_{n\in\mathbb N}\subset (0,1)$.
Then $\rho_{\tau_n}\to\rho$ strongly in $L^2((0,T)\times\mathbb R^d)$ for any $T>0$.
\end{prop}
\begin{proof}
Let $\mathsf{T}_{\tau}^{k+1}$ be the unique optimal transport map (for the quadratic cost) from $\rho_{\tau}^{k+1}(x)\,dx$ to $\rho_{\tau}^{k}(x)\,dx$. Let $\varphi\in W^{2,\infty}(\mathbb R^d)$ be smooth, let $h>0$, and notice that
\begin{equation}\label{l4}\begin{aligned}
&\left|\int_{\mathbb R^d}\varphi(x)(\rho_{\tau}(t+h,x)-\rho_{\tau}(t,x))\right|=
\left|\sum_{k=\lceil t/\tau\rceil}^{\lceil(t+h)/\tau\rceil-1}\int_{\mathbb R^d}\varphi(x)\left(\rho_\tau^{k+1}(x)-\rho_\tau^k(x)\right)\,dx\right|\\&
=\left|\sum_{k=\lceil t/\tau\rceil}^{\lceil(t+h)/\tau\rceil-1}\int_{\mathbb R^d}\left(\varphi(\mathsf{T}_{\tau}^{k+1}(x))-\varphi(x)\right)\,\rho_{\tau}^{k+1}(x)\,dx\right|
\\&\le
\sum_{k=\lceil t/\tau\rceil}^{\lceil(t+h)/\tau\rceil-1}\left|\int_{\mathbb R^d}\nabla\varphi(x)\cdot(\mathsf{T}_\tau^{k+1}(x)-x)\rho_\tau^{k+1}(x)\,dx\right|\\
&\;\;+\left|\sum_{k=\lceil t/\tau\rceil}^{\lceil(t+h)/\tau\rceil-1}\frac12\int_0^1 (1-\xi)^{2}d\xi\int_{\mathbb R^d}(\mathsf{T}_{\tau}^{k+1}(x)-x)\,\nabla^2\varphi(\vartheta_{\xi})
\cdot(\mathsf{T}_\tau^{k+1}(x)-x)\rho_\tau^{k+1}(x)\,dx\right|
\end{aligned}
\end{equation}
where we used the Taylor expansion formula
\[\begin{aligned}
\varphi(\mathsf{T}_\tau^{k+1}(x))&=\varphi(x)+\nabla\varphi(x)\cdot (\mathsf{T}_\tau^{k+1}(x)-x)\\&\qquad+\frac{1}{2}\int_{0}^{1}
\left[(\mathsf{T}_{\tau}^{k+1}(x)-x)\,\nabla^2\varphi(\vartheta_{\xi})
\cdot(\mathsf{T}_\tau^{k+1}(x)-x)\right]\,(1-\xi)^{2}\, d\xi,
\end{aligned}\]
where 
$
\vartheta_{\xi}=\xi \mathsf{T}_\tau^{k+1}(x)+(1-\xi) x.
$
Let us separately treat the two terms in the right hand side.  For the first, by \eqref{EL} we have
\[\begin{aligned}&\left|\int_{\mathbb R^d}\nabla\varphi(x)\cdot(\mathsf{T}_\tau^{k+1}(x)-x)\rho_\tau^{k+1}(x)\,dx\right|\le
\left|\tau\int_{\mathbb R^d}\Delta\varphi(x)(\rho_{\tau}^{k+1}(x)^m+\beta\rho_\tau^{k+1}(x)^2)\,dx\right|\\&\qquad
+\left|\tau\,c_{d,s}\,\chi\, \frac{d-2s}{2}\int_{\mathbb R^d}\int_{\mathbb R^d}(\nabla\varphi(x)-\nabla\varphi(y))\cdot(x-y)|x-y|^{2s-d-2}\rho_\tau^{k+1}(x)\rho_\tau^{k+1}(y)\,dx\,dy\right|\\&\;\;\le \tau\|\varphi\|_{W^{2,\infty}(\mathbb R^d)}\left(\int_{\mathbb R^d}(\rho_\tau^{k+1}(x)^m+\beta\rho_\tau^{k+1}(x)^2)\,dx\right.\\&\qquad\qquad\qquad\qquad+\left.\frac\chi2\, (d-2s)\int_{\mathbb R^d}\int_{\mathbb R^d}K_s(|x-y|)\rho_\tau^{k+1}(x)\rho_\tau^{k+1}(y)\,dx\,dy\right)\le
K\tau\|\varphi\|_{W^{2,\infty}(\mathbb R^d)}
\end{aligned}\]
where $K$ is an explicit constant, only depending on $M, m,s,d,\beta,\chi$ and $\rho^0$, which can be obtained by applying Proposition \ref{base}, and in particular by combining \eqref{distantref} and \eqref{basic2}.
Concerning the second, we have
\[\begin{aligned}
&\left| \frac12\int_0^1 (1-\xi)^{2}d\xi\int_{\mathbb R^d}\nabla^2\varphi(\vartheta_{\xi})
(\mathsf{T}_{\tau}^{k+1}(x)-x)^{T}\cdot(\mathsf{T}_\tau^{k+1}(x)-x)\rho_\tau^{k+1}(x)\,dx\right|\\&\qquad\le\frac12 \|\varphi\|_{W^{2,\infty}(\mathbb R^d)} W_2^2(\rho_\tau^{k+1},\rho_\tau^k).
\end{aligned}\]
By inserting the latter two estimates in \eqref{l4} we deduce that for every smooth function $\varphi\in W^{2,\infty}(\mathbb R^d)$
\[\begin{aligned}
\int_{\mathbb R^d}\varphi(x)(\rho_{\tau}(t+h,x)-\rho_{\tau_n}(t,x))\,dx&\le \|\varphi\|_{W^{2,\infty}(\mathbb R^d)} \sum_{k=\lceil t/\tau\rceil}^{\lceil(t+h)/\tau\rceil-1}\left(K\tau+ \frac12 W_2^2(\rho_\tau^{k+1},\rho_\tau^k)\right)\\&\le \left(K(\tau+h)+2\tau(\bar C+\mathcal H_{m}[\rho^0])\right)
\|\varphi\|_{W^{2,\infty}(\mathbb R^d)},
\end{aligned}\]
where $\mathcal H_m$ is defined by \eqref{hm} and $\bar C$ is defined by \eqref{cbar},
and where the latter inequality follows again from Proposition \ref{base} together with \eqref{distantref}, and from the basic inequalities
$
\lceil \frac{t+h}{\tau}\rceil\leq \frac{t+h}{\tau}+1,\quad \lceil\frac{t}{\tau}\rceil\geq \frac{t}{\tau}.
$
Therefore, fixing $q\in\mathbb N$ with $q>2+d/2$, so that the continuous embedding given by Morrey's theorem $H^q(\mathbb R^d)\subset W^{2,\infty}(\mathbb R^d)$ holds with constant $Q$, we deduce
\begin{equation*}\label{-q}\begin{aligned}\|\rho_{\tau}(t+h,\cdot)-\rho_{\tau}(t,\cdot)\|_{H^{-q}(\mathbb R^d)}&=
\sup_{\|\varphi\|_{H^q(\mathbb R^d)}\le 1}\int_{\mathbb R^d}\varphi(x)(\rho_{\tau}(t+h,x)-\rho_{\tau}(t,x))\,dx\\&\le
Q\left(K(\tau+h)+2\tau(\bar C+\mathcal H_{m}[\rho^0])\right).
\end{aligned}\end{equation*}
In particular
\begin{equation}\label{AA}
\limsup_{\tau\to0}\|\rho_\tau(t+h,\cdot)-\rho_\tau(t,\cdot)\|_{H^{-q}(\mathbb R^d)}\le QKh.
\end{equation}

The conclusion is similar to the one in   \cite[Proposition 14]{BL}.  Let $(\tau_n)$ and $\rho$ be the vanishing sequence and the limit function in the statement.
Let us consider the set of functions
 $\{\rho_{\tau_n}(t,\cdot): t\in [0,T], n\in\mathbb N\}$. This is a set of functions having uniformly bounded second moments and uniformly bounded $L^1\cap L^2(\mathbb R^d)$ norm, thanks to the estimates \eqref{basic2} and \eqref{basic3}. Hence, it is relatively compact  in $H^{-q}(\mathbb R^d)$ by Lemma \ref{compact} below. Thanks to this fact and to the equicontinuity estimate \eqref{AA}, we may apply  \cite[Proposition 3.3.1]{AGS} and deduce that there exists a vanishing subsequence $(\tau_{n_k})$ and $\bar\rho\in C([0,T];H^{-q}(\R^d))$ such that
 \begin{equation}\label{weakq}
 \rho_{\tau_{n_k}}(t,\cdot)\to \bar\rho(t,\cdot)\quad\mbox{ in $H^{-q}(\mathbb R^d)$}\qquad\mbox{ for every $t\in[0,T]$.}
 \end{equation}
 By uniqueness of the limit we have $\bar\rho\equiv\rho$ and the above convergence holds along the original sequence $(\tau_n)$.

 The conclusion of the proof is split in two cases. We first consider the case $\beta>0$, and we will take advantage of Corollary \eqref{gradestinterp}.
 We observe that for every $f\in H^1(\mathbb R^d)$ there holds $|f|^{\frac{2}{m}}\in X_m:=W^{\frac{2}{m},m}(\mathbb R^d)$ since $m>2$: this is shown in \cite{M}
 along with the estimate $\||f|^{\frac2{m}}\|_{X_m}^{m}\le c \|\nabla f\|_{L^2(\mathbb R^d)}^2$ for a suitable constant $c$.
  Therefore, from \eqref{spacetime} and Jensen inequality we deduce that
\begin{equation}\label{strongbound}
\left(\frac1T\int_0^T \|\rho_{\tau_n}(t,\cdot)\|^2_{X_m}\,dt\right)^{m/2}\le\frac1T\int_0^T \|\rho_{\tau_n}(t,\cdot)\|^{m}_{X_m}\,dt\le C_T
\end{equation}
where $C_T$ is a constant depending only on $M,m,s,d,\beta,\chi,T$ and the initial datum $\rho^0$. This shows that the sequence $\rho_{\tau_n}$ is bounded in $L^{2}((0,T);X_{m})$.
%
 Since $L^2(\mathbb R^d)$ continuously embeds in $H^{-q}(\mathbb R^d)$, by the uniform $L^2(\mathbb R^d)$  bound deduced from \eqref{basic2} and by \eqref{weakq} we may apply the dominated convergence theorem and obtain the convergence of $\rho_{\tau_n}$ to $\rho$
 in $L^2((0,T); H^{-q}(\mathbb R^d))$.
 The latter convergence,
 together with \eqref{strongbound} allows for an application of the compactness result in the space $L^2((0,T); L^2(\mathbb R^d))$ from \cite[Lemma 9]{S}
   so that we conclude that  $\rho_{\tau_n}\to\rho$ strongly in $L^2((0,T)\times\mathbb R^d)$.
   Here, \cite[Lemma 9]{S} is applied by using the Banach triple  $Y\subset L^{2}(\R^d) \subset H^{-q}(\R^d)$, where $Y=X_{m}\cap L^{1}(\R^{d},(1+|x|^{2})dx)$, and where the first embedding is compact by Lemma \ref{compact}.   The proof for the case $\beta>0$ is concluded

Eventually, let us consider the case $\beta=0$, $d\ge 2$, $1/2\le s<1$. In this case we change the definition of the Sobolev space $X_m$ and we let $X_m=W^{\frac{1}{m-1},\,2m-2}(\R^d)$, and by invoking Corollary \eqref{secondcoro} instead of Corollary \eqref{gradestinterp}, we conclude by repeating the same argument that we have used for $\beta>0$.
\end{proof}

 \begin{lem}\label{compact} Let $m>2$. The spaces
 $Y_1:=W^{\frac2m,m}(\R^d)\cap L^1(\R^d,(1+|x|^2)\,dx)$ and $Y_2:=W^{\frac1{m-1},\,2m-2}(\R^d)\cap L^1(\R^d,(1+|x|^2)\,dx)$ are compactly embedded into $L^2(\R^d)$. The space $L^2(\R^d)\cap L^1(\R^d,(1+|x|^2)\,dx)$ is compactly embedded into $H^{-q}(\R^d)$ if $q>d/2$.
 \end{lem}
 \begin{proof}
 Let us give the proof for $Y_1$ (the argument for $Y_2$ is analogous).
  Let us consider a sequence $(u_{n})_{n\in\mathbb N}$ which is bounded in $Y_1$, thus in particular it is bounded in $W^{\frac2m,m}(\R^d)$ and in $L^{m}\cap L^{1}(\R^d)$. By fractional Sobolev embedding, since $m>2$ we have that $W^{\frac2m,m}(\R^d)$ is embedded compactly in $L^{2}(B)$ for every ball in $\R^d$, so that there is $u\in L^2(\R^d)$ and a not relabeled subsequence $(u_{n})$ such that $u_{n}\rightarrow u$ strongly in $L^{2}(B)$ for every ball $B$ and weakly in $L^{m}(\R^d)$.
  Let $\eps>0$ and choose $B=B_\eps$ to be a large enough ball, such that $\int_{\R^d\setminus B}|u|+\int_{\R^d\setminus B}|u_n|<\eps$ for every $n\in\mathbb N$: this is possible thanks to the tightness of the sequence $(u_n)$, which has uniformly bounded second moments by assumption.
 We have
\begin{equation}\label{eee}
\|u_{n}-u\|_{L^2(\R^d)}\le \|u_{n}-u\|^{1-\theta}_{L^{1}(\R^d)}\|u_{n}-u\|^{\theta}_{L^{m}(\R^d)}.
\qquad \theta:=\frac m{2m-2},
\end{equation}
We also have
\[
\int_{\R^{d}}|u_{n}-u|\,dx\leq\int_{B}|u_{n}-u| dx+\int_{\R^d\setminus B}|u_{n}-u|\, dx\leq \int_{B}|u_{n}-u| dx
+\eps
\]
and since $u_n\to u$ strongly in $L^1(B)$,
the arbitrariness of $\eps$ shows that $u_n\to u $ strongly in $L^1(\R^d)$ as well. Therefore, the boundedness of $(u_n)$ in $L^m(\R^d)$ and \eqref{eee} show that $u_n\to u$ strongly in $L^2(\R^d)$.

Similarly, let $(u_n)_{n\in\mathbb N}$ be a bounded sequence in $L^2(\R^d)\cap L^1(\R^d,(1+|x|^2)\,dx)$,  and given $\eps>0$ let as above $B=B_\eps$ such that $\int_{\R^d\setminus B}|u|+\int_{\R^d\setminus B}|u_n|<\eps$ for every $n\in\mathbb N$.
By Sobolev embedding,  $H^q(B)$ compactly embeds into $L^2(B)$ and then (by Schauder's theorem) $L^2(B)$ compactly embeds in the dual space $H^q(B)^*$, therefore up to subsequences we have $u_n\to u$ weakly in $L^2(\R^d)$ and strongly in $H^q(B)^*$. We have
\[\begin{aligned}
\|u_n-u\|_{H^{-q}(\R^d)}&\le \sup_{\|\varphi\|_{H^q(\R^d)}\le1} \left|\int_{B}\varphi(u_n-u)\right|+ \sup_{\|\varphi\|_{H^q(\R^d)}\le1} \left|\int_{\R^d\setminus B}\varphi(u_n-u)\right|\\&\le \sup_{\|\varphi\|_{H^q(\R^d)}\le1} \|\varphi\|_{H^q(B)}\|\rho_n-\rho\|_{H^q(B)^*}+C\eps
\le \|\rho_n-\rho\|_{H^q(B)^*}+C\eps,
\end{aligned}
\]
where we have also used the continuous embedding $\|\cdot\|_{L^\infty(\R^d)}\le C \|\cdot\|_{H^q(\R^d)}$, since $q>d/2$.
Taking the limit as $n\to+\infty$, since $u_n\to u$ strongly in $H^{q}(B)^*$ and since $\eps$ is arbitrary, we deduce that  $\|u_n-u\|_{H^{-q}(\R^d)}\to0.$
 \end{proof}

We are ready to prove Theorem \eqref{existencethm}, recalling that by a gradient flow  solution $\rho=\rho(t,x)$ to \eqref{cauchy1} we mean
a weak solution according to \eqref{veryweak} which is a limit of the JKO scheme, i.e., $\rho$ is a limit function (obtained from {\rm Proposition \ref{convMM}} along a vanishing sequence   $(\tau_n)_{n\in\mathbb N}\subset (0,1)$)  of the piecewise constant interpolations $\rho_\tau$ (defined by \eqref{morceaux}) constructed from a sequence $(\rho_\tau^k)_{k\ge 0}$ of discrete minimizers from  \eqref{recursive}.

%
\begin{proofad4}
We apply Proposition \ref{Simon}: we have $\rho_{\tau_n}\to\rho$ strongly in $L^2((0,T)\times\mathbb R^d)$.
In particular, up to extracting a further subsequence,  we have the pointwise a.e. space-time convergence
of $\rho_{\tau_n}$ to $\rho$ and thus of $\rho_{\tau_n}^m$ to $\rho^m$. Thanks to the Sobolev embedding $ H^1(\mathbb R^d)\subset L^{\frac{2d}{d-2}}(\mathbb R^d)$  if $d\ge 3$ (resp. $ H^1(\mathbb R^d)\subset L^{4}(\mathbb R^d)$  if $d=1,2$) we deduce  from \eqref{spacetime} if $\beta>0$ and from \eqref{spacetimebis} if $\beta=0$
that the sequence $(\rho_{\tau_n}^m)_{n\in\mathbb N}$ is also bounded in $L^{\frac{d}{d-2}}((0,T)\times\mathbb R^d)$ if $d\ge 3$ (resp. in $L^2((0,T)\times\mathbb R^d)$ if $d=1,2$),  
hence by interpolation we also find that $(\rho_{\tau_n}^2)_{n\in\mathbb N}$ is bounded in the same space. Thus up the the extraction of one more subsequence we get
\begin{equation*}
\begin{split}
&\lim_{n\to+\infty}\int_0^T\int_{\mathbb R^d}\eta(t)\Delta\varphi(x)\left(\rho_{\tau_n}(t,x)^m+\beta \rho_{\tau_n}(t,x)^2\right)\,dx\,dt\\&=\int_0^T\int_{\mathbb R^d}\eta(t)\Delta\varphi(x)\left(\rho(t,x)^m+\beta \rho(t,x)^2 \right)\,dx\,dt,
\end{split}
\end{equation*}
for every $\eta\in C^\infty_c(0,T)$ and every $\varphi\in C^\infty_c(\mathbb R^d)$. The weak $L^2(\mathbb R^d)$ convergence of $\rho_{\tau_n}(t,\cdot)$ to $\rho(t,\cdot)$ is sufficient for obtaining
\[\begin{aligned}
&\lim_{n\to+\infty} \int_{\mathbb R^d}\int_{\mathbb R^d}(\nabla\varphi(x)-\nabla\varphi(y))\cdot(x-y)|x-y|^{2s-d-2}\rho_{\tau_n}(t,x)\,\rho_{\tau_n}(t,y)\,dx\,dy\\&\qquad\qquad= \int_{\mathbb R^d}\int_{\mathbb R^d}(\nabla\varphi(x)-\nabla\varphi(y))\cdot(x-y)|x-y|^{2s-d-2}\rho(t,x)\,\rho(t,y)\,dx\,dy
\end{aligned}\]
for every $t$, by making use of the same argument of the proof of Proposition \ref{existencediscrete}.
By dominated convergence, the associated time integrals on $(0,T)$ also converge.
Finally, we have for every $\eta\in C^\infty_c(0,T)$, by the convergence properties of minimizing movements, see for instance \cite[Theorem 11.1.6]{AGS},
\[
\lim_{n\to+\infty}\frac1{\tau_n}\int_0^T\eta(t)\int_{\mathbb R^d}(\mathsf T_{\tau_n}(t,x)-x)\cdot\nabla\varphi(x)\,\rho_{\tau_n}(t,x)\,dx\,dt=\int_0^T\partial_t\eta(t)\int_{\mathbb R^d}\varphi(x)\rho(t,x)\,dx\,dt,
\]
where $\mathsf T_{\tau_n}(t,\cdot)$ is the unique optimal transport map from $\rho_{\tau_n}(t,x)\,dx$ to $\rho_{\tau_n}^{\lceil t/\tau_n\rceil-1}(x)\,dx$.
Recalling the definition of $\rho_\tau$ as the piecewise constant interpolation of discrete minimizers, by writing \eqref{EL} for $\rho_\tau$ we have
\begin{equation*}\label{ELfinal}\begin{aligned}&\frac{1}{\tau_{n}}\int_{\mathbb R^d}(\mathsf T_{\tau_{n}}(t,x)-x)\cdot\nabla\varphi(x)\,\rho_{\tau_{n}}(t,x)\,dx=
-\int_{\mathbb R^d}\Delta\varphi(x)\left(\,\rho_{\tau_{n}}(t,x)^{m}+\beta \rho_{\tau_{n}}(t,x)^{2}\right)\,dx\\&\qquad+\frac{(d-2s)\chi\,c_{d,s}}2\int_{\mathbb R^d}\int_{\mathbb R^d}(\nabla\zeta(x)-\nabla\zeta(y))\cdot(x-y)|x-y|^{2s-d-2}\,\rho_{\tau_{n}}(t,x)\,\rho_{\tau_{n}}(t,y)\,dx\,dy\end{aligned}\end{equation*}
for every $t>0$. 
By multiplying the latter by $\eta\in C^\infty_c(0,T)$ and by integrating on $(0,T)$, we may therefore pass to the limit along the sequence
$({\tau_n})_{n}$ and conclude that $\rho$ is a  weak solution to problem \eqref{cauchy1}.
 \end{proofad4}
 Let us collect some properties of the constructed solution.
 \begin{prop}\label{14} Let $\beta\ge 0$. If $\beta=0$, assume in addition that $d\ge 2$, $1/2\le s<1$. Let $\rho^0\in \mathcal Y_{M,2}$. \KKK
Let  $(\rho_\tau^k)_{k\ge 0}$ be a sequence of discrete minimizers defined by \eqref{recursive}, let $\rho_\tau$ be  the piecewise constant interpolation defined by \eqref{morceaux}, and let $\rho$ be a limit function obtained from {\rm Proposition \ref{convMM}} along a vanishing sequence   $(\tau_n)_{n\in\mathbb N}\subset (0,1)$.
Then the following properties hold.
\begin{itemize}
\item[(i)] The function $[0,+\infty)\ni t\mapsto \rho(t,\cdot)\in\mathcal Y_{M,2}$ is  absolutely continuous with respect to the Wasserstein distance $W_2$. 
\item[(ii)] The following $L^m(\R^d)$ estimate holds for every $t>0$
\begin{equation*}
\frac{1}{2(m-1)}\int_{\mathbb R^d}(\rho(t,x))^m\,dx\le \frac1{m-1}\int_{\mathbb R^d}(\rho^0(x))^m\,dx+\beta\int_{\mathbb R^d}(\rho^0(x))^2\,dx+\bar C,
\end{equation*}
where $\bar C$ is defined by \eqref{cbar} and \eqref{sds}.
\item[(iii)] if $\beta>0$, then $\rho^{m/2}\in L^2((0,T);H^1(\R^d))$ for every $T>0$ along with the estimate
\begin{equation*}
\frac{4}{m}\int_{0}^T\int_{\mathbb R^d}|\nabla(\rho(t,x))^{m/2}|^2\,dx\,dt\le C^*_1+TC_2^*+TC_3^*\,\chi s\,\left(\frac{\chi(1-s)}{2\beta}\right)^{\frac{1-s}{s}},\end{equation*}
where $C_i^*$ are the explicit constants defined in the proof of {\rm Proposition \ref{gradestinterp}}.
\item[(iv)] for every $T>0$ there holds
\begin{equation*}\begin{aligned}
\int_{\mathbb R^d}|x|^2\rho(t,x)\,dx&\le 4T \mathcal F_s[\rho^0]+2\int_{\mathbb R^d}|x|^2\rho^0(x)\,dx\\&\le 4T\left(\frac1{m-1}\int_{\mathbb R^d}(\rho^0(x))^m\,dx+\beta\int_{\mathbb R^d}(\rho^0(x))^2\,dx\right)+2\int_{\mathbb R^d}|x|^2\rho^0(x)\,dx.
\end{aligned}
\end{equation*}
\item[(v)] if $\beta=0$, then $\rho^{m-1}\in L^2((0,T);H^1(\R^d))$ for every $T>0$ along with the estimate
\begin{equation*}
\int_{0}^T\int_{\mathbb R^d}|\nabla(\rho(t,x))^{m-1}|^2\,dx\,dt\le C^{**}_1+TC_2^{**},
\end{equation*}
where $C_i^{**}$ are the explicit constants appearing in {\rm Corollary  \ref{secondcoro}}.
\end{itemize}
\end{prop}
\begin{proof}
Point {\it (i)} was shown in Proposition \ref{convMM}. Point {(ii)} follows from the uniform $L^m(\R^d)$ bound from \eqref{basic2} along with the narrow lower semicontinuity of the $L^m$ norm. Points {\it (iii)} and {\it (iv)} respectively follow from the uniform bounds \eqref{spacetime} and \eqref{basic3}, again by  lower semicontinuity properties. Similarly, point {\it (v)} follows from \eqref{spacetimebis}.
\end{proof}

 We conclude this section by proving Theorem \ref{stozero}.
 Before giving the proof, we include a couple of technical lemmas, whose proofs are postponed to the Appendix.

\begin{lem}\label{lastbutone}
Let $\rho\in L^1\cap L^2(\mathbb R^d)$. Then for every $\varphi\in C^\infty_c(\mathbb R^d)$ there holds
\[
\lim_{s\downarrow 0}\, (d-2s)\,c_{d,s}\int_{\R^d}\int_{\R^d}(\nabla\varphi(x)-\nabla\varphi(y))\,\cdot(x-y)|x-y|^{2s-d-2}\rho(x)\rho(y)\,dx\,dy=\int_{\R^d}\rho^2(x)\Delta\varphi(x)\,dx.
\]
\end{lem}

A generalization of the previous lemma is the following

\begin{lem}\label{lastlemma}
Let $\rho\in L^1\cap L^2(\mathbb R^d)$ and let $(\rho_s)_{s\in (0,1/2)}\subset L^1\cap L^2(\R^d)$ be a family of functions such that $\rho_s\to\rho$ in $L^2(\mathbb R^d)$ as $s\downarrow0$ and $\|\rho_s\|_{L^1(\R^d)}=\|\rho\|_{L^1(\R^d)}$ for every $s\in(0,1/2)$. Then for every $\varphi\in C^\infty_c(\mathbb R^d)$ there holds
\[
\lim_{s\downarrow 0}\, (d-2s)\,c_{d,s}\int_{\R^d}\int_{\R^d}(\nabla\varphi(x)-\nabla\varphi(y))\,\cdot(x-y)|x-y|^{2s-d-2}\rho_s(x)\rho_s(y)\,dx\,dy=\int_{\R^d}\rho^2(x)\Delta\varphi(x)\,dx.
\]
\end{lem}


We are ready for the proof of our last result.

\begin{proofad5}
Let $s=s_n$. Let $\varphi\in C^{\infty}_c(\mathbb R^d)$. Since $[0,+\infty)\ni t\mapsto \rho_{n}(t,\cdot)$ is absolutely continuous with respect to $W_2$ as recalled in point {\it (i)} of Proposition \ref{14}, the map $t\mapsto \int_{\R^d}\varphi(x)\rho_n(t,x)\,dx$ is absolutely continuous on $[0,+\infty)$ and we may write  the time integrated version of \eqref{veryweak}, i.e.,
\[\begin{aligned}
&\int_{\mathbb R^d}(\rho_n(t_2,x)-\rho_n(t_1,x))\varphi(x)\,dx=\int_{t_1}^{t_2}\int_{\R^d}\Delta\varphi(x)(\rho_n^m(t,x)+\beta\rho_n^2(t,x))\,dx\,dt
\\&\quad-\frac\chi2\,(d-2s)\,c_{d,s}\int_{t_1}^{t_2}dt\int_{\R^d}\int_{\mathbb R^d}(\nabla\varphi(x)-\nabla\varphi(y))\cdot(x-y)|x-y|^{2s-d-2}\rho_n(t,x)\rho_n(t,y)\,dx\,dy,\end{aligned}
\]
 for every $0\le t_1<t_2<+\infty$.
We estimate the last term as done in \eqref{18}-\eqref{distantref}, obtaining
\[
\begin{aligned}
&\left|\frac\chi2\,(d-2s)\,c_{d,s}\int_{t_1}^{t_2}\int_{\R^d}\int_{\mathbb R^d}(\nabla\varphi(x)-\nabla\varphi(y))\cdot(x-y)|x-y|^{2s-d-2}\rho_n(t,x)\rho_n(t,y)\,dx\,dy\right|
\\&\qquad\le \frac\chi2(d-2s) \|\nabla^2\varphi\|_{L^\infty(\R^d)}\,\int_{t_1}^{t_2}\int_{\R^d}(K_{s_n}\ast\rho_n)(t,x)\,\rho_n(t,x)\,dx\\
&\qquad \le d\, \|\nabla^2\varphi\|_{L^\infty(\R^d)}
\int_{t_1}^{t_2}\left(\bar C(\chi, m,s_n,d, M)+\frac1{2(m-1)}\int_{\R^d}\rho_n^m(t,x)\,dx\right)\,dt,
\end{aligned}
\]
where $\bar C$ is defined by \eqref{cbar} and \eqref{sds},
therefore we have by  $L^1-L^2-L^m$ interpolation
\[\begin{aligned}
&\left|\int_{\mathbb R^d}(\rho_n(t_2,x)-\rho_n(t_1,x))\varphi(x)\,dx\right|\le \|\nabla^2\varphi\|_{L^\infty(\R^d)}\int_{t_1}^{t_2}\int_{\R^d}(\rho_n^m(t,x)+\beta\rho_n^2(t,x))\,dx\,dt
\\&\quad+\frac\chi2\,(d-2s)\,\|\nabla^2\varphi\|_{L^\infty(\R^d)}\int_{t_1}^{t_2}\int_{\R^d}(K_{s_n}\ast\rho_n)(t,x)\,\rho_n(t,x)\,dx,\\
&\le \|\nabla^2\varphi\|_{L^\infty(\R^d)}\int_{t_1}^{t_2}\int_{\R^d}\rho_n^m(t,x)\,dx\,dt\\&\qquad+\beta\|\nabla^2\varphi\|_{L^\infty(\R^d)}\int_{t_1}^{t_2}\left(\frac{m-2}{m-1}\,M+\frac1{m-1}\int_{\mathbb R^d}\rho_n^m(t,x)\,dx\right)\,dt\\&\qquad+d\, \|\nabla^2\varphi\|_{L^\infty(\R^d)}
\int_{t_1}^{t_2}\left(\bar C(\chi, m,s_n,d, M)+\frac1{2(m-1)}\int_{\R^d}\rho_n^m(t,x)\,dx\right)\,dt.
\end{aligned}\]
It is immediate to check that $\sup_{s\in(0,1/2)}\bar C(\chi,m,s,d,M)<+\infty$, therefore by including the estimate in point {\it (ii)} of Proposition \ref{14}
we deduce that
\[
\left|\int_{\mathbb R^d}(\rho_n(t_2,x)-\rho_n(t_1,x))\varphi(x)\,dx\right|\le \tilde C(t_2-t_1) \|\varphi\|_{W^{2,\infty}(\R^d)}
\]
where $\tilde C$ is a suitable constant, depending on $m,d,\chi,\beta, M$ and $\rho^0$, but not on $n$. We take $q\in\mathbb N$, $q>2+d/2$, so that we have the continuous embedding $H^q(\R^d)\subset W^{2,\infty}(\R^d)$ with embedding constant $Q$, and we deduce the time equi-Lipschitz estimate
\[
\|\rho_n(t_2,\cdot)-\rho_n(t_1,\cdot)\|_{H^{-q}(\R^d)}\le Q\tilde C(t_2-t_1)\qquad \mbox{ for every $n\in\mathbb N$}.
\]

Letting $T>0$, we repeat the same arguments of the proof of Proposition \ref{Simon}: indeed, the family of functions $\{\rho_n(t,\cdot): t\in[0,T], \,n\in\mathbb N\}$ has uniformly bounded second moments and $L^1\cap L^2(\R^d)$ norms, which is seen by applying to $\rho_n$ the estimates of points {\it (ii)} and {\it (iv)} of Proposition \ref{14} (as already noticed, the right hand side of {\it (ii)} can be estimated uniformly with respect to $n\in\mathbb N$). Therefore by Lemma \ref{compact} such a family of functions is relatively compact in $H^{-q}(\R^d)$, so that in view of the above equi-Lipschitz estimate we may apply \cite[Lemma 3.3.1]{AGS} to find $\rho\in C^0([0,T];H^{-q}(\R^d))$ such that for every $t\in[0,T]$ there holds  $\rho_n(t,\cdot)\to\rho(t,\cdot) $ in $H^{-q}(\R^d) $ as $n\to+\infty$.
 By the dominated convergence theorem, we also get $\rho_n\to \rho \in L^2((0,T);H^{-q}(\R^d))$ as $n\to+\infty$: here, the dominating function for showing that  $\int_0^T\|\rho_n(t,\cdot)-\rho(t,\cdot)\|^2_{H^{-q}(\R^d)}\,dt\to 0$ as $n\to+\infty$ is obtained by using the    continuous embedding of $L^2(\R^d)$ into $H^{-q}(\R^d)$ and the estimate in point {\it (ii)} of Proposition \ref{14}, where again the right hand side is uniformly bounded with respect to $n$. The same estimate and estimate in point {\it (iv)} of the same Proposition implies that $\rho_{n}(t,\cdot)$ converges weakly to $\rho(t,\cdot)$ in $L^{1}\cap L^{m}$, therefore the weak lower semicontinuity of the $L^m$ norm also shows that $\rho\in L^\infty((0,+\infty);L^m(\R^d)).$

The estimate in point {\it (iii)} of Proposition \ref{14}, applied to $\rho_n$, implies
\begin{equation*}
\sup_{n\in\mathbb N}\int_{0}^T\int_{\mathbb R^d}|\nabla(\rho_n(t,x))^{m/2}|^2\,dx\,dt\le\frac m4\,\sup_{n\in\mathbb N}\left(C_1^*+TC_2^*+TC_3^*\,\chi s_n\,\left(\frac{\chi(1-s_n)}{2\beta}\right)^{\frac{1-s_n}{s_n}}\right),\end{equation*}
and the supremum in the right hand side is finite, thanks to the crucial assumption $\beta\ge \chi/2$ (as observed in Remark \ref{redremark}).
Therefore, we may reason as done in the proof of Proposition \eqref{Simon} to get that the sequence $(\rho_n)$ enjoys a uniform $L^2((0,T);W^{2/m,m}(\R^d))$ bound.
Similarly, in view of point {\it (iv)} of Proposition \ref{14},  the sequence $(\rho_n)$ is also uniformly bounded in  $L^2((0,T);L^1(\R^d,(1+|x|^2)\,dx))$.
By the same argument at the end of the proof of Proposition \ref{Simon}, we have $\rho_n\to \rho$ strongly in  $L^2((0,T)\times\mathbb R^d)$.

Let us conclude by passing to the limit in the equation. Let $\varphi\in C^\infty_c(\R^d)$ and $\eta\in C^\infty_c((0,T))$. By definition of weak solution, for each $n\in\mathbb N$ we have that $\rho_n$ satisfies
\begin{equation}\label{weakn}\begin{aligned}
&-\int_0^T\int_{\mathbb R^d}\rho_n(t,x)\partial_t\eta(t)\varphi(x)\,dx\,dt=\int_0^T\int_{\mathbb R^d}\eta(t)\Delta\varphi(x)\,\big(\rho_n(t,x)^m+\beta\rho_n(t,x)^2\big)\,dx\,dt
\\&\qquad-\frac{(d-2s_n)\,c_{d,s_n}\,\chi}{2}\int_{0}^T\eta(t)\int_{\mathbb R^d\times \R^d}\frac{\big(\nabla\varphi(x)-\nabla\varphi(y)\big)
\cdot(x-y)}{|x-y|^{d+2-2s_n}}\rho_n(t,x)\,\rho_n(t,y)\,dx\,dy\,dt.
\end{aligned}
 \end{equation}
 Since
 $\rho_n\to \rho$ strongly in  $L^2((0,T)\times\mathbb R^d)$, up to taking another subsequence we have $\rho_n(t,\cdot)\to\rho(t,\cdot)$ in $L^2(\R^d)$ for a.e. $t\in(0,T)$. An application of Lemma \ref{lastlemma} entails therefore
 \[\begin{aligned}
&\lim_{n\to+\infty}\frac\chi 2\, (d-2s_n)\,c_{d,s_n}\int_{\R^d}\int_{\R^d}(\nabla\varphi(x)-\nabla\varphi(y))\,\cdot(x-y)|x-y|^{2s_n-d-2}\rho_n(t,x)\rho_n(t,y)\,dx\,dy\\&\qquad=\frac\chi2\int_{\R^d}\rho^2(t,x)\Delta\varphi(x)\,dx\qquad\quad\mbox{ for a.e. $t\in(0,T).$}
\end{aligned}\]
After multiplying by $\eta$ and integrating on $(0,T)$, the time integral passes to the limit by dominated convergence: a dominating function is obtained by the usual estimates of the form \eqref{18}-\eqref{distantref}, yielding for a.e. $t\in(0,T)$
\[
\begin{aligned}
&\left|\frac\chi 2\, (d-2s_n)\,c_{d,s_n}\int_{\R^d}\int_{\R^d}(\nabla\varphi(x)-\nabla\varphi(y))\,\cdot(x-y)|x-y|^{2s_n-d-2}\rho_n(t,x)\rho_n(t,y)\,dx\,dy\right|\\
&\qquad\le d\, \|\nabla^2\varphi\|_{L^\infty(\R^d)}\|\eta\|_{L^\infty(\R)}
\left(\bar C(\chi, m,s_n,d, M)+\frac1{2(m-1)}\int_{\R^d}\rho_n^m(t,x)\,dx\right)\\
&\qquad \le d\, \|\nabla^2\varphi\|_{L^\infty(\R^d)}\|\eta\|_{L^\infty(\R)} \left(2\bar C+\frac1{m-1}\int_{\mathbb R^d}(\rho^0(x))^m\,dx+\beta\int_{\mathbb R^d}(\rho^0(x))^2\,dx\right),
\end{aligned}
\]
where we have also used point {\it (ii)} of Proposition \ref{14}, and $\bar C=\bar C(\chi, m,s_n,d, M)$ stays bounded as $n\to+\infty.$

Eventually,
we take advantage of the previously obtained uniform $L^2((0,T);W^{2/m,m}(\R^d))$ estimate: as in the proof of Theorem \ref{existencethm}, by Sobolev embedding it implies that $(\rho_n^m)_{n\in\mathbb N}$ is also uniformly  bounded in $L^{\frac{d}{d-2}}((0,T)\times\mathbb R^d)$ if $d\ge 3$ (and in $L^2((0,T)\times\mathbb R^d)$ if $d=1,2$). Therefore, up to subsequences, we have $\rho_n^m\to\rho^m$ weakly in $L^{\frac{d}{d-2}}((0,T)\times\mathbb R^d)$ if $d\ge 3$ (weakly in $L^2((0,T)\times\mathbb R^d)$ if $d=1,2$) which allow to pass to the limit in the other two terms of \eqref{weakn}.
 \end{proofad5}

\section{Some qualitative properties of solutions}\label{qualitative}
In this section we present some numerical simulations about the evolution problem \eqref{cauchy1}, using the scheme developed in~\cite{MR3372289}.
The  time evolution is shown in Figure~\ref{fig:twobump} and ~\ref{fig:evob04} for different initial data in one dimension, providing a numerical illustration of the expected asymptotic behaviors, i.e.,
solutions approaching the unique stationary states. In general, if $s$ is not too close to zero, the stationary states are reached quickly, otherwise the
convergence may take longer with the appearance of ``disturbances'' near the boundary of the support as in Figure~\ref{fig:evob04} below.

\bigskip

 \begin{figure}[htp]
  \begin{center}
   \includegraphics[totalheight=0.28\textheight]{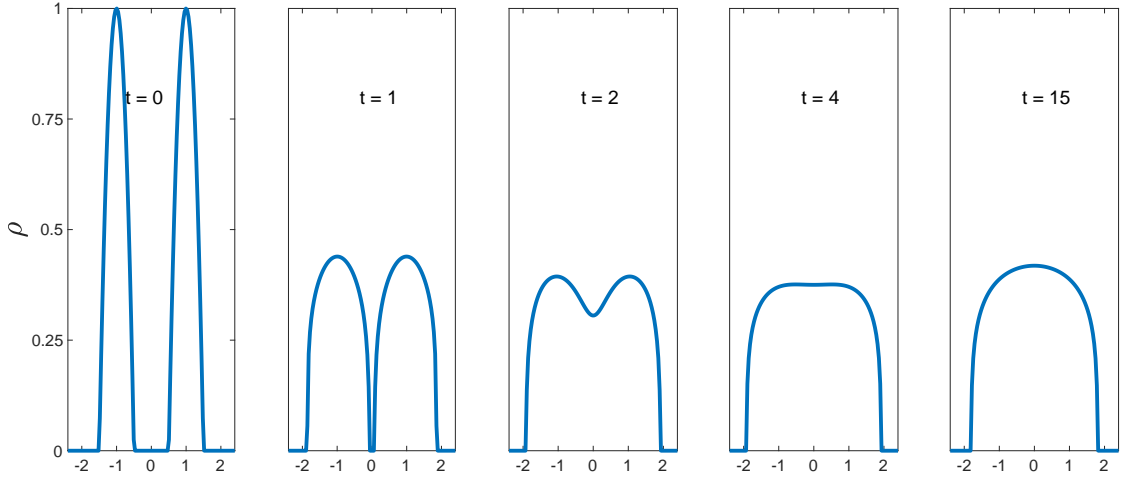}
  \end{center}
  \caption{The evolution of the solution starting with two bumps with parameters $m=3, \chi=1, s=0.1$ and $\beta=0.2$, reaching the stationary state reasonably
  fast.}
\label{fig:twobump}
\end{figure}

\bigskip\bigskip

 \begin{figure}[htp]
 \begin{center}
 \includegraphics[totalheight=0.28\textheight]{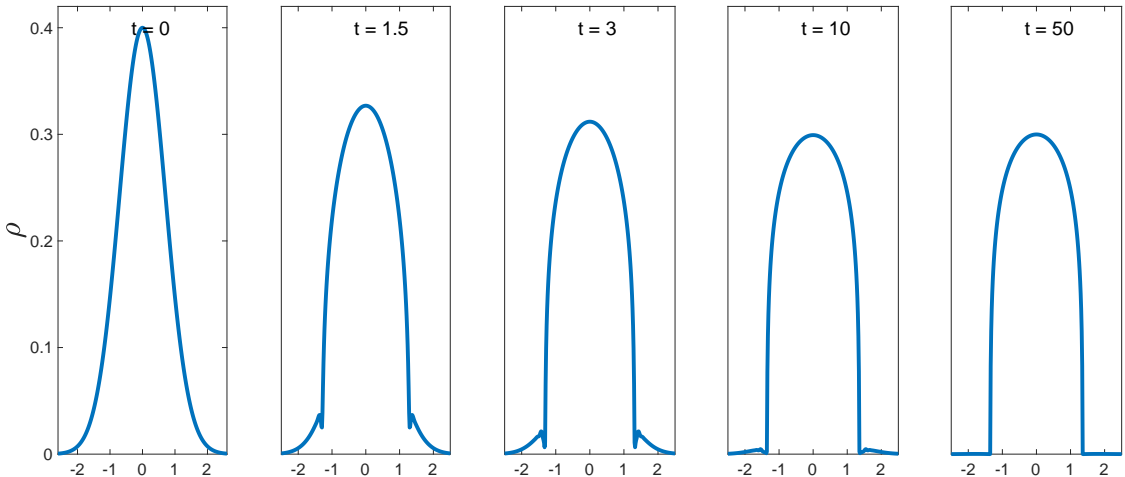}
 \end{center}
 \caption{The evolution  starting with a rescaled Gaussian, with $m=3$, $s=0.08$, $\chi=1$ and $\beta=0.4$, where the solution
 does converge to the expected stationary state. 
 }
 \label{fig:evob04}
 \end{figure}
\newpage

A big open problem concerning the Cauchy problem \eqref{cauchy1} is the \emph{uniqueness of the solution}. Assuming that this property holds true, by the rotationally invariant property of the main equation \eqref{cauchy1} it follows that for a given radial initial datum $\rho_{0}=\rho_{0}(|x|)$ we have that the density solution $\rho$ is radial w.r. t. $x$, \emph{i.e.} $\rho=\rho(|x|,t)$. But the property of being radially decreasing may not be preserved during the evolution, even with initial data (and limiting steady states) sharing this property.  The following counterexample is an adaptation of the one contained in \cite[Proposition 4.3]{KY}.
 Set
 \begin{equation}\label{sample0}
 \rho_{0,\varepsilon}(x)=\varphi_{\varepsilon}\ast(\delta_{0}+\varepsilon^{\alpha} \mathbbm{1}_{B(0,1)})=\varphi_{\varepsilon}(x)+\varepsilon^{\alpha}\varphi_{\varepsilon}\ast \mathbbm{1}_{B(0,1)},
 \end{equation}
 being $\varepsilon>0$, $\varphi$ a mollifier with mass 1 supported in the ball $B(0,1)$,  being
$
\varphi_{\varepsilon}(x)=\frac{1}{\varepsilon^{d}}\varphi\left(\frac{x}{\varepsilon}\right)
$
 and $\alpha>d+2$. Assume that there exists a radial  solution $\rho(x,t)$ to \eqref{cauchy1} with datum $\rho_{0,\varepsilon}$ and suppose we know that the solution $\rho(x,t)$ is smooth enough up to $t=0$. Then it is possible to show that $\rho$ is not radially decreasing. Indeed, it is immediate to see that for $\varepsilon$ small and $\varepsilon<|x|<1-\varepsilon$ we have
 \[
 \rho_{0,\varepsilon}(x)=\varepsilon^{\alpha},
 \]
 while $\rho_{0,\varepsilon}$ is supported in the ball $B(0,1+\varepsilon)$. Taking two points $x_{1},\,x_{2}$ such that $\varepsilon<|x_{1}|<|x_{2}|<1-\varepsilon$ and taking into account that $\rho_{0,\varepsilon}$ is constant in the interval $(\varepsilon, 1-\varepsilon)$, we have for $|x_{1}|\leq|x|\leq |x_{2}|$
 \begin{equation}
 \partial_{t}\rho(x,0)=-\rho_{0,\varepsilon}(x)\Delta(K_{s}\ast\rho_{0,\varepsilon}).\label{equatatt=0}
 \end{equation}
 Now, observe that since $\varphi_{\varepsilon}\rightarrow \delta_{0}$ in $\mathcal{D}^{\prime}(\R^{d})$ for $\varepsilon \rightarrow0$, we have
 \[
 K_{s}\ast\rho_{0,\varepsilon}\rightarrow K_{s}\quad \text{in }\mathcal{D}^{\prime}(\R^{d}).
 \]
as $\varepsilon \rightarrow0$.
But
\begin{equation}
\Delta(K_{s}\ast\rho_{0,\varepsilon})=K_{s}\ast\Delta \rho_{0,\varepsilon}=K_{s}\ast\Delta \varphi_{\varepsilon}+K_{s}\ast\varepsilon^{\alpha}(\Delta \varphi_{\varepsilon}\ast
\mathbbm{1}_{B(0,1)}).\label{Laplacconv}
\end{equation}
Since
\[
\Delta\varphi_{\varepsilon}\rightarrow \Delta\delta_{0}\quad\text{in }\mathcal{D}^{\prime}(\R^{d})
\]
and $\Delta\delta_{0}$ is supported at $0$, taking a cutoff function $\eta_{\varepsilon}$ such that $\eta_{\varepsilon}=1$ in a $B(0,\varepsilon)$
we find
\[
(K_{s}\ast\Delta \rho_{0,\varepsilon})(x)=\int_{\R^{d}} (\Delta\varphi_{\varepsilon}(y))\eta_{\varepsilon}(y)K_{s}(x-y)dy\rightarrow
\langle\Delta \delta,\eta_{\varepsilon}K_{s}(x-\cdot)\rangle=\Delta(\eta_{\varepsilon}K_{s}(x-\cdot))(0)
\]
as $\varepsilon \rightarrow0$ and an easy computation shows that
$
\Delta(\eta_{\varepsilon}K_{s}(x-\cdot))(0)=\Delta K_{s}(x).
$
Now, since
\[
|\Delta \varphi_{\varepsilon}(x)|=\varepsilon^{-d-2}|\Delta\varphi(x/\varepsilon)|\leq C \varepsilon^{-d-2},
\]
we have
\[
\|\Delta \varphi_{\varepsilon}\ast \mathbbm{1}_{B(0,1)}\|_{L^{\infty}}\leq\|\Delta \varphi_{\varepsilon}\|_{L^{\infty}} \|\mathbbm{1}_{B(0,1)}\|_{L^{1}}\leq C \varepsilon^{-d-2}
\]
and by Young inequality
\[
\|\Delta \varphi_{\varepsilon}\ast \mathbbm{1}_{B(0,1)}\|_{L^{1}}\leq \|\Delta \varphi_{\varepsilon}\|_{L^{1}} \|\mathbbm{1}_{B(0,1)}\|_{L^{1}}
=C\varepsilon^{-2}\|\Delta \varphi\|_{L^{1}}.
\]

\begin{figure}[htp]
  \begin{center}
   \includegraphics[totalheight=0.28\textheight]{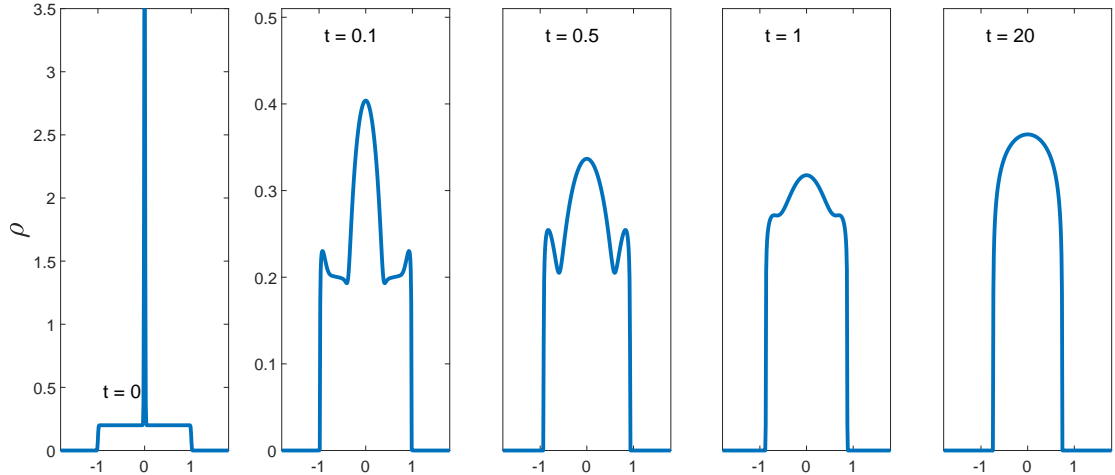}
  \end{center}
  \caption{Numerical demonstration of the fact that radially decreasing initial data does not necessarily remain radially decreasing. The initial
  condition is the one in \eqref{sample0}, with the parameters $m=3, \chi=1,s=0.1$ and $\beta=0$.}
\label{fmollifier}
 \end{figure}

Therefore,
\begin{align*}
|K_{s}\ast(\Delta \varphi_{\varepsilon}\ast\mathbbm{1}_{B(0,1)})| &\leq C(\|\Delta \varphi_{\varepsilon}\ast \mathbbm{1}_{B(0,1)}\|_{L^{\infty}}
+\|\Delta \varphi_{\varepsilon}\ast \mathbbm{1}_{B(0,1)}\|_{L^{1}})\nonumber\\
&\leq C\varepsilon^{-d-2}(1+\varepsilon^{N}).
\end{align*}
Thus choosing $\alpha>d+2$, from \eqref{Laplacconv} we have
\[
\Delta(K_{s}\ast\rho_{0,\varepsilon})(x)\rightarrow \Delta K_{s}(x)=c(d,s)|x|^{2s-d-2},
\]
as $\varepsilon\rightarrow 0$. This implies that for $\varepsilon$ small,
\[
\Delta(K_{s}\ast\rho_{0,\varepsilon})(x_{1})> \Delta(K_{s}\ast\rho_{0,\varepsilon})(x_{2}),
\]
hence \eqref{equatatt=0} gives (recalling that $\rho_{0,\varepsilon}$ is radially decreasing)
\[
\partial_{t}\rho(x_1,0)<\partial_{t}\rho(x_2,0),
\]
meaning that the radially decreasing monotonicity is \emph{not} preserved for small times.
The non-monotonicity of the solution is shown in the simulation from Figure~\ref{fmollifier}.

  Figure \ref{fig:nodec} shows a further simulation, which takes into account a characteristic function of a symmetric interval as initial datum:  also in this case the radial monotonicity is not preserved.

\begin{figure}[hhh]
  \begin{center}
   \includegraphics[totalheight=0.28\textheight]{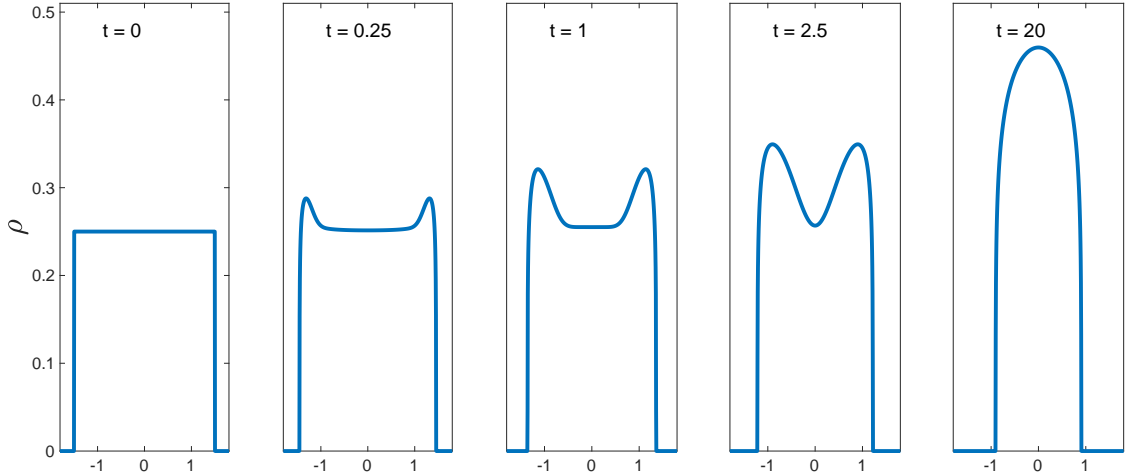}
  \end{center}
  \caption{Another example: again the radially decreasing initial datum does not remain radially decreasing. The initial
  condition is $\rho_0(X)=\frac{1}{4}\bC_{|x|<3/2}$, with the parameters $m=3, \chi=1,s=0.1$ and $\beta=0$.}
\label{fig:nodec}
 \end{figure}


\section{Comments, extensions and open problems}\label{open}
\noindent {$\bullet$} As mentioned in the previous section, an open problem concerns the \emph{uniqueness} of solutions, which would give radiality of solutions with radial initial data as a direct consequence.\\[0.2pt]

\noindent {$\bullet$} A second open problem is to rigorously prove that every solution to the evolution problem \eqref{cauchy1} does converge to the unique stationary state provided by Theorem \ref{furtherregularity*}. We mention that a similar result is available in the two dimensional setting, in the case of aggregation with the Newtonian potential instead of the Riesz potential, with $\beta=0$ and $m>1$ (i.e., diffusion-dominated regime), see \cite{CHVY}.
\\[0.2pt]

\noindent {$\bullet$} 
Concerning Theorem \ref{stozero}, uniqueness of the distributional solution for the Cauchy problem for \eqref{pme} with $\beta\ge \chi/2$ is known under additional conditions. For instance, according to the classical result by \cite{Pierre}, uniqueness holds among distributional solutions that are  essentially bounded on any strip $\R^{d}\times (\tau,T)$, for all $T>0$ and $\tau\in (0,T)$. Therefore, in order to obtain a unique limit as $s\to 0$ for families of gradient flow solutions $\rho_{s}$ to \eqref{cauchy1}, further a-priori $L^\infty$ bounds (uniformly in $s$) should be established for $\rho_s$.   
\\[0.2pt]

\noindent {$\bullet$} Another interesting open problem is to show that the family of solutions $\rho_{s}$ to problem \eqref{cauchy1} converges as $s\rightarrow 0$ to a solution (in an appropriate sense) to the equation \ref{pme} even in the case $\beta<\chi/2$.
We notice that  such equation has the form $\partial_{t}\rho=\Delta\varphi(\rho)$, and if
 $\beta<\chi/2$ the nonlinearity $\varphi$ is nonmonotone and equation \eqref{pme} is of  \emph{forward-backward} type, with the \emph{unstable} phase  given by the interval $[0,(\frac{\chi-2\beta}{m})^{1/(m-2)}]$ and the \emph{stable} phase by $[(\frac{\chi-2\beta}{m})^{1/(m-2)},+\infty)$. The nontrivial zero of $\varphi$ is
\[
\rho=\left(\frac{\chi-2\beta}{2}\right)^{1/(m-2)}
\]
which coincides exactly with the height of the minimizer of the free energy limit functional $\mathcal{F}_{0}$ given in \eqref{rad0}. If $\beta<\chi/2$, we would like to consider equation \eqref{pme} as a \emph{singular limit} as $s\rightarrow0$ of the main equation in \eqref{cauchy1}. An existence theory for equation \eqref{pme} supplemented with an initial condition $\rho(x,0)=\rho_{0}(x)$ could be given in the setting of Young measure solutions, see for instance \cite{Plot}, where such notion of solution is recovered  for cubic-like nonlinearities $\varphi$ as vanishing limit as $\varepsilon\rightarrow 0$ of a third order pseudo-parabolic regularization
$
\partial_{t}\rho=\Delta\varphi(\rho)+\varepsilon \Delta \rho_{t}.
$
 It would be  interesting to show that  even a weak limit $\rho$ as $s\to 0$ of a family of densities $\rho_{s}$ solving the equation \eqref{cauchy1} in the sense of Theorem \ref{existencethm} fits  the above mentioned existence theory.

\section*{Appendix}
We provide the proof of Lemma \ref{RieszW11}, Lemma \ref{lastbutone} and Lemma \ref{lastlemma}.

 \begin{proofad6}
First of all, $K_s\ast\rho\in L^\infty(\mathbb R^d)$ by \cite[Lemma 1]{CHMV}, and let us observe that, for each $i=1,...,d$, $K_{s}\ast\rho_{x_{i}}\in
L^{1}_{loc}(\R^{d})$. Indeed, $K_s$ can be written as the sum of two functions,  $K_s(x)=\bC_{B_1}(x)K_s(x)+(1-\bC_{B_1}(x))K_s(x)$,
supported on the unit ball $B_1(x)$ around $x$ and its complement. By Young convolution inequality we have $(\bC_{B_1}K_s)\ast\rho_{x_i}\in L^1(\R^d)$ and $((1-\bC_{B_1})K_s)\ast\rho_{x_i}\in L^\infty(\R^d)$.
 Moreover, for any smooth compactly supported test function $\varphi$ we have $(K_s\ast\rho_{x_i})\varphi\in L^1(\mathbb R^d)$ and we have, by the symmetry of $K_{s}$,
 \begin{equation}\label{re1}\begin{aligned}
 \int_{\R^{d}}(K_{s}\ast\rho_{x_{i}})\varphi\,dx&=c_{d,s}\int_{\R^{d}}\int_{\R^{d}}\frac{\rho_{x_{i}}(y)\varphi(x)}{|x-y|^{d-2s}}dx\,dy\\
 &=c_{d,s}\int_{\R^{d}}\int_{\R^{d}}\frac{\rho_{x_{i}}(x)\varphi(y)}{|x-y|^{d-2s}}dx\,dy=\int_{\R^{d}}\rho_{x_{i}}(K_{s}\ast\varphi)\,dx
 \end{aligned}
 \end{equation}
 and similarly
 \begin{equation}\label{re2}
 \int_{\R^d}(K_s\ast\varphi_{x_i})\rho\,dx=\int_{\R^d}(K_s\ast\rho)\varphi_{x_i}\,dx.
 \end{equation}
If we take a large ball $B_{R}$ centered at the origin, an integration by parts leads to
 \begin{equation}\label{br}
 \int_{B_{R}}(K_{s}\ast\varphi)\rho_{x_i}dx=-\int_{B_{R}}\rho(x)(K_{s}\ast\varphi_{x_{i}})dx+\int_{\partial B_{R}}\rho\,(K_{s}\ast\varphi)\nu_{i}\,d\sigma
 \end{equation}
and since $|(K_{s}\ast\varphi)(x)|\leq C/|x|^{d-2s}$ for large $x$ (as $\varphi$ is compactly supported),  thanks to the continuity and boundedness of $\rho$ we have if $s<1/2$
\[
\int_{\partial B_{R}}\rho\,(K_{s}\ast\varphi)\nu_{i}\,d\sigma\leq C\|\rho\|_{L^{\infty}(\R^{d})}R^{2s-1}\rightarrow0	\qquad\textit{ as }R\rightarrow+\infty,
\]
 along with
 \[
  \int_{B_{R}}(K_{s}\ast\varphi)\rho_{x_i}dx\to \int_{\R^d}(K_{s}\ast\varphi)\rho_{x_i}dx\quad\mbox{and}\quad\int_{B_{R}}\rho(K_{s}\ast\varphi_{x_{i}})dx\to\int_{\mathbb R^d}\rho(K_{s}\ast\varphi_{x_{i}})dx
 \]
 as $R\to+\infty$,
 which hold by dominated convergence due to the fact that $\rho(K_s\ast\varphi_{x_i})\in L^1(\R^d)$ and $(K_{s}\ast\varphi)\rho_{x_i}\in L^1(\mathbb R^d)$. Hence, we may pass to the limit in \eqref{br} and get
 \[
 \int_{\R^{d}}(K_{s}\ast\varphi)\rho_{x_i}dx=-\int_{\R^{d}}\rho(x)(K_{s}\ast\varphi_{x_{i}})dx,
 \]
 which can be combined with \eqref{re1} and \eqref{re2} to imply
 \[
  \int_{\R^{d}}(K_{s}\ast\rho_{x_i})\varphi\,dx= -\int_{\R^{d}}(K_{s}\ast\rho)\varphi_{x_i}dx.
 \]
 Therefore, we have that $K_{s}\ast\rho\in W^{1,1}_{loc}(\R^{d})$ and $\nabla(K_s\ast\rho)=K_s\ast\nabla\rho$ if $s<1/2$. On the other hand, if $d\ge 2$ and $s>1/2$ we even obtain $K_s\ast\rho\in W^{1,\infty}(\mathbb R^d)$, see \cite[Lemma 1]{CHMV}. Else if $d\ge 2$ and $s=1/2$ we obtain $K_s\ast\rho\in W^{1,p}$ for every $p\in(\tfrac{d}{d-1},+\infty)$: indeed, since $\rho\in L^1(\R^d)\cap L^\infty(\R^d)$, by the Hardy-Littlewood-Sobolev inequality we get $K_s\ast\rho \in L^p(\R^d)$ for every $p\in(\tfrac{d}{d-1},+\infty)$, thus $K_s\ast\rho$  belongs to the Bessel potential space $\mathcal L^{1,p}$ defined as $\mathcal L^{1,p}:=L^p(\R^d)\cap\{K_{1/2}\ast g:g\in L^p(\R^d)\}$, which coincides with $W^{1,p}(\R^d)$, see \cite[Theorem 3, pp 135]{Stein}.

The fact that $\rho^{m-1}\in W^{1,1}(\mathbb R^d)$ follows from the chain rule in Sobolev spaces, since $m> 2$ and $\rho\in L^\infty(\mathbb R^d)\cap
W^{1,1}(\mathbb R^d)$. Therefore $\psi\in W^{1,1}_{loc}(\mathbb R^d)$.

In order to conclude,  we write for every $\varphi\in C^\infty_c(\mathbb R^d)$
 \begin{align*}
 &\int_{\mathbb R^d}dx\int_{\mathbb R^d}(\nabla\varphi(x)-\nabla\varphi(y))\cdot(x-y)|x-y|^{2s-d-2}\rho(x)\,\rho(y)\,dy\\
 &\qquad=\lim_{\varepsilon\rightarrow0}\int_{\mathbb R^d}dx\int_{|x-y|>\varepsilon}(\nabla\varphi(x)-\nabla\varphi(y))\cdot(x-y)|x-y|^{2s-d-2}\rho(x)\,\rho(y)\,dy,
 \end{align*}
 then using the antisymmetry of the gradient of $K_{s}$ we have
 \[\begin{aligned}
& c_{d,s}\int_{\mathbb R^d}\int_{\mathbb R^d}(\nabla\varphi(x)-\nabla\varphi(y))\cdot(x-y)|x-y|^{2s-d-2}\rho(x)\,\rho(y)\,dx\,dy
\\&\qquad =\frac{2}{2s-d}\lim_{\varepsilon\rightarrow0}\int_{\mathbb R^d}\rho(x)\nabla\varphi \cdot \nabla(K_{s,\varepsilon}\ast\rho)dx
 \end{aligned}\]
 where $K_{s,\varepsilon}$ is the following truncation of $K_{s}$:
\begin{equation*}\label{gradS}
 K_{s,\varepsilon}(x) :=
 \begin{cases}
  K_{s}(x)\, ,
  &\text{if} \, \, |x|>\varepsilon\, , \\[2mm]
  c_{d,s}\varepsilon^{2s-d},
  &\text{if} \, \, |x|\leq\varepsilon\, .
 \end{cases}
\end{equation*}
At this point, we observe that
\begin{align*}
&\left|\int_{\mathbb R^d}\rho(x)\nabla\varphi \cdot \nabla(K_{s,\varepsilon}\ast\rho)dx-\int_{\mathbb R^d}\rho(x)\nabla\varphi \cdot \nabla(K_{s}\ast\rho)dx\right|\\
&\qquad\leq C\int_{\mathbb R^d}\left|\nabla\varphi(x)\cdot (K_{s,\varepsilon}-K_{s})\ast\nabla\rho\right|dx\leq C \int_{\mathbb R^d}\left|(K_{s,\varepsilon}-K_{s})\ast\nabla\rho\right|dx,
\end{align*}
thus by Young convolution inequality
\begin{align*}
&\left|\int_{\mathbb R^d}\rho(x)\nabla\varphi \cdot \nabla(K_{s,\varepsilon}\ast\rho)dx-\int_{\mathbb R^d}\rho(x)\nabla\varphi \cdot \nabla(K_{s}\ast\rho)dx\right|\\
&\qquad\leq C\|\nabla\rho\|_{L^{1}}\int_{\R^{d}}|K_{s,\varepsilon}-K_{s}|dx=C\|\nabla\rho\|_{L^{1}}\int_{|x|\leq\varepsilon}(|x|^{2s-d}-\varepsilon^{2s-d})
\rightarrow 0
\end{align*}
as $\varepsilon\rightarrow0$.
Therefore we can write
\begin{equation}\begin{aligned}
&c_{d,s}\int_{\mathbb R^d}\int_{\mathbb R^d}(\nabla\varphi(x)-\nabla\varphi(y))\cdot(x-y)|x-y|^{2s-d-2}\rho(x)\,\rho(y)\,dx\,dy\\&\qquad\qquad=\frac{2}{2s-d}
\int_{\mathbb R^d}\rho(x)\nabla\varphi \cdot \nabla(K_{s}\ast\rho)dx.\end{aligned}\label{identitydouble}
\end{equation}
Since \eqref{continuityeq} implies
 \[
\chi \int_{\mathbb R^d}\rho(x)\nabla\varphi \cdot \nabla(K_{s}\ast\rho)dx=\int_{\R^{d}}\rho\nabla\left(\frac{m}{m-1}\rho^{m-1}+2\beta\rho\right)\cdot\nabla\varphi\,dx
=\int_{\R^{d}}\nabla (\rho^{m}+\beta\rho^{2})\cdot\nabla\varphi\,dx,
 \]
 by \eqref{identitydouble} the identity \eqref{varequsteady} follows. Vice versa, if $\rho$ verifies \eqref{varequsteady}, the same computation gives that $\psi$ solves  \eqref{continuityeq}.
 \end{proofad6}

\begin{proofad7}
Through the proof, for every $f\in L^2(\mathbb R^d)$ and every $\varphi\in C^\infty_c(\R^d)$ we shall use the notation
\[
\mathcal I_s(f;\varphi):=(d-2s)\,c_{d,s}\int_{\R^d}\int_{\R^d}(\nabla\varphi(x)-\nabla\varphi(y))\,\cdot(x-y)|x-y|^{2s-d-2}f(x)f(y)\,dx\,dy
\]
The result is true if $\rho\in C^\infty_c(\mathbb R^d)$, since in this case we may apply \eqref{identitydouble}, and we may  integrate by parts and take advantage of the fact that $K_s\to \delta_0$ in the sense of distributions to get
\[
\begin{aligned}
\lim_{s\downarrow 0}\mathcal I_s(\rho;\varphi)
&=-2\,\lim_{s\downarrow 0}\int_{\R^d}\rho(x)\nabla\varphi(x)\cdot\nabla(K_s\ast\rho)(x)\,dx =-2\int_{\R^d}\rho(x)\nabla\varphi(x)\cdot\nabla\rho(x)\,dx\\
&\qquad=-\int_{\mathbb R^d}\nabla\varphi(x)\cdot\nabla(\rho^2(x))\,dx=\int_{\R^d}\Delta\varphi(x)\rho^2(x)\,dx.
\end{aligned}
\]
In order to obtain the result for $\rho\in L^1\cap  L^2(\R^d)$, let $(\rho_n)_{n\in\mathbb N}\subset C^\infty_c(\mathbb R^d)$ be a sequence that converges to $\rho$ in $L^2(\R^d)$ and in $L^1(\R^d)$ as $n\to+\infty$.
Thanks to \eqref{HLS} and by interpolation of $L^p$ norms,  we have
\begin{equation}\label{longestimate}
\begin{aligned}
&|\mathcal I_s(\rho;\varphi)-\mathcal I_s(\rho_n;\varphi)|\\&\quad\le(d-2s)\,c_{d,s}\int_{\R^d}\int_{\R^d}|\nabla\varphi(x)-\nabla\varphi(y)|\,|x-y|^{2s-d-1}|\rho(x)\rho(y)-\rho_n(x)\rho_n(y)|\,dx\,dy\\
&\quad\le (d-2s)\,\sup_{x\in\R^d} |\nabla^2\varphi(x)| \int_{\R^d}\int_{\R^d}\left|K_s(|x-y|)\,(\rho(x)\rho(y)-\rho_n(x)\rho_n(y))\right|\,dx\,dy\\
&\quad\le d\,\sup_{x\in\R^d} |\nabla^2\varphi(x)| \int_{\R^d}\int_{\R^d}K_s(|x-y|)\rho(x)|\rho(y)-\rho_n(y)|\,dx\,dy\\&\qquad\qquad+  d\,\sup_{x\in\R^d} |\nabla^2\varphi(x)|
\int_{\R^d}\int_{\R^d}K_s(|x-y|)|\rho(x)-\rho_n(x)|\rho_n(y)\,dx\,dy\\
&\quad\le 2d\,\sup_{x\in\R^d} |\nabla^2\varphi(x)| \,S_{d,s} \|\rho\|_{L^\frac{2d}{d+2s}(\mathbb R^d)}\|\rho-\rho_n\|_{L^\frac{2d}{d+2s}(\mathbb R^d)}\\
&\quad \le 2d\,\sup_{x\in\R^d} |\nabla^2\varphi(x)| \,S_{d,s}\, \|\rho\|_{L^1(\R^d)}^{\frac{4s}{d}}\,\|\rho-\rho_n\|_{L^1(\R^d)}^{\frac{4s}{d}}\,\|\rho\|_{L^2(\R^d)}^{\frac{d-2s}{d}}\,\|\rho-\rho_n\|_{L^2(\R^d)}^{\frac{d-2s}{d}},
\end{aligned}
\end{equation}
so that for every $n\in\mathbb N$ and every $\varphi\in C^{\infty}_c(\R^d)$
\begin{equation}\label{Ieps}
\limsup_{s\downarrow 0}|\mathcal I_s(\rho;\varphi)-\mathcal I_s(\rho_n;\varphi)|\le 2d \sup_{x\in\R^d} |\nabla^2\varphi(x)| \,\|\rho\|_{L^2(\R^d)}\,\|\rho-\rho_n\|_{L^2(\R^d)}.
\end{equation}
Moreover, for every $n\in\mathbb N$ and every $\varphi\in C^\infty_c(\R^d)$, since $\rho_n\in C^\infty_c(\mathbb R^d)$, we have
\[
\lim_{s\downarrow 0} \mathcal I_s(\rho_n;\varphi)=\int_{\R^d}\Delta\varphi(x)\rho_n^2(x)\,dx,
\]
which together with \eqref{Ieps} entails
\[
\begin{aligned}
&\limsup_{s\downarrow 0}\left|\mathcal I_s(\rho;\varphi)-\int_{\R^d}\rho^2\Delta\varphi\,dx\right|\\
&\qquad\le \limsup_{s\downarrow 0}\left(\left|\mathcal I_s(\rho;\varphi)-\mathcal I_s(\rho_n;\varphi)\right|+\left|\mathcal I_s(\rho_n;\varphi)-\int_{\mathbb R^d}\rho_n^2\Delta\varphi\right|+\int_{\R^d}|\rho^2-\rho_n^2|\Delta\varphi\right)\\
&\qquad\le 2d \sup_{x\in\R^d} |\nabla^2\varphi(x)| \,\left(\|\rho\|_{L^2(\R^d)}+\|\rho+\rho_n\|_{L^2(\R^d)}\right)\,\|\rho-\rho_n\|_{L^2(\R^d)}.
\end{aligned}
\]
 Since $\rho_n\to\rho$ in $L^2(\R^d)$, by taking the limit as $n\to+\infty$ we finally obtain  $\mathcal I_s(\rho;\varphi)\to\int_{\R^d}\rho^2\Delta\varphi$ as $s\downarrow 0$, for every $\varphi\in C^\infty_c(\mathbb R^d)$.
\end{proofad7}

\begin{proofad8} With the same notation of the previous proof,
we have for every $\varphi\in C^\infty_c(\R^d)$
\[\begin{aligned}
&\limsup_{s\downarrow 0} \left|\mathcal I_s(\rho_s;\varphi)-\int_{\R^d}\rho^2\Delta\varphi\right|\\&\qquad\le\limsup_{s\downarrow 0}\left|\mathcal I_s(\rho_s;\varphi)-\mathcal I_s(\rho;\varphi)\right|+\limsup_{s\downarrow 0}\left|\mathcal I_s(\rho;\varphi)-\int_{\R^d}\rho^2\Delta\varphi\right|,
\end{aligned}
\]
therefore in view of Lemma \ref{lastbutone},  it will be enough to prove that $|\mathcal I_s(\rho_s;\varphi)- \mathcal I_s(\rho;\varphi)|\to 0$ as $s\downarrow 0$. But the very same estimates of \eqref{longestimate} allow to obtain
\[
\left|\mathcal I_s(\rho_s;\varphi)-\mathcal I_s(\rho;\varphi)\right| \le 2d\sup_{x\in\mathbb R^d}|\nabla^2\varphi(x)|\,S_{d,s}\, \|\rho\|_{L^1(\R^d)}^{\frac{4s}{d}}\,\|\rho-\rho_s\|_{L^1(\R^d)}^{\frac{4s}{d}}\,\|\rho\|_{L^2(\R^d)}^{\frac{d-2s}{d}}\,\|\rho-\rho_s\|_{L^2(\R^d)}^{\frac{d-2s}{d}}
\]
where the right hand side vanishes as $s\downarrow 0$ thanks to the assumptions on the family $(\rho_s)$.
\end{proofad8}

\subsection*{Acknowledgements} 
E.M. acknowledge support from the MIUR-PRIN  project  No 2017TEXA3H.
E.M. and B.V.  are members of the
GNAMPA group of the Istituto Nazionale di Alta Matematica (INdAM). The work of J. L. V\'azquez was funded by grant PGC2018-098440-B-I00 from the Spanish Government.  He is an Honorary Professor at Univ. Complutense de Madrid.

\bibliographystyle{plain}

\Addresses

\end{document}